\documentclass[11pt]{article}

\usepackage[hidelinks]{hyperref}
\hypersetup{
  colorlinks   = true, 
  urlcolor     = blue, 
  linkcolor    = blue, 
  citecolor   = red 
}
\usepackage{amsmath,amsthm,amssymb}
\usepackage{graphicx}
\usepackage{bbm}
\newcommand{\sy}{\boldsymbol{\Psi}}
\newcommand{\py}{\boldsymbol{\Phi}}

\usepackage{xcolor}
\usepackage[margin=1in]{geometry} 
\usepackage{enumerate,mathtools,mathrsfs,bm,graphicx}
\usepackage[numbers, super]{natbib}
\usepackage[hidelinks]{hyperref}
\newcommand{\N}{\mathbb{N}}									

\newcommand{\R}{\mathbb{R}}

\newcommand{\vertiii}[1]{{\left\vert\kern-0.25ex\left\vert\kern-0.25ex\left\vert #1 
    \right\vert\kern-0.25ex\right\vert\kern-0.25ex\right\vert}}
    
\newcommand{\inner}[2]{\left\langle #1, #2 \right\rangle}
\DeclarePairedDelimiter\abs{\lvert}{\rvert}					

 \newcommand{\norm}[1]{\left\Vert #1 \right\Vert}
\newtheorem{theorem}{Theorem}[section]
\newtheorem{corollary}{Corollary}[theorem]

\newtheorem{lemma}[theorem]{Lemma}
\newtheorem{proposition}[theorem]{Proposition}
\newtheorem{remark}[theorem]{Remark}
\newtheorem{definition}[theorem]{Definition}

\newtheorem{assumption}[theorem]{Assumption}

\begin{document}
	\title{Weak and Strong Solutions to Nonlinear SPDEs with Unbounded Noise}
	\author{Daniel Goodair \footnote{Imperial College London, daniel.goodair16@imperial.ac.uk}}
	\date{\today} 
	\maketitle
\setcitestyle{numbers}	
	\begin{abstract}
We introduce an extended variational framework for nonlinear SPDEs with unbounded noise, defining three different solution types of increasing strength along with criteria to establish their existence. The three notions can be understood as probabilistically and analytically weak, probabilistically strong and analytically weak, as well as probabilistically and analytically strong. Our framework facilitates several well-posedness results for the Navier-Stokes Equation with transport noise, equipped with the no-slip and Navier boundary conditions.
     
	\end{abstract}

\tableofcontents
\thispagestyle{empty}
\newpage

\setcounter{page}{1}
\addcontentsline{toc}{section}{Introduction}

\section*{Introduction}

The theoretical analysis of stochastic partial differential equations, abbreviated to SPDEs, has been recently established as one of mathematics' most exciting prospects. Attention has long been given to these equations due to their rich applications in areas such as turbulence, population dynamics, finance and neuroscience, of which more can be understood from the monographs [\cite{liu2015stochastic}, \cite{lototsky2017stochastic}, \cite{pardoux2021stochastic}] and references therein. The area's newfound spike in popularity can be attributed to developments across both analytical techniques and modelling. Indeed, Martin Hairer's Fields Medal winning work on \textit{regularity structures}, [\cite{hairer2014theory}], gave rigorous meaning to a drastically extended class of SPDEs which are ill-defined in standard function spaces due to spatial irregularities. At a similar time, Gubinelli, Imkeller and Perkowski developed a theory of \textit{paracontrolled distributions}, [\cite{gubinelli2015paracontrolled}],  which again provided a toolkit for the seemingly intractable theory of distribution-valued SPDEs. These techniques are of a very different flavour to the more classical \textit{variational approach} pioneered by \'{E}tienne Pardoux, [\cite{pardoux1975equations}], whereby the equation is formatted in a Gelfand Triple for a noise valued in the Hilbert Space.\\

From the modelling perspective, it is in applications to fluid dynamics and turbulence where the literature has been thriving. The choice of noise to best encapsulate physical properties of fluid equations is in constant study, now yielding a strong argument for transport type stochastic perturbations (where the stochastic integral depends on the gradient of the solution). The paper of Brze\'{z}niak, Capinski and Flandoli [\cite{brzezniak1992stochastic}] in 1992 was one of the first to bring attention to the significance of these equations with transport noise, whilst such ideas have recently been cemented through the specific stochastic transport schemes of [\cite{holm2015variational}] and [\cite{memin2014fluid}]. In these papers Holm and M\'{e}min establish a new class of stochastic equations driven by transport type noise which serve as fluid dynamics models by adding uncertainty in the transport of fluid parcels to reflect the unresolved scales. The physical significance of such equations in modelling, numerical analysis and data assimilation continues to be well documented, see [\cite{alonso2020modelling}, \cite{chapron2024stochastic}, \cite{cotter2020data}, \cite{cotter2018modelling}, \cite{cotter2019numerically},  \cite{crisan2023implementation}, \cite{crisan2021theoretical}, \cite{dufee2022stochastic}, \cite{ephrati2023data}, \cite{flandoli20212d}, \cite{holm2019stochastic}, \cite{holm2020stochastic},  \cite{lang2022pathwise}, \cite{van2021bayesian}, \cite{street2021semi}] and in particular [\cite{flandoli2023stochastic}] for a comprehensive account of the topic. Towards their analysis, however, these equations find themselves in a structural limbo. The first-order noise operator introduces a singularity into the Hilbert Space of the variational framework, so such equations do not fit into this classical theory. The noise is, however, still well-defined in traditional function spaces, suggesting that the heavy machinery and distribution-tailored approaches of Hairer, Gubinelli et al. are not particularly appropriate. Our approach is to extend techniques of the variational framework in order to treat nonlinear SPDEs with potentially unbounded noise in their optimal spaces. More precisely, our object of study is the equation 
\begin{equation} \label{thespde}
    \sy_t = \sy_0 + \int_0^t \mathcal{A}(s,\sy_s)ds + \int_0^t\mathcal{G} (s,\sy_s) d\mathcal{W}_s
\end{equation}
framed for a triplet of embedded Hilbert Spaces $$V \xhookrightarrow{}H \xhookrightarrow{} U,$$ where $\mathcal{W}$ is a Cylindrical Brownian Motion over some auxiliary Hilbert Space $\mathfrak{U}$, and $\mathcal{G}$ (over $\mathfrak{U}$) is unbounded on $H$ but maps $H$ into $U$. The complete setup is given in Subsections \ref{stochastic framework} and \ref{subs functional framework}. We consider three different notions of solution, increasing in strength and hence in the assumptions imposed. In brief, these are:
\begin{itemize}
    \item \textbf{Martingale Weak Solutions}: these solutions are \textit{probabilistically weak}, as well as \textit{analytically weak} in the sense that the identity (\ref{thespde}) is not satisfied in $U$ but rather in the dual space $H^*$ of $H$, with $U$ embedded into $H^*$. Solutions belong, pathwise, to the space $L^\infty\left([0,T];U\right) \cap L^2\left([0,T];H\right)$. Martingale weak solutions are defined in Definition \ref{definitionofspacetimeweakmartingale}. Their existence is the content of Theorem \ref{theorem for martingale weak existence}.
    \item \textbf{Weak Solutions}: these solutions are \textit{probabilistically strong},  whilst \textit{analytically weak} in the above sense. Solutions again belong, pathwise, to the space $L^\infty\left([0,T];U\right) \cap L^2\left([0,T];H\right)$. Weak solutions are defined in Definition \ref{definitionofweak}. Their existence and uniqueness is the content of Theorem \ref{theorem for weak existence}.
     \item \textbf{Strong Solutions}: these solutions are \textit{probabilistically strong}, as well as \textit{analytically strong} in the sense that the identity (\ref{thespde}) is satisfied in $U$. Solutions belong, pathwise, to the space $L^\infty\left([0,T];H\right) \cap L^2\left([0,T];V\right)$. Strong solutions are defined in Definition \ref{definitionofstrong}. Their existence and uniqueness is the content of Theorem \ref{theorem for strong existence}.
\end{itemize}

The variational approach to nonlinear SPDEs with additive and multiplicative noise has long been studied, initially in the works [\cite{gyongy1980stochastic}, \cite{krylov2007stochastic}, \cite{pardoux1975equations}] and more recently [\cite{debussche2011local}, \cite{liu2013well}, \cite{liu2010spde}, \cite{liu2013local}, \cite{neelima2020coercivity}] to name just a few. For an unbounded noise operator, necessary in applications to transport noise models, much less has been developed. Let us first mention the paper of R\"{o}ckner, Shang and Zhang, [\cite{rockner2022well}], where the norm of $\mathcal{G}(s,f)$ in $U$ can depend on the $H$ norm of $f$ but only up to a small constant relative to the coercivity of $\mathcal{A}(s,\cdot)$. Such an assumption is relieved in our case, in practical applications by obtaining precise cancellation through the It\^{o}-Stratonovich corrector of a Stratonovich transport noise. In this case, the It\^{o}-Stratonovich corrector appears in the drift $\mathcal{A}$. Tang and Wang, [\cite{tang2022general}], consider solutions for a doublet of embedded Hilbert Spaces, $H \xhookrightarrow{} U$, with pathwise continuity in $H$ and satisfying the evolution equation in $U$. Contrasting the more traditional triplet of spaces, their result is built for \textit{inviscid} fluid equations in which the additional $L^2\left([0,T],V\right)$ regularity is not expected. Whilst \textit{viscous} fluid equations may well be solvable in this framework, above optimal regularity on the initial condition would be necessary and the additional expected $L^2\left([0,T],V\right)$ property would not be immediately recovered. Therefore, our framework serves a distinct purpose and we see these results as complementary. The final result of which we are aware is the author's previous paper with Crisan and Lang, [\cite{goodair2023existence}], in which a \textit{local} existence result is obtained with a blow-up criterion in the energy norm of existence. Forthwith, this suggests a different role for the current paper in applications. We illustrate how one can use these criteria to obtain solutions to stochastic Navier-Stokes equations, in $D$ dimensions with initial condition $u_0$; a more complete description comes in Section \ref{Section: applications}:
\begin{itemize}
    \item 2D, $u_0 \in L^2$: existence and uniqueness of probabilistically strong, analytically weak solutions is expected; one can apply Theorem \ref{theorem for weak existence}, see Theorems \ref{2D weak no slip}, \ref{2D weak navier}.
    \item 2D, $u_0 \in W^{1,2}$: existence and uniqueness of probabilistically strong, analytically strong solutions is expected; one can apply Theorem \ref{theorem for strong existence}, see Theorem \ref{2d strong  navier}.
    \item 3D, $u_0 \in L^2$: existence of probabilistically weak, analytically weak solutions is expected; one can apply Theorem \ref{theorem for martingale weak existence}, see Theorem \ref{3D weak no slip}. We appreciate, however, that the emergence of convex integration techniques in SPDEs suggests that probabilistically strong, analytically weak solutions could be obtainable, even in the absence of uniqueness; see [\cite{hofmanova2023global}], for example. 
     \item 3D, $u_0 \in W^{1,2}$: existence of probabilistically strong, analytically strong \textit{local} solutions is expected; one can apply the results of [\cite{goodair2023existence}], for which complete details of the application are given in [\cite{goodair20233d}].
\end{itemize}

We now explain the plan of the paper along with the overarching methods and contributions at each stage:

\begin{itemize}
    \item Section \ref{section preliminaries} is devoted to the setup of the problem in terms of notation, along with the basic stochastic and functional framework.
    \item Section \ref{section weak} treats martingale weak solutions. The necessary assumptions on the functions $\mathcal{A}$, $\mathcal{G}$ are stated as Assumption Set 1 in Subsection \ref{subby assumption}. The definition of a solution and statement of the existence is given in Subsection \ref{subby marty weak def and res}, with the remaining subsections dedicated to its proof. Our approach is to consider a finite-dimensional approximating sequence of solutions, known as a \textit{Galerkin System}, which is shown to exhibit a limiting solution through tightness and relative compactness arguments. The Galerkin System is constructed in Subsection \ref{subbiegalerkin}, where uniform moment estimates in the energy norm of solutions are also obtained. The tightness is verified in Subsection \ref{subbietight}, whilst passage to the limit occurs in Subsection \ref{subbie passage to limit marty weak}.
    \item Section \ref{seection weak} concerns weak solutions. Assumption Set 1 is expanded upon in order to obtain uniqueness, giving rise to Assumption Set 2 of Subsection \ref{assumption set 2}. The definition of a solution and statement of the existence and uniqueness is given in Subsection \ref{subbie weak def and res}. The proof is given in Subsection \ref{exis uniqu}, where uniqueness facilitates passage from martingale weak solutions to weak solutions through a classical Yamada-Watanabe result. We draw attention to the fact that martingale weak solutions are only shown to exist for an initial condition belonging to $L^\infty\left(\Omega;U\right)$; weak solutions are initially shown to exist only for such an initial condition as well, though this is then relieved to the unbounded case. The method involves chopping up the unbounded initial condition onto intervals on which it is bounded, obtaining the relevant solutions, and then piecing these solutions together to obtain one for the original unbounded initial condition. We are unable to execute this approach for probabilistically weak solutions, as for each interval on which the initial condition is bounded one may obtain a solution on a distinct probability space.
    
    \item Section \ref{section strong} addresses strong solutions. Our assumptions are again expanded with Assumption Set 3 in Subsection \ref{assumption set 3}. The result boils down to showing that the weak solution has the additional regularity of a strong solution. The first step towards this end is taken in Subsection \ref{subbie improved uniform estimate} where uniform estimates at the level of the Galerkin System are shown, for the energy norm of strong solutions, up until first hitting times of the processes in the energy norm of weak solutions. Relative compactness arguments ensure that our weak solution gains this regularity up until any stopping time which is less than at least a subsequence of the aforementioned first hitting times. Two questions now present themselves in looking to quantify these times:
    \begin{enumerate}
        \item Can such a lesser stopping time be chosen to be $\mathbbm{P}-a.s.$ positive?
        \item If we increase the threshold for the first hitting times, can this lesser stopping time be taken to infinity?
    \end{enumerate}
    The first question was answered by Glatt-Holtz and Ziane in [\cite{glatt2009strong}], Lemma 5.1. The abstract lemma asserts that, under assumptions of a Cauchy property of the Galerkin System up until their first hitting times and some weak equicontinuity at the \textit{initial} time, then a limiting process and $\mathbbm{P}-a.s.$ positive stopping time smaller than a subsequence of the first hitting times exist (which are then argued to be a local strong solution, as a limit of the Galerkin Approximation). No characterisation of this limiting stopping time is given though, hence completely separate arguments are required to pass to a \textit{maximal} solution and further still to show that the maximal solution is global. Complete, and heavy, arguments to give rise to a maximal strong solution, and further to characterise the maximal time by the blow-up in the energy norm of strong solutions, are given in [\cite{goodair2023existence}] Subsection 3.6. Further still, global solutions would then be deduced by directly controlling the energy of the solution, which requires analysis in a space in which the evolution equation may not hold; this is certainly true of the example which we consider in Subsection \ref{applied no slip}.\\
    
    We clarify the second question with an extension of [\cite{glatt2009strong}] Lemma 5.1, Proposition \ref{amazing cauchy lemma}. This proposition greatly simplifies the quoted procedure for maximality and blow-up, and enables global solutions to be deduced through an estimate of the energy norm of \textit{weak} solutions, requiring analysis in a space in which the evolution equation certainly does hold. The extension is such that if instead one imposes that the processes satisfy a weak equicontinuity assumption at \textit{all} times then the limiting stopping time can be taken as a first hitting time of the limiting process for an arbitrarily large hitting parameter. Application of this result \textit{immediately} yields that solutions exist up until blow-up, pertinently as blow-up \textit{in the energy norm of weak solutions} in our application; we emphasise again the advantage that this has, as the previous approaches discussed would only produce a characterisation of the maximal time in the energy norm of the strong solution. Our result demands a higher level estimate only for the Galerkin System, which is clearer as the evolution equation is satisfied in all relevant spaces. We state the result abstractly and in the appendix given its significance in the broader theory of SPDEs. Towards the application of Proposition \ref{amazing cauchy lemma}, the Cauchy property is validated in Subsection \ref{subbie cauchy} with the equicontinuity following in Subsection \ref{subbie equicontinuity}. This is all tied together in Subsection \ref{Subbie passage to limit in strong}, before a discussion around the potential continuity of the process in Subsection \ref{subbie continuity?}.

    \item Section \ref{Section: applications} sketches the applications of the main results of the paper to the 2D and 3D Navier-Stokes Equations with Stochastic Lie Transport, under both the no-slip and Navier boundary conditions. The equation is properly introduced there, with further details in [\cite{goodair20233d}].

    \item Section \ref{section appendix} concludes the paper, firstly in Subsection \ref{subby to prove cauchy} with the statement and proof of the aforementioned Cauchy result (Proposition \ref{amazing cauchy lemma}), followed by Subesection 
    \ref{subby statements from lit} with the statements of useful results from the literature. These include the existence and uniqueness in finite-dimensions (Proposition \ref{Skorotheorem}), the energy identity (Proposition \ref{rockner prop}), a Stochastic Gr\"{o}nwall Lemma (Lemma \ref{gronny}) and tightness criteria (Lemmas \ref{Lemma 5.2}, \ref{lemma for D tight}).
    
\end{itemize}

\section{Preliminaries} \label{section preliminaries}

\subsection{Elementary Notation} \label{sub elementary notation}

In the following $\mathscr{O} \subset \R^N$ will be a smooth bounded domain equipped with Euclidean norm and Lebesgue measure $\lambda$. We consider Banach Spaces as measure spaces equipped with their corresponding Borel $\sigma$-algebra. Let $(\mathcal{X},\mu)$ denote a general topological measure space, $(\mathcal{Y},\norm{\cdot}_{\mathcal{Y}})$ and $(\mathcal{Z},\norm{\cdot}_{\mathcal{Z}})$ be separable Banach Spaces, and $(\mathcal{U},\inner{\cdot}{\cdot}_{\mathcal{U}})$, $(\mathcal{H},\inner{\cdot}{\cdot}_{\mathcal{H}})$ be general separable Hilbert spaces. We introduce the following spaces of functions. 
\begin{itemize}
    \item $L^p(\mathcal{X};\mathcal{Y})$ is the  class of measurable $p-$integrable functions from $\mathcal{X}$ into $\mathcal{Y}$, $1 \leq p < \infty$, which is a Banach space with norm $$\norm{\phi}_{L^p(\mathcal{X};\mathcal{Y})}^p := \int_{\mathcal{X}}\norm{\phi(x)}^p_{\mathcal{Y}}\mu(dx).$$ In particular $L^2(\mathcal{X}; \mathcal{Y})$ is a Hilbert Space when $\mathcal{Y}$ itself is Hilbert, with the standard inner product $$\inner{\phi}{\psi}_{L^2(\mathcal{X}; \mathcal{Y})} = \int_{\mathcal{X}}\inner{\phi(x)}{\psi(x)}_\mathcal{Y} \mu(dx).$$ In the case $\mathcal{X} = \mathscr{O}$ and $\mathcal{Y} = \R^N$ note that $$\norm{\phi}_{L^2(\mathscr{O};\R^N)}^2 = \sum_{l=1}^N\norm{\phi^l}^2_{L^2(\mathscr{O};\R)}, \qquad \phi = \left(\phi^1, \dots, \phi^2\right), \quad \phi^l: \mathscr{O} \rightarrow \R.$$ We denote $\norm{\cdot}_{L^p(\mathscr{O};\R^N)}$ by $\norm{\cdot}_{L^p}$ and $\norm{\cdot}_{L^2(\mathscr{O};\R^N)}$ by $\norm{\cdot}$.
    
\item $L^{\infty}(\mathcal{X};\mathcal{Y})$ is the class of measurable functions from $\mathcal{X}$ into $\mathcal{Y}$ which are essentially bounded. $L^{\infty}(\mathcal{X};\mathcal{Y})$ is a Banach Space when equipped with the norm $$ \norm{\phi}_{L^{\infty}(\mathcal{X};\mathcal{Y})} := \inf\{C \geq 0: \norm{\phi(x)}_Y \leq C \textnormal{ for $\mu$-$a.e.$ }  x \in \mathcal{X}\}.$$

      \item $C(\mathcal{X};\mathcal{Y})$ is the space of continuous functions from $\mathcal{X}$ into $\mathcal{Y}$.

      \item $C_w(\mathcal{X};\mathcal{Y})$ is the space of `weakly continuous' functions from $\mathcal{X}$ into $\mathcal{Y}$, by which we mean continuous with respect to the given topology on $\mathcal{X}$ and the weak topology on $\mathcal{Y}$. 
      
    \item $C^m(\mathscr{O};\R)$ is the space of $m \in \N$ times continuously differentiable functions from $\mathscr{O}$ to $\R$, that is $\phi \in C^m(\mathscr{O};\R)$ if and only if for every $2$ dimensional multi index $\alpha = \alpha_1, \alpha_2$ with $\abs{\alpha}\leq m$, $D^\alpha \phi \in C(\mathscr{O};\R)$ where $D^\alpha$ is the corresponding classical derivative operator $\partial_{x_1}^{\alpha_1} \partial_{x_2}^{\alpha_2}$.
    
    \item $C^\infty(\mathscr{O};\R)$ is the intersection over all $m \in \N$ of the spaces $C^m(\mathscr{O};\R)$.
    
    \item $C^m_0(\mathscr{O};\R)$ for $m \in \N$ or $m = \infty$ is the subspace of $C^m(\mathscr{O};\R)$ of functions which have compact support.
    
    \item $C^m(\mathscr{O};\R^N), C^m_0(\mathscr{O};\R^N)$ for $m \in \N$ or $m = \infty$ is the space of functions from $\mathscr{O}$ to $\R^N$ whose component mappings each belong to $C^m(\mathscr{O};\R), C^m_0(\mathscr{O};\R)$.
    
        \item $W^{m,p}(\mathscr{O}; \R)$ for $1 \leq p < \infty$ is the sub-class of $L^p(\mathscr{O}, \R)$ which has all weak derivatives up to order $m \in \N$ also of class $L^p(\mathscr{O}, \R)$. This is a Banach space with norm $$\norm{\phi}^p_{W^{m,p}(\mathscr{O}, \R)} := \sum_{\abs{\alpha} \leq m}\norm{D^\alpha \phi}_{L^p(\mathscr{O}; \R)}^p,$$ where $D^\alpha$ is the corresponding weak derivative operator. In the case $p=2$ the space $W^{m,2}(\mathscr{O}, \R)$ is Hilbert with inner product $$\inner{\phi}{\psi}_{W^{m,2}(\mathscr{O}; \R)} := \sum_{\abs{\alpha} \leq m} \inner{D^\alpha \phi}{D^\alpha \psi}_{L^2(\mathscr{O}; \R)}.$$
    
    \item $W^{m,\infty}(\mathscr{O};\R)$ for $m \in \N$ is the sub-class of $L^\infty(\mathscr{O}, \R)$ which has all weak derivatives up to order $m \in \N$ also of class $L^\infty(\mathscr{O}, \R)$. This is a Banach space with norm $$\norm{\phi}_{W^{m,\infty}(\mathscr{O}, \R)} := \sup_{\abs{\alpha} \leq m}\norm{D^{\alpha}\phi}_{L^{\infty}(\mathscr{O};\R^N)}.$$
    
          \item $W^{m,\infty}(\mathscr{O}; \R^N)$ is the sub-class of $L^\infty(\mathscr{O}, \R^N)$ which has all weak derivatives up to order $m \in \N$ also of class $L^\infty(\mathscr{O}, \R^N)$. This is a Banach space with norm $$\norm{\phi}_{W^{m,\infty}} := \sup_{l \leq N}\norm{\phi^l}_{W^{m,\infty}(\mathscr{O}; \R)}.$$

    \item $W^{m,p}_0(\mathscr{O};\R), W^{m,p}_0(\mathscr{O};\R^N)$ for $m \in \N$ and $1 \leq p \leq \infty$ is the closure of $C^\infty_0(\mathscr{O};\R),C^\infty_0(\mathscr{O};\R^N)$ in $W^{m,p}(\mathscr{O};\R), W^{m,p}(\mathscr{O};\R^N)$.

    \item $\mathscr{L}(\mathcal{Y};\mathcal{Z})$ is the space of bounded linear operators from $\mathcal{Y}$ to $\mathcal{Z}$. This is a Banach Space when equipped with the norm $$\norm{F}_{\mathscr{L}(\mathcal{Y};\mathcal{Z})} = \sup_{\norm{y}_{\mathcal{Y}}=1}\norm{Fy}_{\mathcal{Z}}$$ and is simply the dual space $\mathcal{Y}^*$ when $\mathcal{Z}=\R$, with operator norm $\norm{\cdot}_{\mathcal{Y}^*}.$
    
     \item $\mathscr{L}^2(\mathcal{U};\mathcal{H})$ is the space of Hilbert-Schmidt operators from $\mathcal{U}$ to $\mathcal{H}$, defined as the elements $F \in \mathscr{L}(\mathcal{U};\mathcal{H})$ such that for some basis $(e_i)$ of $\mathcal{U}$, $$\sum_{i=1}^\infty \norm{Fe_i}_{\mathcal{H}}^2 < \infty.$$ This is a Hilbert space with inner product $$\inner{F}{G}_{\mathscr{L}^2(\mathcal{U};\mathcal{H})} = \sum_{i=1}^\infty \inner{Fe_i}{Ge_i}_{\mathcal{H}}$$ which is independent of the choice of basis (see e.g. [\cite{conway2000course}]).

     \item For any $T>0$, $\mathscr{S}_T$ is the subspace of $C\left([0,T];[0,T]\right)$ of strictly increasing functions.

     \item For any $T>0$, $\mathcal{D}\left([0,T];\mathcal{Y}\right)$ is the space of c\'{a}dl\'{a}g functions from $[0,T]$ into $\mathcal{Y}$. It is a complete separable metric space when equipped with the metric $$d(\phi,\psi) := \inf_{\eta \in \mathscr{S}_T}\left[\sup_{t \in [0,T]}\left\vert \eta(t)- t\right\vert \vee \sup_{t \in [0,T]}\left\Vert \phi(t)-\psi(\eta(t)) \right\Vert_{\mathcal{Y}} \right]$$ which induces the so called Skorohod Topology (see [\cite{billingsley2013convergence}] pp124 for details).

\end{itemize}

\subsection{Stochastic Framework}\label{stochastic framework}

Let $(\Omega,\mathcal{F},(\mathcal{F}_t), \mathbb{P})$ be a fixed filtered probability space satisfying the usual conditions of completeness and right continuity. We take $\mathcal{W}$ to be a cylindrical Brownian motion over some Hilbert Space $\mathfrak{U}$ with orthonormal basis $(e_i)$. Recall (e.g. [\cite{lototsky2017stochastic}], Definition 3.2.36) that $\mathcal{W}$ admits the representation $\mathcal{W}_t = \sum_{i=1}^\infty e_iW^i_t$ as a limit in $L^2(\Omega;\mathfrak{U}')$ whereby the $(W^i)$ are a collection of i.i.d. standard real valued Brownian Motions and $\mathfrak{U}'$ is an enlargement of the Hilbert Space $\mathfrak{U}$ such that the embedding $J: \mathfrak{U} \rightarrow \mathfrak{U}'$ is Hilbert-Schmidt and $\mathcal{W}$ is a $JJ^*-$Cylindrical Brownian Motion over $\mathfrak{U}'$. Given a process $F:[0,T] \times \Omega \rightarrow \mathscr{L}^2(\mathfrak{U};\mathscr{H})$ progressively measurable and such that $F \in L^2\left(\Omega \times [0,T];\mathscr{L}^2(\mathfrak{U};\mathscr{H})\right)$, for any $0 \leq t \leq T$ we define the stochastic integral $$\int_0^tF_sd\mathcal{W}_s:=\sum_{i=1}^\infty \int_0^tF_s(e_i)dW^i_s,$$ where the infinite sum is taken in $L^2(\Omega;\mathscr{H})$. We can extend this notion to processes $F$ which are such that $F(\omega) \in L^2\left( [0,T];\mathscr{L}^2(\mathfrak{U};\mathscr{H})\right)$ for $\mathbb{P}-a.e.$ $\omega$ via the traditional localisation procedure. In this case the stochastic integral is a local martingale in $\mathscr{H}$. \footnote{A complete, direct construction of this integral, a treatment of its properties and the fundamentals of stochastic calculus in infinite dimensions can be found in [\cite{goodair2022stochastic}] Section 1.} We shall make frequent use of the Burkholder-Davis-Gundy Inequality ([\cite{da2014stochastic}] Theorem 4.36), passage of a bounded linear operator through the stochastic integral ([\cite{prevot2007concise}] Lemma 2.4.1) and the It\^{o} Formula (Proposition \ref{rockner prop}). 

\subsection{Functional Framework}\label{subs functional framework}

Recall that our object of study is the It\^{o} SPDE (\ref{thespde}), $$\sy_t = \sy_0 + \int_0^t \mathcal{A}(s,\sy_s)ds + \int_0^t\mathcal{G} (s,\sy_s) d\mathcal{W}_s$$
which we pose for a triplet of embedded Hilbert Spaces $$V \hookrightarrow H \hookrightarrow U$$ whereby the embeddings are continuous linear injections, and $H \xhookrightarrow{} U$ is compact. The equation (\ref{thespde}) is posed on a time interval $[0,T]$ for arbitrary but henceforth fixed $T \geq 0$. The mappings $\mathcal{A},\mathcal{G}$ are such that
    $\mathcal{A}:[0,T] \times V \rightarrow U,
    \mathcal{G}:[0,T] \times H \rightarrow \mathscr{L}^2(\mathfrak{U};U)$ are measurable. Understanding $\mathcal{G}$ as a mapping $\mathcal{G}: [0,T] \times H \times \mathfrak{U} \rightarrow U$, we introduce the notation $\mathcal{G}_i(\cdot,\cdot):= \mathcal{G}(\cdot,\cdot,e_i)$. We further impose the existence of a system of elements $(a_k)$ of $V$ which form an orthogonal basis of $U$ and a basis of $H$. Let us define the spaces $V_n:= \textnormal{span}\left\{a_1, \dots, a_n \right\}$ and $\mathcal{P}_n$ as the orthogonal projection to $V_n$ in $U$, that is $$\mathcal{P}_n:f \mapsto \sum_{k=1}^n\inner{f}{a_k}_Ua_k.$$
    It is required that the $(\mathcal{P}_n)$ are uniformly bounded in $H$, which is to say that there exists a constant $c$ independent of $n$ such that for all $f \in H$, \begin{equation} \label{uniform bounds of projection}
        \norm{\mathcal{P}_nf}_H \leq c\norm{f}_H.
    \end{equation}
    Moreover, our setup can be expanded by considering the induced Gelfand Triple $$H \xhookrightarrow{} U \xhookrightarrow{} H^*$$
defined relative to the inclusion mapping $i: H \rightarrow U$; indeed, the embedding of $U$ into $H^*$ is given by the composition of the isomorphism mapping $U$ into $U^*$ with the adjoint $i^*: U^* \rightarrow H^*$. In particular, the duality pairing between $H$ and $H^*$, $\inner{\cdot}{\cdot}_{H^* \times H}$, is compatible with $\inner{\cdot}{\cdot}_U$ in the sense that for for any $f \in U$, $g \in H$,
$$\inner{f}{g}_{H^* \times H} = \inner{f}{g}_U.$$
We assume that $\mathcal{A}:[0,T] \times H \rightarrow H^*$ is measurable. Specific bounds on the mappings $\mathcal{A}$ and $\mathcal{G}$ will be imposed in Assumption Sets 1, 2 and 3. In order to make the assumptions we introduce some more notation here: we shall let $c_{\cdot}:[0,T]\rightarrow \R$ denote any bounded function, and for any constant $p \in \R$ we define the functions $K_U: U \rightarrow \R$, $K_H: H \rightarrow \R$, $K_V: V \rightarrow \R$ by
\begin{equation} \nonumber
    K_U(\phi)= 1 + \norm{\phi}_U^p, \quad K_H(\phi)= 1 + \norm{\phi}_H^p, \quad K_V(\phi)= 1 + \norm{\phi}_V^p.
\end{equation}
We may also consider these mappings as functions of two variables, e.g. $K_U: U \times U \rightarrow \R$ by $$K_U(\phi,\psi) = 1 + \norm{\phi}_U^p + \norm{\psi}_U^p.$$
Our assumptions will be stated for `the existence of a $K$ such that...' where we really mean `the existence of a $p$ such that, for the corresponding $K$, ...'.

\section{Martingale Weak Solutions} \label{section weak}

We first introduce the necessary assumptions for the main result of this section.

\subsection{Assumption Set 1} \label{subby assumption}

 Recall the setup and notation of Subsection \ref{subs functional framework}. We assume that there exists a $c_{\cdot}$, $K$ and $\gamma > 0$ such that for all $\phi,\psi \in V$, $f \in H$ and $t \in [0,T]$:
 
 
  \begin{assumption} \label{new assumption 1} \begin{align}
     \label{111} \norm{\mathcal{A}(t,f)}_{H^*} +\sum_{i=1}^\infty \norm{\mathcal{G}_i(t,f)}^2_U &\leq c_t K_U(f)\left[1 + \norm{f}_H^2\right],\\ \label{222}
     \norm{\mathcal{A}(t,\phi) - \mathcal{A}(t,\psi)}_U^2 &\leq  c_tK_V\norm{\phi-\psi}_V^2,\\ \label{333}
    \sum_{i=1}^\infty \norm{\mathcal{G}_i(t,\phi) - \mathcal{G}_i(t,\psi)}_U^2 &\leq c_tK_V(\phi,\psi)\norm{\phi-\psi}_H^2.
 \end{align}
 \end{assumption}

\begin{assumption} \label{first assumption for uniform bounds}
    \begin{align}
   \label{uniformboundsassumpt1actual}  2\inner{\mathcal{A}(t,\phi)}{\phi}_U + \sum_{i=1}^\infty\norm{\mathcal{G}_i(t,\phi)}_U^2 &\leq c_t\left[1 + \norm{\phi}_U^2\right] - \gamma\norm{\phi}_H^2,\\  \label{uniformboundsassumpt2actual}
    \sum_{i=1}^\infty \inner{\mathcal{G}_i(t,\phi)}{\phi}^2_U &\leq c_t\left[1 + \norm{\phi}_U^4\right].
\end{align}
\end{assumption}

\begin{assumption}\footnote{In fact in (\ref{tightnessassumpt1}), the exponent $3/2$ could be replaced by any $q < 2$.}\label{tightness assumptions}
    \begin{align}
   \label{tightnessassumpt1}  \inner{\mathcal{A}(t,\phi)}{f}_U  &\leq c_t\left[K_U(\phi) + \norm{\phi}_H^{\frac{3}{2}} \right]\left[K_U(f) + \norm{f}_H^{\frac{3}{2}} \right],\\  \label{tightnessassumpt2}
    \sum_{i=1}^\infty \inner{\mathcal{G}_i(t,\phi)}{f}^2_U &\leq c_tK_U(\phi)K_H(f).
\end{align}
\end{assumption}

We remark that by taking negative $f$, the above is true for the absolute value of the left hand side as well. The same holds with $\psi$ below.

\begin{assumption}\label{limity assumptions}
    \begin{align}
   \label{limityassumpt1}  \inner{\mathcal{A}(t,\phi) - A(t,f)}{\psi}_{H^* \times H}  &\leq c_tK_V(\psi)\left[1 + \norm{\phi}_H + \norm{f}_H \right]\norm{\phi - f}_U,\\  \label{limityassumpt2}
    \sum_{i=1}^\infty \inner{\mathcal{G}_i(t,\phi) - \mathcal{G}_i(t,f)}{\psi}^2_U &\leq c_tK_V(\psi)\norm{\phi - f}_U^2.
\end{align}
\end{assumption}

We briefly comment on the purpose of each assumption:
\begin{itemize}
    \item Assumption \ref{new assumption 1} ensures that the integrals in the definition of a martingale weak solution (Definition \ref{definitionofspacetimeweakmartingale}) are well-defined. Additionally, it ensures that the Galerkin System is well-posed (Definition \ref{definition of local strong solution to galerkin equation}). 
    \item Assumption \ref{first assumption for uniform bounds} is used to show uniform estimates for the Galerkin System (Proposition \ref{prop for first energy}).
    \item Assumption \ref{tightness assumptions} is employed to demonstrate tightness of the Galerkin System (Propositions \ref{prop for tightness}, \ref{prop for tightness two}).
    \item Assumption \ref{limity assumptions} facilitates passage to the limit for a martingale weak solution (Proposition \ref{prop for first energy bar}).

    \end{itemize}

\subsection{Definitions and Main Result} \label{subby marty weak def and res}

We now state the definition and main result for martingale weak solutions.

\begin{definition} \label{definitionofspacetimeweakmartingale}
Let $\sy_0: \Omega \rightarrow U$ be $\mathcal{F}_0-$measurable. If there exists a filtered probability space $\left(\tilde{\Omega},\tilde{\mathcal{F}},(\tilde{\mathcal{F}}_t), \tilde{\mathbbm{P}}\right)$, a Cylindrical Brownian Motion $\tilde{\mathcal{W}}$ over $\mathfrak{U}$ with respect to $\left(\tilde{\Omega},\tilde{\mathcal{F}},(\tilde{\mathcal{F}}_t), \tilde{\mathbbm{P}}\right)$, an $\mathcal{F}_0-$measurable $\tilde{\sy}_0: \tilde{\Omega} \rightarrow U$ with the same law as $\sy_0$, and a progressively measurable process $\tilde{\sy}$ in $H$ such that for $\tilde{\mathbb{P}}-a.e.$ $\tilde{\omega}$, $\tilde{\sy}_{\cdot}(\omega) \in L^{\infty}\left([0,T];U\right)\cap C_w\left([0,T];U\right) \cap L^2\left([0,T];H\right)$ and
\begin{align} 
       \tilde{\sy}_t = \tilde{\sy}_0 + \int_0^t \mathcal{A}(s,\tilde{\sy}_s)ds + \int_0^t\mathcal{G} (s,\tilde{\sy}_s) d\mathcal{W}_s \label{newid1}
\end{align}
holds $\tilde{\mathbb{P}}-a.s.$ in $H^*$ for all $t \in [0,T]$, then $\tilde{\sy}$ is said to be a martingale weak solution of the equation (\ref{thespde}).\footnote{A detailed justification that the terms in this definition are well defined is given in [\cite{goodair2022stochastic}] Subsections 2.2 and 2.4, referring to (\ref{111}).}
\end{definition}

\begin{remark} \label{markydoo}
    The progressive measurability condition on $\tilde{\sy}_{\cdot}$ may look a little suspect as $\sy_0$ itself may only belong to $U$ and not $H$ making it impossible for $\tilde{\sy}_{\cdot}$ to be even adapted in $H$. We are mildly abusing notation here; what we really ask is that there exists a process $\py$ which is progressively measurable in $H$ and such that $\py_{\cdot} = \tilde{\sy}_{\cdot}$ almost surely over the product space $\Omega \times [0,T]$ with product measure $\mathbbm{P}\times \lambda$. In particular, the processes $\py$ and $\tilde{\sy}$ may disagree at time zero. Moreover, we note that whether or not the set of full probability appearing in the $\mathbbm{P}-a.s.$ condition depends on $t$ is not of concern; if it did, we could intersect all such sets at rational times and fill in the irrational times with the continuity in $H^*$, leading to a set of full probability independent of $t$. 
\end{remark}


\begin{theorem} \label{theorem for martingale weak existence}
    Let Assumption Set 1 hold. For any given $\mathcal{F}_0-$measurable $\sy_0 \in L^\infty\left(\Omega;U\right)$, there exists a martingale weak solution of the equation (\ref{thespde}). 
\end{theorem}

The remainder of this section is dedicated to the proof of Theorem \ref{theorem for martingale weak existence}.

\subsection{The Galerkin System} \label{subbiegalerkin}
\label{subsection:Existence Method for a Regular Truncated Initial Condition}

We assume that  \begin{equation} \label{boundedinitialcondition}
    \sy_0 \in L^{\infty}(\Omega;U).
\end{equation} as in Theorem \ref{theorem for martingale weak existence}, and consider the Galerkin Equations
\begin{equation} \label{nthorderGalerkin}
       \sy^n_t = \sy^n_0 + \int_0^t \mathcal{P}_n\mathcal{A}(s,\sy^n_s)ds + \int_0^t\mathcal{P}_n\mathcal{G} (s,\sy^n_s) d\mathcal{W}_s
\end{equation}
for the initial condition $\sy^n_0:= \mathcal{P}_n\sy_0$ and $\mathcal{P}_n\mathcal{G} (e_i,\cdot):=\mathcal{P}_n\mathcal{G}_{i}(\cdot).$ Note that $\norm{\sy^n_0}_{U} \leq \norm{\sy_0}_U$ as each $\mathcal{P}_n$ is an orthogonal projection in $U$, so in particular
\begin{equation} \label{uniformboundofinitialcondition}
 \sup_{n\in\N}\norm{\sy^n_0}_{L^{\infty}(\Omega,U)}^2 < \infty.
\end{equation}

\begin{remark}
We may consider the finite dimensional $V_n$ as a Hilbert Space equipped with any of the equivalent $V,H,U$ inner products. 
\end{remark}

\begin{definition} \label{definition of local strong solution to galerkin equation}
A pair $(\sy^n,\tau)$ where $\tau$ is a $\mathbbm{P}-a.s.$ positive stopping time and $\sy^n$ is an adapted process in $V_n$ such that for $\mathbbm{P}-a.e.$ $\omega$, $\sy^n_{\cdot}(\omega) \in C\left([0,T];V_n\right)$ for all $T>0$, is said to be a local strong solution of the equation (\ref{nthorderGalerkin}) if the identity
\begin{equation} \label{identityindefinitionoflocalgalerkinsolution}
    \sy^n_{t} = \sy^n_0 + \int_0^{t\wedge \tau} \mathcal{P}_n\mathcal{A}(s,\sy^n_s)ds + \int_0^{t \wedge \tau}\mathcal{P}_n\mathcal{G} (s,\sy^n_s) d\mathcal{W}_s
\end{equation}
holds $\mathbbm{P}-a.s.$ in $V_n$ for all $t \in [0,T]$.
\end{definition}

We note that composition with $\mathcal{P}_n$ does not disturb the measurability or boundedness of the mappings. As $\mathcal{P}_n$ is continuous in $U$ then it is measurable, so $\mathcal{P}_n\mathcal{A}: [0,T] \times V \rightarrow U$ is measurable and similarly for $\mathcal{P}_n\mathcal{G}$. Moreover $\mathcal{P}_n$ is bounded with respect to the $H^*$ norm; for $\phi \in U$, $$\norm{\mathcal{P}_n\phi}_{H^*} = \norm{\sum_{k=1}^n\inner{\phi}{a_k}_Ua_k}_{H^*} = \norm{\sum_{k=1}^n\inner{\phi}{a_k}_{H^* \times H}a_k}_{H^*} \leq \left(\sum_{k=1}^n\norm{a_k}_H\norm{a_k}_{H^*}\right)\norm{\phi}_{H^*}.$$ 
 Therefore the terms in Definition \ref{definition of local strong solution to galerkin equation} make sense, where we again refer to [\cite{goodair2022stochastic}] Subsections 2.2 and 2.4. In looking to deduce the existence of such a solution, we first consider a truncated version of the equation. For any fixed $R>0$, we introduce the function $f_R: [0,\infty) \rightarrow [0,1]$ constructed such that $$f_R \in C^{\infty}\left([0,\infty);[0,1]\right), \qquad f_R(x)=1 \ \forall x \in [0,R], \qquad f_R(x)=0 \ \forall x \in [2R,\infty).$$ We consider now the equation
\begin{equation} \label{truncatedgalerkin}
     \sy^{n,R}_t = \sy^{n,R}_0 + \int_0^t f_R\left(\norm{\sy^{n,R}_s}_H^2\right)\mathcal{P}_n\mathcal{A}(s,\sy^{n,R}_s)ds + \int_0^tf_R\left(\norm{\sy^{n,R}_s}_H^2\right)\mathcal{P}_n\mathcal{G} (s,\sy^{n,R}_s) d\mathcal{W}_s
\end{equation}
for $\sy^{n,R}_0:=\sy^{n}_0$ which we use as an intermediary step to deduce the existence of solutions as in Definition \ref{definition of local strong solution to galerkin equation}. Solutions of the truncated equation are defined in the sense of Proposition \ref{Skorotheorem}, and due to Assumption \ref{new assumption 1} we can apply this proposition in the case of $\mathcal{H}:=V_n$, $ \mathscr{A}:=\mathcal{P}_n\mathcal{A}$, $ \mathscr{G}:=\mathcal{P}_n\mathcal{G}$ to deduce the existence of solutions to (\ref{truncatedgalerkin}) for all $n\in\N, R>0$.\\

The motivation for considering (\ref{truncatedgalerkin}) is to prove the existence of local strong solutions to (\ref{nthorderGalerkin}) by considering local intervals of existence on which the truncation threshold isn't reached. More than this, for any first hitting time in the fundamental $L^\infty([0,T];U) \cap L^2([0,T];H)$ norm, we show that a large enough truncation can be chosen so that solutions to (\ref{nthorderGalerkin}) exist up until the first hitting time. 

\begin{lemma} \label{localexistenceofsolutionsofGalerkin}
For any $t>0$, $M>1$ and fixed $n \in N$, there exists an $R>0$ such that the following holds; letting $\sy^{n,R}$ be the solution of (\ref{truncatedgalerkin}) and $$\tau^{M,t}_{n,R}(\omega) := t \wedge \inf\left\{s \geq 0: \sup_{r \in [0,s]}\norm{\sy^{n,R}_{r}(\omega)}^2_U + \int_0^s\norm{\sy^{n,R}_{r}(\omega)}^2_Hdr \geq M + \norm{\sy^n_0(\omega)}_U^2 \right\}$$ then $(\sy^{n,R}_{\cdot \wedge \tau^{M,t}_{n,R}}, \tau^{M,t}_{n,R})$ is a local strong solution of (\ref{nthorderGalerkin}).
\end{lemma}

\begin{proof}
See [\cite{goodair2023existence}] Lemma 3.18. 
\end{proof}


Of course the application of Proposition \ref{Skorotheorem} gave us the uniqueness of such solutions as well, which does not \textit{exactly} pass over to uniqueness of the local strong solutions of (\ref{nthorderGalerkin}) as in Lemma \ref{localexistenceofsolutionsofGalerkin}, however this uniqueness can be proven just as Proposition \ref{Skorotheorem} was ([\cite{goodair2022stochastic}] Theorem 3.1.2). From here, one can deduce the existence of a unique \textit{maximal} strong solution $(\sy^n, \Theta^n)$ of the equation (\ref{nthorderGalerkin}) as in [\cite{goodair2023existence}] Theorems 3.32 and 3.34 (see Definition 3.12). Introducing the stopping times 
\begin{equation}\label{tauMtn}\tau^{M,t}_n(\omega) := t \wedge \inf\left\{s \geq 0: \sup_{r \in [0,s]}\norm{\sy^{n}_{r}(\omega)}^2_U + \int_0^s\norm{\sy^{n}_{r}(\omega)}^2_Hdr \geq M + \norm{\sy^n_0(\omega)}_U^2 \right\}\end{equation}
then we claim that $(\sy^{n}_{\cdot \wedge \tau^{M,t}_{n}}, \tau^{M,t}_{n})$ is a local strong solution of (\ref{nthorderGalerkin}).
By construction of the maximal solution ([\cite{goodair2023existence}] Theorem 3.32) then $\sy^{n}_{\cdot \wedge \tau^{M,t}_{n}}$ and $\sy^{n,R}_{\cdot \wedge \tau^{M,t}_{n,R}}$ are indistinguishable for $R$ chosen as in Lemma \ref{localexistenceofsolutionsofGalerkin}, which implies that $\tau^{M,t}_{n} = \tau^{M,t}_{n,R}$ as well. 

\begin{remark}
The stopping time $\tau^{M,t}_n$ defined in (\ref{tauMtn}) is the first hitting time with respect to a norm which is central to the arguments of the paper. As such for a function $\py \in C([0,t];U) \cap L^2([0,t];H)$ we define the norm \begin{align} \label{new norm}
    \norm{\py}^2_{UH,t}:= \sup_{r \in [0,t]}\norm{\py_{r}}^2_U + \int_0^t\norm{\py_{r}}^2_Hdr
\end{align}
making explicit the dependence on the time $t$.
\end{remark}

From [\cite{goodair2023existence}] Lemma 3.36 we also have the relation  
\begin{equation} \label{relation} \mathbbm{P}\left( \{ \omega \in \Omega: \tau^{M,t}_{n}(\omega) < \Theta^n(\omega)\}\right) = 1\end{equation} for any $M>1$ and $t>0$. We wish to remove the need for localisation in the strong solution of (\ref{nthorderGalerkin}), by showing that our maximal time must be greater than $T$. That is, we want to show that \begin{equation}\label{wanting to showy}\mathbbm{P}\left(\{\omega \in \Omega: \Theta^n(\omega) \leq T\}\right) = 0.\end{equation} Moreover,
\begin{align*}\mathbbm{P}\left(\{\omega \in \Omega: \Theta^n(\omega) \leq T\}\right) &\leq \mathbbm{P}\left(\left\{\omega \in \Omega: \sup_{M \in \N}\tau^{M,T+1}_n(\omega) \leq T\right\}\right)\\
&= \mathbbm{P}\left(\bigcap_{M \in \N}\left\{\omega \in \Omega: \tau^{M,T+1}_n(\omega) \leq T\right\}\right)\\
&= \lim_{M \rightarrow \infty}\mathbbm{P}\left(\left\{\omega \in \Omega: \tau^{M,T+1}_n(\omega) \leq T\right\}\right)
\end{align*}
from (\ref{relation}) and the fact that $\tau^{M,T+1}_n$ is increasing in $M$. From the characterisation of $\tau^{M,T+1}_n$ note that 
\begin{align*}\left\{\omega \in \Omega: \tau^{M,T+1}_n(\omega)\leq T\right\}  = \left\{\omega \in \Omega: \norm{\sy^n(\omega)}_{UH,T \wedge  \tau^{M,T+1}_n(\omega)}^2 \geq M + \norm{\sy^n_0(\omega)}_U^2\right\}\end{align*} so a simple application of Chebyshev's Inequality informs us that \begin{equation}\label{information}\mathbbm{P}\left(\left\{\omega \in \Omega: \tau^{M,T+1}_n(\omega) \leq T\right\}\right) \leq \frac{1}{M}\mathbbm{E}\left[\norm{\sy^n}_{UH,T \wedge  \tau^{M,T+1}_n}^2 - \norm{\sy^n_0}_U^2\right].\end{equation} 
Therefore, to show (\ref{wanting to showy}) we want a control on this expectation independent of $M$. Let us introduce the notation
\begin{equation} \label{truncation notation}
    \sy^{n,M}_{\cdot}:=\sy^n_{\cdot \wedge \tau^{M,T+1}_n}, \qquad \check{\sy}^n_\cdot:= \sy^n_{\cdot}\mathbbm{1}_{\cdot \leq \tau^{M,T+1}_n}
    \end{equation}
    where we note dependence on $M$ is implicit in $\check{\sy}^n$, and by construction
    \begin{equation} \label{norm equivalence truncation}
        \norm{\sy^n}_{UH,T \wedge  \tau^{M,T+1}_n}^2 = \norm{\check{\sy}^n}_{UH,T}^2.
    \end{equation}

\begin{proposition} \label{prop for first energy}
    Let $(\sy^n,\Theta^n)$ be the maximal strong solution of equation (\ref{nthorderGalerkin}). There exists a constant $C$ independent of $M,n$ such that
    \begin{equation}\label{first result}\mathbbm{E}\left(\norm{\sy^n}_{UH,T \wedge  \tau^{M,T+1}_n}^2  \right) \leq C\left[\mathbbm{E}\left(\norm{\sy^n_0}_U^2\right) + 1\right].\end{equation}
\end{proposition}

\begin{proof}
 By equipping $V_n$ with the $U$ inner product we can apply an It\^{o} Formula, Proposition \ref{rockner prop}, to see that for any $0 \leq r \leq T$, the identity \begin{align}\nonumber\norm{\sy^{n,M}_{r}}_U^2 &= \norm{\sy^n_0}_U^2 + 2\int_0^{r\wedge\tau^{M,t}_n}\inner{\mathcal{P}_n\mathcal{A}(s,\sy^n_s)}{\sy^n_s}_Uds\\ & \qquad  + \int_0^{r\wedge\tau^{M,t}_n}\sum_{i=1}^\infty\norm{\mathcal{P}_n\mathcal{G}_i(s,\sy^n_s)}_U^2ds + 2\sum_{i=1}^\infty\int_0^{r\wedge\tau^{M,t}_n}\inner{\mathcal{P}_n\mathcal{G}_i(s,\sy^n_s)}{\sy^n_s}_UdW^i_s\nonumber
\end{align}
holds $\mathbbm{P}-a.s.$. Recalling (\ref{truncation notation}) we write this as 
\begin{align}\nonumber\norm{\sy^{n,M}_{r}}_U^2 = \norm{\sy^n_0}_U^2 &+ 2\int_0^{r}\inner{\mathcal{P}_n\mathcal{A}(s,\sy^n_s)}{\sy^n_s}_U\mathbbm{1}_{s \leq \tau^{M,T+1}_n}ds  + \int_0^{r}\sum_{i=1}^\infty\norm{\mathcal{P}_n\mathcal{G}_i(s,\sy^n_s)}_U^2\mathbbm{1}_{s \leq \tau^{M,T+1}_n}ds\\ & + 2\sum_{i=1}^\infty\int_0^{r}\inner{\mathcal{P}_n\mathcal{G}_i(s,\check{\sy}^n_s)}{\check{\sy}^n_s}_UdW^i_s.\label{nownonwnumb}
\end{align}
Note that we have left the indicator function outside of the time integral terms as we have no linearity assumption to take it through the  $\norm{\mathcal{P}_n\mathcal{G}_i(s,\sy^n_s)}_U^2$ term: more precisely, it may be the case that $\mathcal{P}_n\mathcal{G}_i(s,0) \neq 0$. To do this for the stochastic integral we are just relying on the linearity of the inner product. As $\mathcal{P}_n$ is an orthogonal projection, we use the self-adjoint property and that $\norm{\mathcal{P}_n\mathcal{G}_i(s,\sy^n_s)}_U^2\leq \norm{\mathcal{G}_i(s,\sy^n_s)}_U^2$ to reduce this to
\begin{align}\nonumber\norm{\sy^{n,M}_{r}}_U^2 \leq \norm{\sy^n_0}_U^2 &+ 2\int_0^{r}\inner{\mathcal{A}(s,\sy^n_s)}{\sy^n_s}_U\mathbbm{1}_{s \leq \tau^{M,T+1}_n}ds  + \int_0^{r}\sum_{i=1}^\infty\norm{\mathcal{G}_i(s,\sy^n_s)}_U^2\mathbbm{1}_{s \leq \tau^{M,T+1}_n}ds\\ & + 2\sum_{i=1}^\infty\int_0^{r}\inner{\mathcal{G}_i(s,\check{\sy}^n_s)}{\check{\sy}^n_s}_UdW^i_s.\nonumber
\end{align}
We then apply Assumption \ref{first assumption for uniform bounds}, (\ref{uniformboundsassumpt1actual}), and incorporate the indicator function into the norm to obtain that
\begin{align}\nonumber\norm{\sy^{n,M}_{r}}_U^2 \leq \norm{\sy^n_0}_U^2 + c\int_0^{r}1 + \norm{\check{\sy}^n_s}_U^2ds - \gamma\int_0^{r}\norm{\check{\sy}^n_s}_H^2ds  + 2\sum_{i=1}^\infty\int_0^{r}\inner{\mathcal{G}_i(s,\check{\sy}^n_s)}{\check{\sy}^n_s}_UdW^i_s\nonumber
\end{align}
still $\mathbbm{P}-a.s.$. We now look to take the supremum over all terms, which we can do by controlling the stochastic integral with the absolute value, and further taking expectation. It should be appreciated that \begin{equation}\label{toappreciate}
\norm{\check{\sy}^n(\omega)}_{UH,T}^2 \leq M + \norm{\sy_0}_{L^{\infty}(\Omega,U)}^2\end{equation}
which ensures the integrability of all terms, and that the stochastic integral is a genuine square integrable martingale. By taking the supremum in time, expectation, and applying the Burkholder-Davis-Gundy Inequality all in one step, we arrive at 
\begin{align} \nonumber   \mathbbm{E}\left(\sup_{r\in[0,T]}\norm{\sy^{n,M}_{r}}_U^2\right) &+ \gamma\mathbbm{E}\int_0^{T}\norm{\check{\sy}^n_s}_H^2ds \leq 2\mathbbm{E}\left(\norm{\sy^n_0}_U^2\right) + c\\ & \qquad + c\int_0^{T} \mathbbm{E}\left(\norm{\check{\sy}^{n}_s}^2_U\right) ds + c\mathbbm{E}\left(\int_0^{T}\sum_{i=1}^\infty \inner{\mathcal{G}_i\left(s,\check{\sy}^{n}_s\right)}{\check{\sy}^{n}_s}_U^2 ds\right)^\frac{1}{2} \label{oldnewnonum}.
\end{align}
Note that we formally have to take the supremum individually for each term on the left hand side and then sum them, hence the appearance of a $2$ in front of the initial condition. We will freely use Tonelli's Theorem to interchange between expectation and integration in time. The constant $c$ now depends on $T$, which is not meaningful. Turning to the final term, with (\ref{uniformboundsassumpt2actual}) we see that
   \begin{align}\nonumber
        c\mathbbm{E}\left(\int_0^{T}\sum_{i=1}^\infty \inner{\mathcal{G}_i\left(s,\check{\sy}^{n}_s\right)}{\check{\sy}^{n}_s}_U^2 ds\right)^\frac{1}{2} &\leq  c\mathbbm{E}\left(\int_0^{T} 1 + \norm{\check{\sy}^{n}_s}^4_Uds\right)^\frac{1}{2}\\\nonumber
        &\leq c + c\mathbbm{E}\left(\int_0^{T} \norm{\check{\sy}^{n}_s}_U^4ds\right)^\frac{1}{2}\\\nonumber
        &\leq c + c\mathbbm{E}\left(\sup_{r\in[0,T]}\norm{\check{\sy}^{n}_r}_U^2\int_0^{T} \norm{\check{\sy}^{n}_s}_U^2ds\right)^\frac{1}{2}\\
        &\leq c + \frac{1}{2}\mathbbm{E}\left(\sup_{r\in[0,T]}\norm{\check{\sy}^{n}_r}_U^2\right) + c\mathbbm{E}\int_0^{T} \norm{\check{\sy}^{n}_s}_U^2ds \label{the same process}
    \end{align}
having applied Young's Inequality. Substituting into (\ref{oldnewnonum}), and scaling with $\gamma$ as needed, then we have that 
$$\mathbbm{E}\left(\norm{\check{\sy}^{n}}_{UH,T}^2\right)  \leq c\left[\mathbbm{E}\left(\norm{\sy^n_0}_U^2\right) + 1\right] + c\int_0^{T} \mathbbm{E}\left(\norm{\check{\sy}^{n}_s}^2_U\right) ds.$$
One can control $\norm{\check{\sy}^{n}_s}^2_U \leq \norm{\check{\sy}^{n}}^2_{UH,s}$ which is integrable due to (\ref{toappreciate}), from which the standard Gr\"{o}nwall Inequality gives the result, recalling (\ref{norm equivalence truncation}). 

\end{proof}

The expectation in (\ref{information}) is thus finite, so taking the limit $M \rightarrow \infty$ achieves (\ref{wanting to showy}). It is further evident from our calculations that for any $\tau$ such that $(\sy^n,\tau)$ is a local strong solution of the equation (\ref{nthorderGalerkin}), the inequality (\ref{first result}) holds (that is, with $\tau$ replacing $\tau^{M,T+1}_n$) where $C$ is independent of the choice of $\tau$. Therefore we can choose a $\mathbbm{P}-a.s.$ increasing sequence of stopping times which approach $\Theta^n$ by definition of the maximal time, and apply the Monotone Convergence Theorem through the expectation to yield that  \begin{equation}\label{second result}\mathbbm{E}\left(\norm{\sy^n}_{UH,T}^2  \right) \leq C\left[\mathbbm{E}\left(\norm{\sy^n_0}_U^2\right) + 1\right]\end{equation}
where to be precise, we obtain $\mathbbm{E}\left(\norm{\sy^n}_{UH,T \wedge \Theta^n}^2  \right)$ on the left hand side, which is simply $\mathbbm{E}\left(\norm{\sy^n}_{UH,T}^2  \right)$ due to (\ref{wanting to showy}). Moreover for $\mathbbm{P}-a.e.$ $\omega \in \Omega$, we can choose a $\tau(\omega)>T$ such that $(\sy^n,\tau)$ is a local strong solution of the equation (\ref{nthorderGalerkin}), hence $\sy^n$ satisfies the identity (\ref{nthorderGalerkin}) $\mathbbm{P}-a.s.$ for all $t \in [0,T]$. Furthermore, $\sy^n\in C\left([0,T];V_n\right)$; we thus call $\sy^n$ the unique strong solution of the equation (\ref{nthorderGalerkin}). Of course we can bound $\norm{\sy^n_0}_U^2 \leq \norm{\sy_0}_U^2 \leq \norm{\sy_0}_{L^{\infty}(\Omega;U)}^2$ which is finite independent of $n$ and can be substituted in to (\ref{first result}), (\ref{second result}). Combining this with (\ref{information}) achieves that \begin{equation}\label{achievement}\lim_{M \rightarrow \infty}\sup_{n \in \N}\mathbbm{P}\left(\left\{\omega \in \Omega: \tau^{M,T+1}_n(\omega) \leq T\right\}\right) =0.\end{equation}

\subsection{Tightness} \label{subbietight}

We now look to deduce the existence of a process taken as the limit of $\sy^n$ in some sense, which is done through a tightness argument. We pursue this with Lemma \ref{Lemma 5.2} in the Appendix, with the spaces $\mathcal{H}_1:= H$, $\mathcal{H}_2:=U$. Having already demonstrated (\ref{first condition}) we now justify (\ref{second condition}), so for any $\varepsilon > 0$ define the set $$ A^{\delta, n} := \left\{\int_0^{T-\delta}\norm{\sy^n_{s + \delta} - \sy^n_s}^2_{U}ds > \varepsilon\right\}$$
where we remove the explicit dependence on $\omega$ for brevity, and the required condition (\ref{second condition}) is that \begin{equation}\label{requiredy condition} \lim_{\delta \rightarrow 0^+}\sup_{n \in \N}\mathbbm{P}\left(A^{\delta, n}\right) = 0.\end{equation}
We look to show this condition by taking advantage of the control granted to us with the stopping times $\tau^{M,T+1}_n$, and the property (\ref{achievement}). For any $M > 1$,
\begin{align*}
    A^{\delta, n} = A^{\delta, n} \cap \left[  \left\{ \tau^{M,T+1}_n > T  \right\} \cup \left\{ \tau^{M,T+1}_n \leq T  \right\} \right] &= \left[ A^{\delta, n} \cap  \left\{ \tau^{M,T+1}_n > T  \right\}  \right] \cup \left[ A^{\delta, n} \cap \left\{ \tau^{M,T+1}_n \leq T  \right\} \right]\\ &\subset \left[ A^{\delta, n} \cap  \left\{ \tau^{M,T+1}_n > T  \right\}  \right] \cup \left\{ \tau^{M,T+1}_n \leq T  \right\}.
\end{align*}
In particular,
\begin{equation}\label{particularly doo dah}\mathbbm{P}\left(A^{\delta, n}\right) \leq \mathbbm{P}\left(A^{\delta, n} \cap  \left\{ \tau^{M,T+1}_n > T  \right\} \right) + \mathbbm{P}\left(\left\{ \tau^{M,T+1}_n \leq T  \right\} \right)\end{equation}
where \begin{equation}\label{wherey}A^{\delta, n, M}:= A^{\delta, n} \cap  \left\{ \tau^{M,T+1}_n > T  \right\} = \left\{\int_0^{T-\delta}\norm{\sy^{n,M}_{s + \delta} - \sy^{n,M}_s}^2_{U}ds > \varepsilon\right\}\end{equation}
continuing to use the notation (\ref{truncation notation}). From (\ref{particularly doo dah}),
\begin{align*}
    \lim_{\delta \rightarrow 0^+}\sup_{n \in \N}\mathbbm{P}\left(A^{\delta, n}\right) \leq \lim_{\delta \rightarrow 0^+}\sup_{n \in \N}\mathbbm{P}\left(A^{\delta, n, M}\right) + \lim_{\delta \rightarrow 0^+}\sup_{n \in \N}\mathbbm{P}\left(\left\{ \tau^{M,T+1}_n \leq T  \right\} \right)
\end{align*}
holds for any $M >1$, so it must also hold in the limit as $M \rightarrow \infty$. As there is no dependency on $\delta$ in the final probability, we obtain
\begin{align}\nonumber\lim_{\delta \rightarrow 0^+}\sup_{n \in \N}\mathbbm{P}\left(A^{\delta, n}\right) &\leq \lim_{M \rightarrow \infty}\lim_{\delta \rightarrow 0^+}\sup_{n \in \N}\mathbbm{P}\left(A^{\delta, n, M}\right) + \lim_{M \rightarrow \infty}\sup_{n \in \N}\mathbbm{P}\left(\left\{ \tau^{M,T+1}_n \leq T  \right\} \right)\\ &= \lim_{M \rightarrow \infty}\lim_{\delta \rightarrow 0^+}\sup_{n \in \N}\mathbbm{P}\left(A^{\delta, n, M}\right).\label{obtainable}
\end{align}
Recalling the definition (\ref{wherey}), we can apply Chebyshev's Inequality to see that $$ \mathbbm{P}\left(A^{\delta, n, M}\right) \leq \frac{1}{\varepsilon}\mathbbm{E}\int_0^{T-\delta}\norm{\sy^{n,M}_{s + \delta} - \sy^{n,M}_s}^2_{U}ds.$$
Therefore using (\ref{obtainable}), to justify (\ref{requiredy condition}) one may simply prove that $$\lim_{\delta \rightarrow 0^+}\sup_{n \in \N}\mathbbm{E}\int_0^{T-\delta}\norm{\sy^{n,M}_{s + \delta} - \sy^{n,M}_s}^2_{U}ds = 0$$
for any fixed $M >1$. This prompts the following result, which follows a similar method to [\cite{rockner2022well}] Lemma 2.12.

\begin{proposition} \label{prop for tightness}
    For any $M > 1$,
    \begin{equation} \label{required condition}
    \lim_{\delta \rightarrow 0^+}\sup_{n \in \N}\mathbbm{E}\int_0^{T-\delta}\norm{\sy^{n,M}_{s + \delta} - \sy^{n,M}_s}^2ds= 0.
\end{equation} 
Therefore the sequence of the laws of $(\sy^n)$ is tight in the space of probability measures over $L^2\left([0,T];U\right)$.
\end{proposition}

Just ahead of proving the result, we note that from the definition of $K$ in Subsection \ref{subby assumption} and the property (\ref{toappreciate}), that \begin{equation} \label{galerkinboundsatisfiedbystoppingtime2}
    \sup_{r \in [0,t]}K_U(\check{\sy}^n_r(\omega)) \leq c, 
\end{equation}
for a constant $c$ dependent on $M$ and $\sy_0$, a dependence which is not meaningful in the following arguments as these remain fixed.

\begin{proof}
      Similarly to the proof of Proposition \ref{prop for first energy}, observe that for any $s\in[0,T]$,
    \begin{align*}\sy^{n,M}_{s+\delta} = \sy^n_0 + \int_0^{s+\delta} \mathcal{P}_n\mathcal{A}(r,\sy^n_r)\mathbbm{1}_{r \leq \tau^{M,T+1}_n}dr + \int_0^{s+\delta}\mathcal{P}_n\mathcal{G} (r,\sy^n_r)\mathbbm{1}_{r \leq \tau^{M,T+1}_n} d\mathcal{W}_r.
    \end{align*}
 Therefore
   \begin{align}  \sy^{n,M}_{s+\delta} - \sy^{n,M}_{s} = \int_s^{s+\delta} \mathcal{P}_n\mathcal{A}(r,\sy^n_r)\mathbbm{1}_{r \leq \tau^{M,T+1}_n}dr + \int_s^{s+\delta}\mathcal{P}_n\mathcal{G} (r,\sy^n_r)\mathbbm{1}_{r \leq \tau^{M,T+1}_n} d\mathcal{W}_r \label{difference identity}
    \end{align}
which for any fixed $s$ is just an evolution equation in parameter $\delta$, so we can apply the It\^{o} Formula (e.g. Proposition \ref{rockner prop}) to deduce that
\begin{align*}
    \norm{\sy^{n,M}_{s+\delta} - \sy^{n,M}_{s}}_U^2 &=  2\int_s^{s+\delta}\inner{\mathcal{P}_n\mathcal{A}(r,\check{\sy}^n_r)}{\check{\sy}^n_r-\check{\sy}^n_s}_U\mathbbm{1}_{r \leq \tau^{M,T+1}_n} dr\\ &+\int_s^{s+\delta}\sum_{i=1}^{\infty}\norm{\mathcal{P}_n\mathcal{G}_i(r,\check{\sy}^n_r)}^2_U\mathbbm{1}_{r \leq \tau^{M,T+1}_n}dr + 2\int_s^{s+\delta} \inner{\mathcal{P}_n\mathcal{G}(r,\check{\sy}^n_r)}{\check{\sy}^n_r-\check{\sy}^n_s}_U d\mathcal{W}_r
\end{align*}
where we have incorporated the indicator function into the $\sy^n$ and $\sy^{n,M}$. Just as we proceeded from (\ref{nownonwnumb}) in Proposition \ref{first result}, we can use properties of the projection $\mathcal{P}_n$ and apply the assumption (\ref{uniformboundsassumpt1actual}), to obtain the inequality
\begin{align*}
    \norm{\sy^{n,M}_{s+\delta} - \sy^{n,M}_{s}}_U^2 &\leq 
    c\int_s^{s+\delta}1 + \norm{\check{\sy}^n_r}_U^2dr -
    2\int_s^{s+\delta}\inner{\mathcal{A}(r,\check{\sy}^n_r)}{\check{\sy}^n_s}_U dr\\ & + 2\int_s^{s+\delta} \inner{\mathcal{G}(r,\check{\sy}^n_r)}{\check{\sy}^n_r-\check{\sy}^n_s}_U d\mathcal{W}_r
\end{align*}
where we have simply dropped the helpful $-\gamma\norm{\check{\sy}^n_r}_H^2$ contribution from consideration. Using (\ref{galerkinboundsatisfiedbystoppingtime2}) in the first integral, c.f. (\ref{toappreciate}), we can actually pass to the bound
\begin{align*}
    \norm{\sy^{n,M}_{s+\delta} - \sy^{n,M}_{s}}_U^2 \leq 
    c\delta -
    2\int_s^{s+\delta}\inner{\mathcal{A}(r,\check{\sy}^n_r)}{\check{\sy}^n_s}_U dr + 2\int_s^{s+\delta} \inner{\mathcal{G}(r,\check{\sy}^n_r)}{\check{\sy}^n_r-\check{\sy}^n_s}_U d\mathcal{W}_r
\end{align*}
integrating the constant over the time interval of size $\delta$. To this we can invoke Assumption \ref{tightness assumptions}, (\ref{tightnessassumpt1}), and then take expectation to reach
\begin{align}
    \mathbbm{E}\left(\norm{\sy^{n,M}_{s+\delta} - \sy^{n,M}_{s}}_U^2\right) \leq 
    c\delta + c\mathbbm{E}\int_s^{s+\delta}\left(1 + \norm{\check{\sy}^n_r}_H^{\frac{3}{2}}\right)\left(1 + \norm{\check{\sy}^n_s}_H^{\frac{3}{2}}\right) dr \label{reachfor the}
\end{align}
where the stochastic integral of null expectation drops out. Of course, in aiming to show (\ref{required condition}), we are actually interested in $$\lim_{\delta \rightarrow 0^+}\sup_{n \in \N}\mathbbm{E}\int_0^{T-\delta}\norm{\sy^{n,M}_{s + \delta} - \sy^{n,M}_s}_U^2ds. $$
With (\ref{reachfor the}), use of the Fubini-Tonelli Theorem and considering the iterated integral as an integral over the product space, we obtain
\begin{align}
        \nonumber &\lim_{\delta \rightarrow 0^+}\sup_{n \in \N}\mathbb{E}\int_0^{T-\delta}\norm{\sy^{n,M}_{s+\delta} - \sy^{n,M}_{s}}^2ds\\ \nonumber& \qquad \qquad \qquad = \lim_{\delta \rightarrow 0^+}\sup_{n \in \N}\int_0^{T-\delta}\mathbb{E}\norm{\sy^{n,M}_{s+\delta} - \sy^{n,M}_{s}}^2ds\\\nonumber &\qquad \qquad \qquad\leq  \lim_{\delta \rightarrow 0^+}\sup_{n \in \N}\int_0^{T-\delta}\left[c\delta + c\mathbbm{E}\int_s^{s+\delta}\left(1 + \norm{\check{\sy}^n_r}_H^{\frac{3}{2}}\right)\left(1 + \norm{\check{\sy}^n_s}_H^{\frac{3}{2}}\right) dr \right]ds\\\nonumber
        &\qquad \qquad \qquad=  c\lim_{\delta \rightarrow 0^+}\sup_{n \in \N}\mathbbm{E}\int_0^{T-\delta}\int_s^{s+\delta}\left(1 + \norm{\check{\sy}^n_r}_H^{\frac{3}{2}}\right)\left(1 + \norm{\check{\sy}^n_s}_H^{\frac{3}{2}}\right) drds\\ \label{finalnumber}
        &\qquad \qquad \qquad=  c\lim_{\delta \rightarrow 0^+}\sup_{n \in \N}\mathbb{E}\int_0^{T}\int_{0 \vee (r - \delta)}^{r \wedge T - \delta} \left(1 + \norm{\check{\sy}^n_r}_H^{\frac{3}{2}}\right)\left(1 + \norm{\check{\sy}^n_s}_H^{\frac{3}{2}}\right) dsdr.
\end{align}
Note that for each fixed $r$,
\begin{align}
   \nonumber &\int_{0 \vee (r - \delta)}^{r \wedge T - \delta} \left(1 + \norm{\check{\sy}^n_r}_H^{\frac{3}{2}}\right)\left(1 + \norm{\check{\sy}^n_s}_H^{\frac{3}{2}}\right) ds\\ & \qquad \qquad \qquad = \left(1 + \norm{\check{\sy}^n_r}_H^{\frac{3}{2}}\right)\int_{0 \vee (r - \delta)}^{r \wedge T - \delta} \left(1 + \norm{\check{\sy}^n_s}_H^{\frac{3}{2}}\right) ds\\\nonumber
    &\qquad \qquad \qquad\leq \left(1 + \norm{\check{\sy}^n_r}_H^{\frac{3}{2}}\right)\left[\left(\int_{0 \vee (r - \delta)}^{r \wedge T - \delta}1ds\right)^{\frac{1}{4}}\left(\int_{0 \vee (r - \delta)}^{r \wedge T - \delta} \left(1 + \norm{\check{\sy}^n_s}_H^{\frac{3}{2}}\right)^{\frac{4}{3}} ds\right)^{\frac{3}{4}} \right]\\\nonumber
    &\qquad \qquad \qquad\leq c\delta^{\frac{1}{4}}\left(1 + \norm{\check{\sy}^n_r}_H^{2}\right)\left(\int_{0 \vee (r - \delta)}^{r \wedge T - \delta}1 + \norm{\check{\sy}^n_s}_H^{2}ds\right)^{\frac{3}{4}}\\ \label{usage}
    &\qquad \qquad \qquad \leq c\delta^{\frac{1}{4}}\left(1 + \norm{\check{\sy}^n_r}_H^{2}\right)
\end{align}
having used (\ref{toappreciate}) in the final line, and H\"{o}lder's Inequality beforehand. Substituting this back into (\ref{finalnumber}),
$$\lim_{\delta \rightarrow 0^+}\sup_{n \in \N}\mathbb{E}\int_0^{T-\delta}\norm{\sy^{n,M}_{s+\delta} - \sy^{n,M}_{s}}^2ds  \leq  c\lim_{\delta \rightarrow 0^+}\sup_{n \in \N}\mathbb{E}\int_0^{T}c\delta^{\frac{1}{4}}\left(1 + \norm{\check{\sy}^n_r}_H^2\right)dr \leq \lim_{\delta \rightarrow 0^+}c\delta^{\frac{1}{4}} = 0 $$
which concludes the proof.

\end{proof}

To achieve a characterisation of the limit process at each time, we will need to show tightness in $\mathcal{D}\left([0,T];H^*\right)$. The idea is to apply Lemma \ref{lemma for D tight}, for  $\mathcal{Y} = H$ and $\mathcal{H} = U$. The condition (\ref{first condition primed}) has already been shown from the stronger (\ref{second result}) so to apply the Lemma we only need to verify (\ref{second condition primed}). This is reminiscent of the condition (\ref{second condition}) just verified, so just as we saw for Proposition \ref{prop for tightness} it is sufficient to verify the following.

\begin{proposition} \label{prop for tightness two}
    For any sequence of stopping times $(\gamma_n)$ with $\gamma_n: \Omega \rightarrow [0,T]$, and any  $M > 1$, $f \in H$,
    \begin{equation} \label{required condition primed}
    \lim_{\delta \rightarrow 0^+}\sup_{n \in \N}\mathbbm{E}\left(\left\vert\left\langle \sy^{n,M}_{(\gamma_n + \delta) \wedge T} - \sy^{n,M}_{\gamma_n} , f \right\rangle_U \right\vert\right)= 0.
\end{equation} 
Therefore the sequence of the laws of $(\sy^n)$ is tight in the space of probability measures over $\mathcal{D}\left([0,T];H^*\right)$.
\end{proposition}

\begin{proof}
    Recalling (\ref{difference identity}), substituting in $\gamma_n$ for $s$ and stopping the process at $T$, we see that
$$\sy^{n,M}_{(\gamma_n+\delta)\wedge T} - \sy^{n,M}_{\gamma_n} = \int_{\gamma_n}^{(\gamma_n+\delta)\wedge T} \mathcal{P}_n\mathcal{A}(r,\check{\sy}^n_r)\mathbbm{1}_{r \leq \tau^{M,T+1}_n}dr + \int_{\gamma_n}^{(\gamma_n+\delta) \wedge T}\mathcal{P}_n\mathcal{G} (r,\check{\sy}^n_r)\mathbbm{1}_{r \leq \tau^{M,T+1}_n} d\mathcal{W}_r$$
holds $\mathbbm{P}-a.s.$, to which we take the inner product with arbitrary $f \in H$ and absolute value to see that
\begin{align*}
    \left\vert \left\langle \sy^{n,M}_{(\gamma_n+\delta)\wedge T} - \sy^{n,M}_{\gamma_n} , f \right\rangle_U\right\vert &\leq \left\vert\int_{\gamma_n}^{(\gamma_n+\delta)\wedge T} \inner{\mathcal{A}(r,\check{\sy}^n_r)}{\mathcal{P}_nf}_U\mathbbm{1}_{r \leq \tau^{M,T+1}_n}dr\right\vert\\ &+ \left\vert\int_{\gamma_n}^{(\gamma_n+\delta) \wedge T}\inner{\mathcal{G} (r,\check{\sy}^n_r)}{\mathcal{P}_nf}_U\mathbbm{1}_{r \leq \tau^{M,T+1}_n} d\mathcal{W}_r\right\vert
\end{align*}
having also carried over the projection $\mathcal{P}_n$. Similarly to Proposition \ref{prop for tightness}, we invoke the assumption (\ref{tightnessassumpt1}) to achieve the inequality
\begin{align*}
    \left\vert \left\langle \sy^{n,M}_{(\gamma_n+\delta)\wedge T} - \sy^{n,M}_{\gamma_n} , f \right\rangle_U\right\vert &\leq c\int_{\gamma_n}^{(\gamma_n+\delta)\wedge T} \left[K_U(\check{\sy}^n_r) + \norm{\check{\sy}^n_r}_H^{\frac{3}{2}} \right] \left[K_U(f) + \norm{f}_H^{\frac{3}{2}} \right] dr \\ &+ \left\vert\int_{\gamma_n}^{(\gamma_n+\delta) \wedge T}\inner{\mathcal{G} (r,\check{\sy}^n_r)}{\mathcal{P}_nf}_U\mathbbm{1}_{r \leq \tau^{M,T+1}_n} d\mathcal{W}_r\right\vert
\end{align*}
where (\ref{uniform bounds of projection}) has also been employed. We again use (\ref{galerkinboundsatisfiedbystoppingtime2}), and also allowing our constant $c$ to depend on $f$, we have that 
\begin{align} \nonumber
    \left\vert \left\langle \sy^{n,M}_{(\gamma_n+\delta)\wedge T} - \sy^{n,M}_{\gamma_n} , f \right\rangle_U\right\vert &\leq c\int_{\gamma_n}^{(\gamma_n+\delta)\wedge T} 1 + \norm{\check{\sy}^n_r}_H^{\frac{3}{2}} dr \\ &+ \left\vert\int_{\gamma_n}^{(\gamma_n+\delta) \wedge T}\inner{\mathcal{G} (r,\check{\sy}^n_r)}{\mathcal{P}_nf}_U\mathbbm{1}_{r \leq \tau^{M,T+1}_n} d\mathcal{W}_r\right\vert \label{to be replaced}
\end{align}
still $\mathbbm{P}-a.s.$. Now taking expectation and applying the Burkholder-Davis-Gundy Inequality to the stochastic integral, and then using the assumption (\ref{tightnessassumpt2}),
\begin{align*}
    \mathbbm{E}\left(\left\vert\int_{\gamma_n}^{(\gamma_n+\delta) \wedge T}\inner{\mathcal{G} (r,\check{\sy}^n_r)}{\mathcal{P}_nf}_U\mathbbm{1}_{r \leq \tau^{M,T+1}_n} d\mathcal{W}_r\right\vert \right) &\leq c\mathbbm{E}\left(\int_{\gamma_n}^{(\gamma_n+\delta) \wedge T}\inner{\mathcal{G}_i (r,\check{\sy}^n_r)}{\mathcal{P}_nf}_U^2\mathbbm{1}_{r \leq \tau^{M,T+1}_n} dr\right)^\frac{1}{2}\\
    &\leq c\mathbbm{E}\left(\int_{\gamma_n}^{(\gamma_n+\delta) \wedge T}K_U(\check{\sy}^n_r)K_H(f)dr\right)^\frac{1}{2}\\
    &\leq c\mathbbm{E}\left(\int_{\gamma_n}^{(\gamma_n+\delta) \wedge T}1dr\right)^\frac{1}{2}\\
    &\leq c\delta^{\frac{1}{2}}.
\end{align*}
Putting all of this back into (\ref{to be replaced}), and reducing the constant integral to a $\delta$ as just seen,
$$   \mathbbm{E}\left(\left\vert \left\langle \sy^{n,M}_{(\gamma_n+\delta)\wedge T} - \sy^{n,M}_{\gamma_n} , f \right\rangle_U\right\vert\right) \leq c\mathbbm{E}\int_{\gamma_n}^{(\gamma_n+\delta)\wedge T} \norm{\check{\sy}^n_r}_H^{\frac{3}{2}} dr + c\delta + c\delta^{\frac{1}{2}}.$$
The remaining integral is treated exactly as was seen in (\ref{usage}), from which taking the supremum over $n$ and limit as $\delta \rightarrow 0^+$ gives the result.

\end{proof}

\subsection{Existence of Martingale Weak Solutions} \label{subbie passage to limit marty weak}

With tightness achieved, it is now a standard procedure to apply the Prohorov and Skorohod Representation Theorems to deduce the existence of a new probability space on which a sequence of processes with the same distribution as a subsequence of $(\sy^{n})$ have some almost sure convergence to a limiting process. For notational simplicity we take this subsequence and keep it simply indexed by $n$. We state the precise result in the below theorem, following e.g. [\cite{nguyen2021nonlinear}] Proposition 4.9 and Theorem 4.10.

\begin{theorem} \label{theorem for new prob space}
There exists a filtered probability space $\left(\tilde{\Omega},\tilde{\mathcal{F}},(\tilde{\mathcal{F}}_t), \tilde{\mathbbm{P}}\right)$, a cylindrical Brownian Motion $\tilde{\mathcal{W}}$ over $\mathfrak{U}$ with respect to $\left(\tilde{\Omega},\tilde{\mathcal{F}},(\tilde{\mathcal{F}}_t), \tilde{\mathbbm{P}}\right)$, a sequence of random variables $(\tilde{\sy}^n_0)$, $\tilde{\sy}^n_0: \tilde{\Omega} \rightarrow U$ and a $\tilde{\sy}_0:\tilde{\Omega} \rightarrow U$, a sequence of processes $(\tilde{\sy}^n)$, $\tilde{\sy}^n:\tilde{\Omega} \times [0,T] \rightarrow H$ is progressively measurable and a process $\tilde{\sy}:\tilde{\Omega} \times [0,T] \rightarrow U$ such that:
\begin{enumerate}
    \item For each $n \in \N$, $\tilde{\sy}^n_0$ has the same law as $\sy^{n}_0$;
    \item \label{new item 2} For $\tilde{\mathbbm{P}}-a.e.$ $\omega$, $\tilde{\sy}^n_0(\omega) \rightarrow \tilde{\sy}_0(\omega)$  in $U$, and thus $\tilde{\sy}_0$ has the same law as $\sy_0$;
    \item \label{new item 3} For each $n \in \N$ and $t\in[0,T]$, $\tilde{\sy}^n$ satisfies the identity
    \begin{equation} \nonumber
    \tilde{\sy}^n_t = \tilde{\sy}^n_0 + \int_0^t \mathcal{P}_n \mathcal{A}(s, \tilde{\sy}^n_s) ds + \int_0^t \mathcal{P}_n\mathcal{G}(s,\tilde{\sy}^n_s) d\tilde{\mathcal{W}}_s 
\end{equation}
$\tilde{\mathbbm{P}}-a.s.$ in $V_n$;
\item For $\tilde{\mathbbm{P}}-a.e$ $\omega$, $\tilde{\sy}^n(\omega) \rightarrow \tilde{\sy}(\omega)$ in $L^2\left([0,T]; U \right)$ and $\mathcal{D}\left([0,T];H^* \right)$. \label{new item 4}
\end{enumerate}

\end{theorem}

We now have our candidate martingale weak solution, and to prove that this is such a solution we need only to verify that $\tilde{\sy}$ is progressively measurable in $H$, for $\tilde{\mathbbm{P}}-a.e.$ $\omega$ $\tilde{\sy}_{\cdot}(\omega) \in L^{\infty}\left([0,T];U\right)\cap C_w\left([0,T];U\right) \cap L^2\left([0,T];H\right)$ and the identity (\ref{newid1}). In fact from item \ref{new item 3}, we can deduce that \begin{equation}\label{second result new }\tilde{\mathbbm{E}}\left(\norm{\tilde{\sy}^n}_{UH,T}^2  \right) \leq C\left[\mathbbm{E}\left(\norm{\tilde{\sy}^n_0}_U^2\right) + 1\right] \leq C\left[\norm{\tilde{\sy}_0}^2_{L^\infty(\tilde{\Omega};U)} + 1\right] < \infty\end{equation}
in the same manner as we showed (\ref{second result}), without any need for localisation. The fact that $\norm{\tilde{\sy}^n_0} \leq \norm{\tilde{\sy}_0}$ $\tilde{\mathbbm{P}}-a.s.$ and $\norm{\tilde{\sy}_0}^2_{L^\infty(\tilde{\Omega};U)} < \infty$ is inherited from $\sy^n_0$, $\sy_0$ of the same law in $U$. This prompts the following results.

\begin{lemma} \label{the lemma just before}
    $\tilde{\sy}^n \rightarrow \tilde{\sy}$ in $L^2\left(\tilde{\Omega};L^2\left([0,T]; U \right)\right)$.
\end{lemma}

\begin{proof}
    This is immediate from an application of the Dominated Convergence Theorem, using the convergence in item \ref{new item 4} and the uniform boundedness (\ref{second result new }).
\end{proof}

\begin{proposition} \label{prop for regularity of limit}
    $\tilde{\sy}$ is progressively measurable in $H$ and for $\tilde{\mathbbm{P}}-a.e.$ $\omega$, $\tilde{\sy}_{\cdot}(\omega) \in L^{\infty}\left([0,T];U\right)\cap L^2\left([0,T];H\right)$.
\end{proposition}

\begin{proof}
    From (\ref{second result new }) we have that the sequence $(\tilde{\sy}^n)$ is uniformly bounded in\\ $L^2\left(\tilde{\Omega};L^2\left([0,T];H\right)\right)$ and $L^2\left(\tilde{\Omega};L^\infty\left([0,T];U\right)\right)$. Firstly then we can deduce the existence of a subsequence $(\tilde{\sy}^{n_k})$ which is weakly convergent in the Hilbert Space $L^2\left(\tilde{\Omega};L^2\left([0,T];H\right)\right)$ to some $\py^1$, but we may also identify $L^2\left(\tilde{\Omega};L^\infty\left([0,T];U\right)\right)$ with the dual space of $L^2\left(\tilde{\Omega};L^1\left([0,T];U\right)\right)$ and as such from the Banach-Alaoglu Theorem we can extract a further subsequence $(\tilde{\sy}^{n_l})$ which is convergent to some $\py^2$ in the weak* topology. These limits imply that $(\tilde{\sy}^{n_l})$ is convergent to both $\py^1$ and $\py^2$ in the weak topology of $L^2\left(\tilde{\Omega};L^2\left([0,T];U\right)\right)$, but from Lemma \ref{the lemma just before} then $(\tilde{\sy}^{n_l})$ converges to $\tilde{\sy}$ strongly (hence weakly) in this space as well. By uniqueness of limits in the weak topology then $\tilde{\sy} = \py^1 = \py^2$ as elements of $L^2\left(\tilde{\Omega};L^2\left([0,T];U\right)\right)$, so they agree $\tilde{\mathbbm{P}} \times \lambda-a.s.$. Thus for $\tilde{\mathbbm{P}}-a.e.$ $\omega$, $\tilde{\sy}_{\cdot}(\omega) \in L^{\infty}\left([0,T];U\right)\cap L^2\left([0,T];H\right)$.\\

    The progressive measurability is justified similarly; for any $t\in [0,T]$, we can use the progressive measurability of $(\tilde{\sy}^{n_k})$ to instead deduce $\tilde{\sy}$ as the weak limit in $L^2\left(\tilde{\Omega} \times [0,t];H\right)$ where $\tilde{\Omega} \times [0,t]$ is equipped with the $\tilde{\mathcal{F}}_t \times \mathcal{B}\left([0,t]\right)$ sigma-algebra. Therefore $\tilde{\sy}: \tilde{\Omega} \times [0,t] \rightarrow H$ is measurable with respect to this product sigma-algebra which justifies the progressive measurability. 
\end{proof}

\begin{proposition} \label{prop for first energy bar}
    $\tilde{\sy}$ satisfies the identity (\ref{newid1}). Moreover for $\tilde{\mathbbm{P}}-a.e.$ $\omega$, $\tilde{\sy}_{\cdot}(\omega) \in C_w\left([0,T];U\right)$.
\end{proposition}

\begin{proof}
As a consequence of the Hahn-Banach Theorem we know that the dual space separates points, from which it is immediate that $\tilde{\sy}$ satisfies the identity (\ref{newid1}) if and only if for every $f \in H$, $\tilde{\sy}$ satisfies 
$$\left\langle \tilde{\sy}_t, f \right\rangle_{H^* \times H} = \left\langle\tilde{\sy}_0, f\right\rangle_{H^* \times H} + \left\langle\int_0^t \mathcal{A}(s, \tilde{\sy}_s) ds , f\right\rangle_{H^* \times H}+ \left\langle\int_0^t \mathcal{G}(s,\tilde{\sy}_s) d\tilde{\mathcal{W}}_s, f\right\rangle_{H^* \times H}$$
$\tilde{\mathbbm{P}}-a.s.$ in $\R$ for all $t \in [0,T]$\footnote{We need not concern ourselves with whether the set of full probability depends on $f$, just as was true for the dependence on $t$, by separability of $H$: see Remark \ref{markydoo}. }. In fact, as the system $(a_k)$ forms a basis of $H$, then it is sufficient to show the identity for any $\psi \in \bigcup_nV_n$ replacing $f$. Using the continuity and linearity of the duality pairing, it is thus sufficient to show the new identity
\begin{equation} \label{newestestid} \left\langle \tilde{\sy}_t, \psi \right\rangle_{H^* \times H} = \left\langle\tilde{\sy}_0, \psi\right\rangle_{H^* \times H} + \int_0^t \left\langle\mathcal{A}(s, \tilde{\sy}_s), \psi \right\rangle_{H^* \times H}ds + \int_0^t \left\langle\mathcal{G}(s,\tilde{\sy}_s) , \psi\right\rangle_{H^* \times H}d\tilde{\mathcal{W}}_s\end{equation}
for any $\psi \in \bigcup_nV_n$, $\mathbbm{P}-a.s.$ for all $t \in [0,T]$. Of course $\tilde{\sy}^n$ satisfies a corresponding identity
$$\left\langle \tilde{\sy}^n_t, \psi \right\rangle_{H^* \times H} = \left\langle\tilde{\sy}^n_0, \psi\right\rangle_{H^* \times H} + \int_0^t \left\langle\mathcal{P}_n\mathcal{A}(s, \tilde{\sy}^n_s), \psi \right\rangle_{H^* \times H}ds + \int_0^t \left\langle \mathcal{P}_n\mathcal{G}(s,\tilde{\sy}^n_s) , \psi\right\rangle_{H^* \times H}d\tilde{\mathcal{W}}_s$$
so we look to show convergence in each of the corresponding terms. Fixing any such $\psi$ and $t$, we appreciate that the required identity (\ref{newestestid}) would follow from only showing either $\mathbbm{P}-a.s.$ or $L^p\left(\tilde{\Omega};\R\right)$ convergence of a subsequence for the corresponding terms, as a $\mathbbm{P}-a.s.$ convergent subsequence can be extracted from the $L^p\left(\tilde{\Omega};\R\right)$ limit. Henceforth, we are concerned with the limits
\begin{align}
    \label{limity1}\left\langle \tilde{\sy}^n_t, \psi \right\rangle_{H^* \times H} &\longrightarrow \left\langle \tilde{\sy}_t, \psi \right\rangle_{H^* \times H}\\ \label{limity2}
    \left\langle\tilde{\sy}^n_0, \psi\right\rangle_{H^* \times H} &\longrightarrow \left\langle\tilde{\sy}_0, \psi\right\rangle_{H^* \times H}\\ \label{limity3}
    \int_0^t \left\langle\mathcal{P}_n\mathcal{A}(s, \tilde{\sy}^n_s), \psi \right\rangle_{H^* \times H}ds &\longrightarrow \int_0^t \left\langle\mathcal{A}(s, \tilde{\sy}_s), \psi \right\rangle_{H^* \times H}ds\\ \label{limity4}
    \int_0^t \left\langle \mathcal{P}_n\mathcal{G}(s,\tilde{\sy}^n_s) , \psi\right\rangle_{H^* \times H}d\tilde{\mathcal{W}}_s &\longrightarrow \int_0^t \left\langle \mathcal{G}(s,\tilde{\sy}_s) , \psi\right\rangle_{H^* \times H}d\tilde{\mathcal{W}}_s
\end{align}
of a subsequence, $\tilde{\mathbbm{P}}-a.s.$ or in $L^p\left(\tilde{\Omega};\R\right)$. The first convergence, (\ref{limity1}), holds $\tilde{\mathbbm{P}}-a.s.$ as convergence in the Skorohod Topology implies convergence at each $t$ (see e.g. [\cite{billingsley2013convergence}] pp.124), whilst the same convergence is true of (\ref{limity2}) as $\tilde{\sy}^n_0 \rightarrow \tilde{\sy}_0$ in $U$, from the fact that the system $(a_k)$ forms an orthogonal basis of $U$. We look to show (\ref{limity3}) in $L^1\left(\tilde{\Omega};\R\right)$, and do so by first observing that as $\psi \in \bigcup_nV_n$ then it is certainly in $V_k$ for some $k$, so without loss of generality in the limit as $n \rightarrow \infty$ we can assume that $n > k$. Moreover, 
$$\left\langle\mathcal{P}_n\mathcal{A}(s, \tilde{\sy}^n_s), \psi \right\rangle_{H^* \times H} = \left\langle\mathcal{P}_n\mathcal{A}(s, \tilde{\sy}^n_s), \psi \right\rangle_{U} = \left\langle\mathcal{A}(s, \tilde{\sy}^n_s), \psi \right\rangle_{U} = \left\langle\mathcal{A}(s, \tilde{\sy}^n_s), \psi \right\rangle_{H^* \times H}$$
by construction of the Gelfand Triple and that $\mathcal{P}_n$ is self-adjoint in $U$ and is the identity on $V_k \subset V_n$. Revisiting (\ref{limity3}),
\begin{align*}
    \left\vert\left\langle\mathcal{A}(s, \tilde{\sy}^n_s), \psi \right\rangle_{H^* \times H} - \left\langle\mathcal{A}(s, \tilde{\sy}_s), \psi \right\rangle_{H^* \times H}\right\vert &= \left\vert\left\langle \mathcal{A}(s, \tilde{\sy}^n_s) - \mathcal{A}(s, \tilde{\sy}_s), \psi \right\rangle_{H^* \times H}\right\vert \\ &\leq c_tK_V(\psi)\left[1 + \norm{\tilde{\sy}^n_s}_H + \norm{\tilde{\sy}_s}_H \right]\norm{\tilde{\sy}^n_s - \tilde{\sy}_s}_U
\end{align*}
where we have invoked Assumption \ref{limity assumptions}, (\ref{limityassumpt1}). Therefore, incorporating the $K_V(\psi)$ into the constant,
\begin{align*}
    &\tilde{\mathbbm{E}}\left\vert \int_0^t\left\langle\mathcal{A}(s, \tilde{\sy}^n_s), \psi \right\rangle_{H^* \times H} - \left\langle\mathcal{A}(s, \tilde{\sy}_s), \psi \right\rangle_{H^* \times H} ds\right\vert \\ & \qquad \qquad \qquad \qquad \qquad \leq c\tilde{\mathbbm{E}}\int_0^t\left[1 + \norm{\tilde{\sy}^n_s}_H + \norm{\tilde{\sy}_s}_H \right]\norm{\tilde{\sy}^n_s - \tilde{\sy}_s}_U ds\\ 
    & \qquad \qquad \qquad \qquad \qquad  \leq c\left( \tilde{\mathbbm{E}}\int_0^t 1 + \norm{\tilde{\sy}^n_s}_H^2 + \norm{\tilde{\sy}_s}_H^2 ds\right)^{\frac{1}{2}}\left( \tilde{\mathbbm{E}}\int_0^t \norm{\tilde{\sy}^n_s - \tilde{\sy}_s}_U^2 ds\right)^{\frac{1}{2}}
\end{align*}
where we have used H\"{o}lder's Inequality over the product integral, recalling that the progressive measurability ensures no issues in applying Fubinelli-Tonelli for the order of integration. This approaches zero, owing to both the uniform boundedness of the first expectation due to (\ref{second result new }) and Proposition \ref{prop for regularity of limit}, as well as the convergence in Lemma \ref{the lemma just before} for the second expectation. The convergence (\ref{limity3}) is justified, to which we move on to (\ref{limity4}), to be shown in $L^2\left(\tilde{\Omega};\R\right)$. The duality pairing is simply the $U$ inner product in this case, and by again passing over the $\mathcal{P}_n$, we can apply the Burkholder-Davis-Gundy Inequality to obtain
\begin{align*}\tilde{\mathbbm{E}}\left\vert \int_0^t\left\langle\mathcal{G}(s, \tilde{\sy}^n_s), \psi \right\rangle_{U} - \left\langle\mathcal{G}_i(s, \tilde{\sy}_s), \psi \right\rangle_{U} d\tilde{\mathcal{W}}_s\right\vert^2
&\leq \tilde{\mathbbm{E}}\int_0^t \sum_{i=1}^\infty  \left\langle\mathcal{G}_i(s, \tilde{\sy}^n_s) -\mathcal{G}(s, \tilde{\sy}_s), \psi \right\rangle_{U}^2 ds\\
&\leq c\tilde{\mathbbm{E}}\int_0^t \norm{\tilde{\sy}^n_s - \tilde{\sy}^s}_U^2ds
\end{align*}
having applied the assumed (\ref{limityassumpt2}) and contained the $K_V(\psi)$ in our constant. This again approaches zero due to Lemma \ref{the lemma just before}, concluding the justification that $\tilde{\sy}$ satisfies (\ref{newid1}). It now only remains to determine that for $\tilde{\mathbbm{P}}-a.e.$ $\omega$, $\tilde{\sy}_{\cdot}(\omega) \in C_w\left([0,T];U\right).$ By the identity just shown it is clear that $\inner{\tilde{\sy}_{\cdot}(\omega)}{f} \in C\left([0,T];\R\right)$ where $f \in H$ was arbitrary, but to conclude the weak continuity we must instead show this for any $\eta \in U$. Furthermore we fix such an $\omega$ and $\eta \in U$, any $t \in [0,T]$ and sequence of times $(t_k)$ in $[0,T]$ such that $t_k \rightarrow t$. To demonstrate the continuity let's fix $\varepsilon >0$, and choose an $f \in H$ such that $$\norm{\eta - f}_U < \frac{\varepsilon}{4}\sup_{s\in[0,T]}\norm{\tilde{\sy}_s(\omega)}_U$$ where the right hand side is of course finite from Proposition \ref{prop for regularity of limit}. Note that there exists a $K \in \N$ such that for all $k \geq K$, $$\left\vert\left\langle\tilde{\sy}_{t_k}(\omega) -\tilde{\sy}_{t}(\omega),f\right\rangle_U\right\vert < \frac{\varepsilon}{2}.$$
Then for all $k \geq K$ we have that
\begin{align*}
   \left\vert\inner{\tilde{\sy}_{t_k}(\omega) -\tilde{\sy}_{t}(\omega)}{\eta}_U\right\vert &\leq \left\vert\inner{\tilde{\sy}_{t_k}(\omega) -\tilde{\sy}_{t}(\omega)}{\eta - f}_U\right\vert + \left\vert\inner{\tilde{\sy}_{t_k}(\omega) -\tilde{\sy}_{t}(\omega)}{f}_U\right\vert\\
   &<2 \sup_{s\in[0,T]}\norm{\tilde{\sy}_s(\omega)}_U\norm{\eta - f}_U + \frac{\varepsilon}{2}\\
   &< \varepsilon
\end{align*}
demonstrating the weak continuity and finishing the proof.
\end{proof}

\section{Weak Solutions} \label{seection weak}

We first introduce the necessary extension of Assumption Set 1 for the main result of this section.

\subsection{Assumption Set 2} \label{assumption set 2}

 Recall the setup and notation of Subsection \ref{subs functional framework}. We assume that there exists a $c_{\cdot}$, $K$ and $\gamma > 0$ such that for all $f,g \in H$ and $t \in [0,T]$:

  \begin{assumption} \label{new assumptionzizzle 1} \begin{align}
     \label{111zizzle} \norm{\mathcal{A}(t,f)}_{H^*}^2  &\leq c_t K_U(f)\left[1 + \norm{f}_H^2\right].
 \end{align}
 \end{assumption}

\begin{assumption} \label{therealcauchy assumptions}
\begin{align}
  \nonumber 2\inner{\mathcal{A}(t,f) - \mathcal{A}(t,g)}{f - g}_{H^* \times H} &+ \sum_{i=1}^\infty\norm{\mathcal{G}_i(t,f) - \mathcal{G}_i(t,g)}_U^2\\ \label{therealcauchy1} &\leq  c_{t}K_U(f,g)\left[1 + \norm{f}_H^2 + \norm{g}_H^2\right]\norm{f-g}_U^2 - \gamma\norm{f-g}_H^2,\\ \label{therealcauchy2}
    \sum_{i=1}^\infty \inner{\mathcal{G}_i(t,f) - \mathcal{G}_i(t,g)}{f-g}^2_U & \leq c_{t} K_U(f,g)\left[1 + \norm{f}_H^2 + \norm{g}_H^2\right] \norm{f-g}_U^4.
\end{align}
\end{assumption}

We briefly comment on the purpose of each assumption:

\begin{itemize}
    \item Assumption \ref{new assumptionzizzle 1} ensures sufficient regularity in the drift to apply Proposition \ref{rockner prop} (Lemma \ref{new regularity for martingale weak}).
    \item Assumption \ref{therealcauchy assumptions} is used to show uniqueness of weak solutions (Proposition \ref{first uniqueness prop}). 
\end{itemize}

\subsection{Definitions and Results} \label{subbie weak def and res}

We now state the definitions and main result for weak solutions.

\begin{definition} \label{definitionofweak}
Let $\sy_0: \Omega \rightarrow U$ be $\mathcal{F}_0-$measurable. A process $\sy$ which is progressively measurable in $H$ and such that for $\mathbb{P}-a.e.$ $\omega$, $\sy_{\cdot}(\omega) \in C\left([0,T];U\right) \cap L^2\left([0,T];H\right)$, is said to be a weak solution of the equation (\ref{thespde}) if the identity (\ref{thespde}) holds $\mathbb{P}-a.s.$ in $H^*$ for all $t\in[0,T]$.
\end{definition}

\begin{definition} \label{definitionunique}
    A weak solution $\sy$ of the equation (\ref{thespde}) is said to be the unique solution if for any other such solution $\py$, $$ \mathbb{P}\left(\left\{\omega \in \Omega: \sy_t(\omega) = \py_t(\omega) \quad \forall t \geq 0\right\}\right) = 1.$$
\end{definition}

\begin{theorem} \label{theorem for weak existence}
    Let Assumption Sets 1 and 2 hold. For any given $\mathcal{F}_0-$measurable $\sy_0: \Omega \rightarrow U$, there exists a unique weak solution of the equation (\ref{thespde}). 
\end{theorem}

The rest of this section is devoted to the proof of Theorem \ref{theorem for weak existence}.

\subsection{Existence and Uniqueness of Weak Solutions} \label{exis uniqu}

We continue to work with $\sy_0 \in L^\infty\left(\Omega;U\right)$, and take $\tilde{\sy}$ to be a martingale weak solution of the equation (\ref{thespde}) as given in Theorem \ref{theorem for martingale weak existence}. Our approach to uniqueness is to consider the energy of the difference of two solutions, by applying Proposition \ref{rockner prop}. For this, improved regularity is needed. 

\begin{lemma} \label{new regularity for martingale weak}
    For $\tilde{\mathbbm{P}}-a.e.$ $\omega$, $\mathcal{A}\left(\cdot, \tilde{\sy}_{\cdot}(\omega)\right) \in L^2\left([0,T];H^*\right)$. Moreover, $\tilde{\sy}_{\cdot}(\omega) \in C\left([0,T];U\right)$.
\end{lemma}

\begin{proof}
    This additional regularity comes from Assumption \ref{new assumptionzizzle 1}, (\ref{111zizzle}), as
    \begin{align*}
        \int_0^T\norm{\mathcal{A}\left(s, \tilde{\sy}_{s}(\omega)\right)}_{H^*}^2ds &\leq c\int_0^TK\left(\tilde{\sy}_{s}(\omega)\right)\left[1 +\norm{\tilde{\sy}_{s}(\omega)}_H^2\right]ds\\ &\leq c\sup_{r\in[0,T]}K\left(\tilde{\sy}_{r}(\omega)\right)\int_0^T\left[1 +\norm{\tilde{\sy}_{s}(\omega)}_H^2\right]ds\\ &< \infty
    \end{align*}
    using that $\tilde{\sy}_{\cdot}(\omega) \in L^\infty\left([0,T];U\right) \cap L^2\left([0,T];H\right)$ as a martingale weak solution. The continuity now follows as an application of Proposition \ref{rockner prop}.  
\end{proof}

\begin{proposition} \label{first uniqueness prop}
    Suppose that $\tilde{\py}$ is another martingale weak solution of (\ref{thespde}) with respect to the same filtered probability space $\left(\tilde{\Omega},\tilde{\mathcal{F}},(\tilde{\mathcal{F}}_t), \tilde{\mathbbm{P}}\right)$, cylindrical Brownian Motion $\tilde{\mathcal{W}}$ and initial condition $\tilde{\py}_0 = \tilde{\sy}_0$ $\tilde{\mathbbm{P}}-a.s.$. In addition assume that for $\tilde{\mathbbm{P}}-a.e.$ $\omega$, $\tilde{\py}_{\cdot}(\omega) \in C\left([0,T];U\right)$. Then
    $$ \tilde{\mathbbm{P}}\left(\left\{\omega \in \tilde{\Omega}: \tilde{\sy}_t(\omega) = \tilde{\py}_t(\omega) \quad \forall t\in[0,T]\right\}\right) = 1.$$
\end{proposition}

    \begin{proof}
    We make our argument by considering the expectation of the difference of the solutions $\tilde{\sy},\tilde{\py}$, and to do so we need to manufacture an increased regularity through stopping times once more. To this end let's define the stopping times $(\alpha_R)$ by $$\alpha_R := T \wedge \inf \left\{r \geq 0: \norm{\tilde{\sy}}_{UH,r}^2 \geq R \right\} \wedge \inf \left\{r \geq 0: \norm{\tilde{\py}}_{UH,r}^2 \geq R \right\}$$ and subsequent processes
    $$\tilde{\sy}^R_{\cdot}:=\tilde{\sy}_{\cdot}\mathbbm{1}_{\cdot \leq \alpha_R}, \qquad \tilde{\py}^R_{\cdot}:=\tilde{\py}_{\cdot}\mathbbm{1}_{\cdot \leq \alpha_R}, \qquad \boldsymbol{\Pi} = \tilde{\sy}^R - \tilde{\py}^R.$$ Moreover the difference process satisfies
\begin{align*}
    \tilde{\sy}_{t\wedge \alpha_R} - \tilde{\py}_{t\wedge \alpha_R} = \int_0^t\mathbbm{1}_{s \leq \alpha_R} \left[\mathcal{A}\left(s,\tilde{\sy}^R_s\right) - \mathcal{A}\left(s,\tilde{\py}^R_s\right)\right] ds  + \int_0^t\mathbbm{1}_{s \leq \alpha_R} \left[\mathcal{G}\left(s,\tilde{\sy}^R_s\right) - \mathcal{G}\left(s,\tilde{\py}^R_s\right) \right]d\tilde{\mathcal{W}}_s
\end{align*}
and we can apply the Energy Equality of Proposition \ref{rockner prop} to see that
\begin{align*}
    \norm{\tilde{\sy}_{t\wedge \alpha_R} - \tilde{\py}_{t\wedge \alpha_R}}^2_U &= 2\int_0^t\mathbbm{1}_{s \leq \alpha_R}\inner{\mathcal{A}\left(s,\tilde{\sy}^R_s\right) - \mathcal{A}\left(s,\tilde{\py}^R_s\right)}{\boldsymbol{\Pi}_s}_{H^* \times H} ds\\ &+\int_0^t\mathbbm{1}_{s \leq \alpha_R}\sum_{i=1}^\infty\norm{\mathcal{G}_i\left(s,\tilde{\sy}^R_s\right) - \mathcal{G}_i\left(s,\tilde{\py}^R_s\right)}_U^2ds\\ &+ 2\int_0^t \inner{\mathcal{G}\left(s,\tilde{\sy}^R_s\right) - \mathcal{G}\left(s,\tilde{\py}^R_s\right)}{\boldsymbol{\Pi}_s}_{U}d\tilde{\mathcal{W}}_s.
\end{align*}
Motivated by the use of Lemma \ref{gronny}, we consider arbitrary stopping times $0 \leq \theta_j \leq \theta_k \leq T$ and substitute $\theta_j$ into the above, then subtract this from the identity for any $\theta_j \leq r \leq T$, to give that 
\begin{align*}
    \norm{\tilde{\sy}_{t\wedge \alpha_R} - \tilde{\py}_{t\wedge \alpha_R}}^2_U &= \norm{\tilde{\sy}_{\theta_j\wedge \alpha_R} - \tilde{\py}_{\theta_j\wedge \alpha_R}}^2_U + 2\int_{\theta_j}^t\mathbbm{1}_{s \leq \alpha_R}\inner{\mathcal{A}\left(s,\tilde{\sy}^R_s\right) - \mathcal{A}\left(s,\tilde{\py}^R_s\right)}{\boldsymbol{\Pi}_s}_{H^* \times H} ds\\ &+\int_{\theta_j}^t\mathbbm{1}_{s \leq \alpha_R}\sum_{i=1}^\infty\norm{\mathcal{G}_i\left(s,\tilde{\sy}^R_s\right) - \mathcal{G}_i\left(s,\tilde{\py}^R_s\right)}_U^2ds\\ &+ 2\int_{\theta_j}^t \inner{\mathcal{G}\left(s,\tilde{\sy}^R_s\right) - \mathcal{G}\left(s,\tilde{\py}^R_s\right)}{\boldsymbol{\Pi}_s}_{U}d\tilde{\mathcal{W}}_s.
\end{align*}
We now employ Assumption \ref{therealcauchy assumptions}, (\ref{therealcauchy1}), and combine a few steps into one. Firstly we use the boundedness coming from the stopping time $\alpha_R$ to control the $K_U$ term in (\ref{therealcauchy1}), much like we saw from (\ref{galerkinboundsatisfiedbystoppingtime2}); it should be noted that without the continuity of the processes in $U$, we could not guarantee such a control. Secondly, we simply ignore the helpful $\gamma$ term. All of this leads to the new inequality
\begin{align*}
    \norm{\tilde{\sy}_{r\wedge \alpha_R} - \tilde{\py}_{r\wedge \alpha_R}}^2_U &\leq \norm{\tilde{\sy}_{\theta_j\wedge \alpha_R} - \tilde{\py}_{\theta_j\wedge \alpha_R}}^2_U + c\int_{\theta_j}^r\left(1 + \norm{\tilde{\sy}^R_s}_H^2 + \norm{\tilde{\py}^R_s}_H^2\right)\norm{\boldsymbol{\Pi}_s}_U^2ds\\ &+ 2\int_{\theta_j}^t \inner{\mathcal{G}\left(s,\tilde{\sy}^R_s\right) - \mathcal{G}\left(s,\tilde{\py}^R_s\right)}{\boldsymbol{\Pi}_s}_{U}d\tilde{\mathcal{W}}_s.
\end{align*}
We now take the absolute value on the right hand side, followed by the supremum over $r \in [\theta_j,\theta_k]$, then the expectation and immediately apply the Burkholder-Davis-Gundy Inequality to achieve that
\begin{align*}
    \tilde{\mathbbm{E}}\left(\sup_{r\in[\theta_j,\theta_k]}\norm{\tilde{\sy}_{r\wedge \alpha_R} - \tilde{\py}_{r\wedge \alpha_R}}_U^2\right)& \leq \tilde{\mathbbm{E}}\left(\norm{\tilde{\sy}_{\theta_j\wedge \alpha_R} - \tilde{\py}_{\theta_j\wedge \alpha_R}}_U^2\right)\\ & + c\tilde{\mathbbm{E}}\int_{\theta_j}^{\theta_k}\left(1 + \norm{\tilde{\sy}^R_s}_H^2 + \norm{\tilde{\py}^R_s}_H^2\right)\norm{\boldsymbol{\Pi}_s}_U^2ds\\ &+ c\tilde{\mathbbm{E}}\left(\int_{\theta_j}^{\theta_k} \sum_{i=1}^\infty \inner{\mathcal{G}_i\left(s,\tilde{\sy}^R_s\right) - \mathcal{G}_i\left(s,\tilde{\py}^R_s\right) }{\boldsymbol{\Pi}_s}^2_Uds\right)^{\frac{1}{2}}.
\end{align*}
We now use the assumption (\ref{therealcauchy2}) and follow the same process as (\ref{the same process}) to obtain that
\begin{align*}
    &c\tilde{\mathbbm{E}}\left(\int_{\theta_j}^{\theta_k} \sum_{i=1}^\infty \inner{\mathcal{G}_i\left(s,\tilde{\sy}^R_s\right) - \mathcal{G}_i\left(s,\tilde{\py}^R_s\right) }{\boldsymbol{\Pi}_s}^2_Uds\right)^{\frac{1}{2}}\\ & \qquad \qquad \qquad \leq c\tilde{\mathbbm{E}}\left(\int_{\theta_j}^{\theta_k} \left(1 + \norm{\tilde{\sy}^R_s}_H^2 +\norm{\tilde{\py}^R_s}_H^2\right)\norm{\boldsymbol{\Pi}_s}_U^4 ds\right)^{\frac{1}{2}}\\&\qquad \qquad \qquad \leq \frac{1}{2}\tilde{\mathbbm{E}}\left(\sup_{r\in[\theta_j,\theta_k]}\norm{\boldsymbol{\Pi}_r}^2_U\right) + c\tilde{\mathbbm{E}}\int_{\theta_j}^{\theta_k} \left(1 + \norm{\tilde{\sy}^R_s}_H^2 +\norm{\tilde{\py}^R_s}_H^2\right)\norm{\boldsymbol{\Pi}_s}^2_Uds.
\end{align*}
We now use that $\norm{\boldsymbol{\Pi}_r}_U^2 \leq \norm{\tilde{\sy}_{r\wedge \alpha_R} - \tilde{\py}_{r\wedge \alpha_R}}_U^2$ and rearrange to give
\begin{align*}
    \tilde{\mathbbm{E}}\left(\sup_{r\in[\theta_j,\theta_k]}\norm{\tilde{\sy}_{r\wedge \alpha_R} - \tilde{\py}_{r\wedge \alpha_R}}_U^2\right) &\leq 2\tilde{\mathbbm{E}}\left(\norm{\tilde{\sy}_{\theta_j\wedge \alpha_R} - \tilde{\py}_{\theta_j\wedge \alpha_R}}_U^2\right)\\ &+ c\tilde{\mathbbm{E}}\int_{\theta_j}^{\theta_k}\left(1 + \norm{\tilde{\sy}^R_s}_H^2 + \norm{\tilde{\py}^R_s}_H^2\right)\norm{\tilde{\sy}_{s\wedge \alpha_R} - \tilde{\py}_{s\wedge \alpha_R}}_U^2ds.
\end{align*}
We can now apply Lemma \ref{gronny} to deduce that $$\tilde{\mathbbm{E}}\left(\sup_{r\in[0,T]}\norm{\tilde{\sy}_{r\wedge \alpha_R} - \tilde{\py}_{r\wedge \alpha_R}}^2_U\right) = 0$$ as of course $\tilde{\sy}_0 = \tilde{\py}_0$ $\tilde{\mathbbm{P}}-a.s.$. We note that $\left(\sup_{r\in[0,T]}\norm{\tilde{\sy}_{r\wedge \alpha_R} - \tilde{\py}_{r\wedge \alpha_R}}^2\right)$ is a monotone increasing sequence in $R$, hence we take the limit as $R \rightarrow \infty$ and apply the Monotone Convergence Theorem to obtain $$\tilde{\mathbbm{E}}\left(\sup_{r\in[0,T]}\norm{\tilde{\sy}_{r} - \tilde{\py}_{r}}^2_U\right) = 0$$ which gives the result.
\end{proof}

It is now immediate that Theorem \ref{theorem for weak existence} holds in this case of the bounded initial condition.

\begin{corollary}
    There exists a unique weak solution $\sy$ of the equation (\ref{thespde}).
\end{corollary}

\begin{proof}
    This follows from a classical Yamada-Watanabe type result, proven rigorously in this setting in [\cite{rockner2008yamada}].
\end{proof}

To prove Theorem \ref{theorem for weak existence} it thus only remains to extend the result to an arbitrary $\mathcal{F}_0-$measurable $\sy_0: \Omega \rightarrow U$, which we now fix.

\begin{proof}[Proof of Theorem \ref{theorem for weak existence}:]
    We first show the existence of such a solution. The idea is as in [\cite{goodair2023existence}] Theorem 3.40 where we use the fact that for each $k \in \N \cup \{0\}$ there exists a weak solution $\sy^k$ of the equation (\ref{thespde}) for the bounded initial condition $\sy_0\mathbbm{1}_{k \leq \norm{\sy_0}_U < k+1}$. We argue that the process $\sy$ defined by $$\sy_t(\omega):= \sum_{k=1}^\infty \sy^k_t(\omega)\mathbbm{1}_{k \leq \norm{\sy_0(\omega)}_U < k+1}$$ is a weak solution. Appreciating that the infinite sum is merely formal and that for each $\omega$, $\sy(\omega):=\sy^k(\omega)$ for some $k$, then clearly $\sy$ inherits the pathwise regularity of the weak solutions $(\sy^k)$. As for the required identity, we introduce the more compact notation $$A_k:=\left\{\omega \in \Omega: k \leq \norm{\sy_0(\omega)}_U < k+1 \right\}$$ and as the $(A_k)$ partition $\Omega$, it is sufficient to show that 
\begin{align} \nonumber
     \mathbbm{1}_{A_k}\sy_t = \mathbbm{1}_{A_k}\sy_0  + \mathbbm{1}_{A_k}\int_0^{t}\mathcal{A}\left(s,\sy_s\right)ds  + \mathbbm{1}_{A_k}\int_0^{t} \mathcal{G}\left(s,\sy_s\right) d\mathcal{W}_s
\end{align}
$\mathbbm{P}-a.s.$ for all $t \in [0,T]$, or equivalently
  \begin{align} \nonumber
    \mathbbm{1}_{A_k}\sy^k_t = \mathbbm{1}_{A_k}\sy^k_0  + \mathbbm{1}_{A_k}\int_0^{t}\mathcal{A}\left(s,\sy^k_s\right)ds  + \mathbbm{1}_{A_k}\int_0^{t} \mathcal{G}\left(s,\sy_s\right) d\mathcal{W}_s.
\end{align}
We have to be more precise for the stochastic integral as we cannot simply take any random function through the integral, however $A_k$ is $\mathcal{F}_0-$measurable so it is justified in this case (see e.g. [\cite{goodair2022stochastic}] Proposition 1.6.14). Hence \begin{align*}
    \mathbbm{1}_{A_k}\int_0^{t}  \mathcal{G}\left(s,\sy_s\right) d\mathcal{W}_s = \int_0^{t} \mathbbm{1}_{A_k} \mathcal{G}\left(s,\sy_s\right) d\mathcal{W}_s = \int_0^{t} \mathbbm{1}_{A_k} \mathcal{G}\left(s,\sy^k_s\right) d\mathcal{W}_s =  \mathbbm{1}_{A_k}\int_0^{t} \mathcal{G}\left(s,\sy^k_s\right)d\mathcal{W}_s
\end{align*} 
so the identity that must be shown is
\begin{align} \nonumber
    \mathbbm{1}_{A_k}\sy^k_t = \mathbbm{1}_{A_k}\sy^k_0  + \mathbbm{1}_{A_k}\int_0^{t}\mathcal{A}\left(s,\sy^k_s\right)ds  + \mathbbm{1}_{A_k}\int_0^{t} \mathcal{G}\left(s,\sy_s^k\right) d\mathcal{W}_s.
\end{align}
This is granted from $\sy^k$ being a weak solution for the initial condition $\sy_0\mathbbm{1}_{A_k}$. To conclude the existence we only need to verify the progressive measurability, for which we understand $\sy$ as the pointwise almost everywhere limit of the sequence $\left(\sum_{k=1}^n \sy^k\mathbbm{1}_{k \leq \norm{\sy_0}_U < k+1} \right)$ over the product space $\Omega \times [0,t]$ equipped with the product sigma algebra $\mathcal{F}_t \times \mathcal{B}([0,t])$ in $H$. Each $\sy^k$ is progressively measurable hence so too is $\sy^k\mathbbm{1}_{k \leq \norm{\sy_0}_U < k+1}$ (as this indicator function is $\mathcal{F}_0-$measurable), thus measurable with respect to $\mathcal{F}_t \times \mathcal{B}([0,t])$, and the pointwise almost everywhere limit preserves the measurability which provides the result. This concludes the proof that $\sy$ is a weak solution of (\ref{thespde}), and one can show it is the unique such solution identically to Proposition \ref{first uniqueness prop}.
\end{proof}

\section{Strong Solutions} \label{section strong}

We first introduce the necessary extension of Assumption Sets 1 and 2 for the main result of this section.


\subsection{Assumption Set 3} \label{assumption set 3}

Recall the setup and notation of Subsection \ref{subs functional framework}. We now impose the existence of a new Banach Space $\bar{H}$ which is an extension of $H$, or precisely, $H \subseteq \bar{H} \subseteq U$ and for every $f \in \bar{H}$, $\norm{f}_{\bar{H}} = \norm{f}_H.$ In addition, $\mathcal{G}:[0,T] \times V \rightarrow \mathscr{L}^2\left(\mathfrak{U};\bar{H}\right)$ is assumed measurable. We also suppose that there exists a real valued sequence $(\mu_n)$ with $\mu_n \rightarrow \infty$ such that for any $f \in \bar{H}$, \begin{align}
     \label{mu2}
    \norm{(I - \mathcal{P}_n)f}_U \leq \frac{1}{\mu_n}\norm{f}_{\bar{H}}
\end{align}
where $I$ represents the identity operator in $U$. Furthermore we assume that there exists a $\gamma > 0$ such that for any $\varepsilon > 0$, there exists a $c_\cdot$, $K$ (dependent on $\varepsilon$) such that for any $\phi \in V$, $\phi^n \in V_n$ and $t \in [0,T]$:

\begin{assumption} \label{newfacilitator}
\begin{align}
 \label{111fac} \norm{\mathcal{A}(t,\phi)}_{U}^2 + \sum_{i=1}^\infty\norm{\mathcal{G}_i(t,\phi)}_{\bar{H}}^2 \leq c_t K_U(\phi)\left[1 + \norm{\phi}_H^4 + \norm{\phi}_V^2\right]
    \end{align}
\end{assumption}

\begin{assumption} \label{assumptions for uniform bounds2}
 \begin{align}
   \label{uniformboundsassumpt1}  2\inner{\mathcal{P}_n\mathcal{A}(t,\phi^n)}{\phi^n}_H + \sum_{i=1}^\infty\norm{\mathcal{P}_n\mathcal{G}_i(t,\phi^n)}_H^2 &\leq c_tK_U(\phi^n)\left[1 + \norm{\phi^n}_H^4\right] - \gamma\norm{\phi^n}_V^2,\\  \label{uniformboundsassumpt2}
    \sum_{i=1}^\infty \inner{\mathcal{P}_n\mathcal{G}_i(t,\phi^n)}{\phi^n}^2_H &\leq c_tK_U(\phi^n)\left[1 + \norm{\phi^n}_H^6\right] + \varepsilon \norm{\phi^n}_V^2.
\end{align}
\end{assumption}

We briefly comment on the purpose of each assumption:
\begin{itemize}
    \item The property (\ref{mu2}) is used to show the Cauchy result (Lemma \ref{lemma to use cauchy}).
    \item Assumption \ref{newfacilitator} ensures that the integrals in the definition of a strong solution (Definition \ref{definitionofstrong}) are well-defined. It also plays a role in the Cauchy result (Lemma \ref{lemma to use cauchy}).
    \item Assumption \ref{assumptions for uniform bounds2} is employed to demonstrate higher uniform estimates for the Galerkin System (Proposition \ref{theorem:uniformbounds}).
\end{itemize}

\subsection{Definitions and Results} \label{def and results strong sols}

We now state the definitions and main result for strong solutions.

\begin{definition} \label{definitionofstrong}
Let $\sy_0: \Omega \rightarrow H$ be $\mathcal{F}_0-$measurable. A process $\sy$ which is progressively measurable in $V$ and such that for $\mathbb{P}-a.e.$ $\omega$, $\sy_{\cdot}(\omega) \in L^{\infty}\left([0,T];H\right) \cap L^2\left([0,T];V\right)$, is said to be a strong solution of the equation (\ref{thespde}) if the identity (\ref{thespde}) holds $\mathbb{P}-a.s.$ in $U$ for all $t\in[0,T]$.
\end{definition}

Note that a strong solution necessarily has continuous paths in $U$, from the evolution equation satisfied in this space.

\begin{definition} \label{definitionunique2}
    A strong solution $\sy$ of the equation (\ref{thespde}) is said to be unique if for any other such solution $\py$, $$ \mathbb{P}\left(\left\{\omega \in \Omega: \sy_t(\omega) = \py_t(\omega) \quad \forall t \geq 0\right\}\right) = 1.$$
\end{definition}

\begin{theorem}\label{theorem for strong existence}
    Let Assumption Sets 1, 2 and 3 hold. For any given $\mathcal{F}_0-$measurable $\sy_0:\Omega \rightarrow H$, there exists a unique strong solution of the equation (\ref{thespde}). 
\end{theorem}

Our approach to proving this comes from the following lemma.

\begin{lemma} \label{lemma for regular weak is strong}
    Let $\py_0 :\Omega \rightarrow H$ be $\mathcal{F}_0-$measurable and suppose that $\py$ is a unique weak solution of the equation (\ref{thespde}) such that $\py$ is progressively measurable in $V$, and for $\mathbb{P}-a.e.$ $\omega$, $\py_{\cdot}(\omega) \in L^{\infty}\left([0,T];H\right) \cap L^2\left([0,T];V\right)$. Then $\py$ is a unique strong solution of the equation (\ref{thespde}).
\end{lemma}

\begin{proof}
    We first address the existence. It only needs to be justified that $\py$ satisfies the required identity (\ref{thespde}) in $U$. By definition of the weak solution, this identity is satisfied in $H^*$ and in fact the only term failing to belong to $U$ is $\int_0^t\mathcal{A}(s,\py_s)ds$, $\mathbbm{P}-a.s.$. However, the additional regularity in $\py$ and Assumption \ref{newfacilitator} ensure that $\mathcal{A}(\cdot,\py_{\cdot})$ belongs to $L^2\left([0,T];U\right)$ $\mathbbm{P}-a.s.$; see, for example, [\cite{goodair2022stochastic}] Subsection 2.2, which justifies that the terms are well defined. This concludes the fact that $\py$ is a strong solution. For uniqueness, suppose that $\sy$ is another strong solution. Then in particular $\sy$ is a weak solution, from which uniqueness of weak solutions gives the result.
\end{proof}

The rest of this section is dedicated to the proof of Theorem \ref{theorem for strong existence}.

\subsection{An Improved Uniform Estimate for the Galerkin System} \label{subbie improved uniform estimate}

We first assume that \begin{equation}\label{first initial condition bound in H}
    \sy_0 \in L^\infty\left(\Omega;H\right)
\end{equation}
similarly to (\ref{boundedinitialcondition}). In order to match the regularity of a strong solution, better estimates are required for the Galerkin System introduced in Subsection \ref{subbiegalerkin}. We recall the notation (\ref{tauMtn}), (\ref{truncation notation}), for the unique strong solution $\sy^n$ of the equation (\ref{nthorderGalerkin}), as well as introducing the corresponding notation to (\ref{new norm}) for the higher norm: for a function $\py \in C([0,t];H) \cap L^2([0,t];V)$ we define the norm \begin{align} \label{new norm2}
    \norm{\py}^2_{HV,t}:= \sup_{r \in [0,t]}\norm{\py_{r}}^2_H + \int_0^t\norm{\py_{r}}^2_Vdr
\end{align}
making explicit the dependence on the time $t$. 

\begin{proposition} \label{theorem:uniformbounds}
There exists a constant $C$, dependent on $M$ but independent of $n$, such that \begin{equation} \label{firstresultofuniformbounds}
    \mathbbm{E}\norm{\sy^n}_{HV,\tau^{M,T+1}_{n}}^2\leq C\left[\mathbbm{E}\left(\norm{\sy^n_{0}}_H^2\right) + 1\right].
\end{equation}

\end{proposition}

\begin{proof}
By equipping $V_n$ with the $H$ inner product, exactly as we saw for (\ref{nownonwnumb}), we obtain the identity
\begin{align}\nonumber\norm{\sy^{n,M}_{r}}_H^2 = \norm{\sy^n_0}_H^2 &+ 2\int_0^{r}\inner{\mathcal{P}_n\mathcal{A}(s,\sy^n_s)}{\sy^n_s}_H\mathbbm{1}_{s \leq \tau^{M,T+1}_n}ds  + \int_0^{r}\sum_{i=1}^\infty\norm{\mathcal{P}_n\mathcal{G}_i(s,\sy^n_s)}_H^2\mathbbm{1}_{s \leq \tau^{M,T+1}_n}ds\\ & + 2\sum_{i=1}^\infty\int_0^{r}\inner{\mathcal{P}_n\mathcal{G}_i(s,\check{\sy}^n_s)}{\check{\sy}^n_s}_HdW^i_s.\nonumber
\end{align}
$\mathbbm{P}-a.s.$ in $\R$ for all $r\in[0,T]$.
In the direction of Lemma \ref{gronny}, similarly to Proposition \ref{first uniqueness prop}, let $0 \leq \theta_j < \theta_k \leq t$ be two arbitrary stopping times. By substituting in $\theta_j$ to the above, and then subtracting this from the identity for any $\theta_j \leq r \leq t$ $\mathbbm{P}-a.s.$, then we also have the equality
\begin{align}\nonumber\norm{\sy^{n,M}_{r}}_H^2 = \norm{\sy^{n,M}_{\theta_j}}_H^2 &+ 2\int_{\theta_j}^{r}\left(\inner{\mathcal{P}_n\mathcal{A}(s,\sy^n_s)}{\sy^n_s}_H + \sum_{i=1}^\infty\norm{\mathcal{P}_n\mathcal{G}_i(s,\sy^n_s)}_H^2\right)\mathbbm{1}_{s \leq \tau^{M,T+1}_n}ds\\ &+ 2\sum_{i=1}^\infty\int_{\theta_j}^{r}\inner{\mathcal{P}_n\mathcal{G}_i(s,\check{\sy}^n_s)}{\check{\sy}^n_s}_HdW^i_s\nonumber
\end{align}
$\mathbbm{P}-a.s.$. Invoking Assumption \ref{assumptions for uniform bounds2}, (\ref{uniformboundsassumpt1}), and immediately applying (\ref{galerkinboundsatisfiedbystoppingtime2}), we deduce the inequality
\begin{align} \nonumber \norm{\sy^{n,M}_{r}}_H^2 \leq \norm{\sy^{n,M}_{\theta_j}}_H^2 &+ c\int_{\theta_j}^{r} 1 + \norm{\check{\sy}^n_s}_H^4 ds  -\gamma\int_{\theta_j}^r \norm{\check{\sy}^n_s}_V^2 ds\\  & \qquad \qquad \qquad \qquad + 2\sum_{i=1}^\infty\int_{\theta_j}^{r}\inner{\mathcal{P}_n\mathcal{G}_i(s,\check{\sy}^n_s)}{\check{\sy}^n_s}_HdW^i_s \label{continuity checky}
\end{align}
having now assimilated the indicator function through the norms into the $\check{\sy}^n$, and noting that $c$ again depends on $M$ (and the initial condition). We now take the absolute value of the stochastic integral, followed by the supremum over $r \in [\theta_j,\theta_k]$, then the expectation and immediately apply the Burkholder-Davis-Gundy Inequality to see that  
\begin{align*}\mathbbm{E}\left(\sup_{r\in [\theta_j,\theta_k]}\norm{\sy^{n,M}_{r}}_H^2\right) +  \gamma\mathbbm{E}\int_{\theta_j}^{\theta_k}\norm{\check{\sy}^n_s}_V^2ds &\leq 2\mathbbm{E}\left(\norm{\sy^{n,M}_{\theta_j}}_H^2\right) + c\mathbbm{E}\int_{\theta_j}^{\theta_k}1 + \norm{\check{\sy}^n_s}_H^4 ds\\ &+ c\mathbbm{E}\left(\int_{\theta_j}^{\theta_k}\sum_{i=1}^\infty\inner{\mathcal{P}_n\mathcal{G}_i(s,\check{\sy}^n_s)}{\check{\sy}^n_s}_H^2ds\right)^{\frac{1}{2}}.
\end{align*}
Now is the time to use the assumed (\ref{uniformboundsassumpt2}), which gives us that for any choice of $\varepsilon > 0$, 
\begin{align}\nonumber \mathbbm{E}\left(\sup_{r\in [\theta_j,\theta_k]}\norm{\sy^{n,M}_{r}}_H^2\right) &+  \gamma\mathbbm{E}\int_{\theta_j}^{\theta_k}\norm{\check{\sy}^n_s}_V^2ds \leq 2\mathbbm{E}\left(\norm{\sy^{n,M}_{\theta_j}}_H^2\right) + c\mathbbm{E}\int_{\theta_j}^{\theta_k}1 + \norm{\check{\sy}^n_s}_H^4 ds\\ &+ c_{\varepsilon}\mathbbm{E}\left(\int_{\theta_j}^{\theta_k}1 + \norm{\check{\sy}^n_s}_H^6 ds\right)^{\frac{1}{2}} + \varepsilon\mathbbm{E}\left(\int_{\theta_j}^{\theta_k} \norm{\check{\sy}^n_s}_V^2 ds\right)^{\frac{1}{2}} \label{first labelled aligny}
\end{align}
having used the simple property that $(a+b)^{1/2} \leq a^{1/2} + b^{1/2}$. We now apply Young's Inequality to see that 
\begin{equation}\label{youngers}\left(\int_{\theta_j}^{\theta_k} \norm{\check{\sy}^n_s}_V^2 ds\right)^{\frac{1}{2}} \leq \frac{1}{2}\left[\int_{\theta_j}^{\theta_k} \norm{\check{\sy}^n_s}_V^2 ds + 1\right]\end{equation} which can be plugged back into (\ref{first labelled aligny}) for the choice $\varepsilon:=\gamma$ to see that 
\begin{align}\nonumber \mathbbm{E}\left(\sup_{r\in [\theta_j,\theta_k]}\norm{\sy^{n,M}_{r}}_H^2\right) +  \frac{\gamma}{2}\mathbbm{E}\int_{\theta_j}^{\theta_k}\norm{\check{\sy}^n_s}_V^2ds &\leq 2\mathbbm{E}\left(\norm{\sy^{n,M}_{\theta_j}}_H^2\right) + c\mathbbm{E}\int_{\theta_j}^{\theta_k}1 + \norm{\check{\sy}^n_s}_H^4 ds\\ &+ c\mathbbm{E}\left(\int_{\theta_j}^{\theta_k}1 + \norm{\check{\sy}^n_s}_H^6 ds\right)^{\frac{1}{2}} \nonumber
\end{align}
where $c$ is dependent on $\gamma$, which is not meaningful. Taking out the constant integrals as a constant for simplicity, we reach the inequality 
\begin{align}\nonumber \mathbbm{E}\left(\sup_{r\in [\theta_j,\theta_k]}\norm{\sy^{n,M}_{r}}_H^2\right) +  \frac{\gamma}{2}\mathbbm{E}\int_{\theta_j}^{\theta_k}\norm{\check{\sy}^n_s}_V^2ds &\leq 2\mathbbm{E}\left(\norm{\sy^{n,M}_{\theta_j}}_H^2\right) + c + c\mathbbm{E}\int_{\theta_j}^{\theta_k}\norm{\check{\sy}^n_s}_H^4 ds\\ &+ c\mathbbm{E}\left(\int_{\theta_j}^{\theta_k} \norm{\check{\sy}^n_s}_H^2\norm{\check{\sy}^n_s}_H^4 ds\right)^{\frac{1}{2}}. \nonumber
\end{align}
We handle the last term just as we did in (\ref{the same process}), also using that $\norm{\check{\sy}^n_s}_H^2 \leq \norm{\sy^{n,M}_s}_H^2$, to obtain
\begin{equation}\label{prophecy}
    \mathbbm{E}\left(\sup_{r\in [\theta_j,\theta_k]}\norm{\sy^{n,M}_{r}}_H^2\right) +  \mathbbm{E}\int_{\theta_j}^{\theta_k}\norm{\check{\sy}^n_s}_V^2ds \leq c\mathbbm{E}\left(\norm{\sy^{n,M}_{\theta_j}}_H^2\right) + c + c\mathbbm{E}\int_{\theta_j}^{\theta_k}\norm{\check{\sy}^n_s}_H^4 ds.
\end{equation}
where we have also scaled $\gamma$ out of the inequality, through a depedence in $c$. Now we may apply the Stochastic Gr\"{o}nwall lemma, Lemma \ref{gronny}, for the processes $$\boldsymbol{\phi} = \norm{\check{\sy}^n}_H^2, \qquad \boldsymbol{\psi}= \norm{\check{\sy}^n}_V^2, \qquad \boldsymbol{\eta}= \norm{\check{\sy}^n}_H^2$$ noting that the bound (\ref{boundingronny}) owes to $\tau^{M,T+1}_n$, (\ref{tauMtn}). The application of this result conclude the proof.

\end{proof}

\subsection{The Cauchy Property} \label{subbie cauchy}

Towards an application of Proposition \ref{amazing cauchy lemma} for the Galerkin System and $\tau^{M,T}_n$ as in (\ref{tauMtn}) and used throughout the paper, we verify (\ref{supposition 1}) in this subsection.

\begin{lemma} \label{lemma to use cauchy}
For arbitrary $m<n$, $\phi^n \in V_n, \psi^m \in V_m$, define
\begin{align*}
    A &:= 2\inner{\mathcal{P}_n\mathcal{A}(s,\phi^n) - \mathcal{P}_m\mathcal{A}(s,\psi^m)}{\phi^n - \psi^m}_U + \sum_{i=1}^\infty\norm{\mathcal{P}_n\mathcal{G}_i(s,\phi^n) - \mathcal{P}_m\mathcal{G}_i(s,\psi^m)}_U^2\\
    B &:= \sum_{i=1}^\infty \inner{\mathcal{P}_n\mathcal{G}_i(s,\phi^n) - \mathcal{P}_m\mathcal{G}_i(s,\psi^m)}{\phi^n-\psi^m}^2_U.
\end{align*}
Then there exists a sequence $(\lambda_m)$ with $\lambda_m \rightarrow \infty$ such that
\begin{align}
\nonumber A &\leq c_{s}\left[K_U(\phi^n,\psi^m) + \norm{\phi^n}_H^2 + \norm{\psi^m}_H^2\right] \norm{\phi^n-\psi^m}_U^2 - \frac{\gamma}{2}\norm{\phi^n-\psi^m}_H^2\\ \label{bound of A} & \qquad \qquad \qquad \qquad \qquad + \frac{c_s}{\lambda_m}K_U(\phi^n,\psi^m)\left[1 + \norm{\phi^n}_H^4 +\norm{\psi^m}_H^4 +  \norm{\phi^n}_V^2 + \norm{\psi^m}_V^2\right] ;\\ \nonumber
  B &\leq c_{s} \left[K_U(\phi^n,\psi^m) + \norm{\phi^n}_H^2 + \norm{\psi^m}_H^2\right]\norm{\phi^n-\psi^m}_U^4\\ & \qquad \qquad \qquad \qquad \qquad + \frac{c_{s}}{\lambda_m} K_U(\phi^n,\psi^m)\left[1 + \norm{\phi^n}_H^4 + \norm{\psi^m}_H^4\right]. \label{happy achievement}
\end{align}
\end{lemma}

\begin{proof}
We show the inequalities independently, starting with $A$ through
\begin{align}
    \nonumber A &=  2\inner{\mathcal{P}_n\left[\mathcal{A}(s,\phi^n) - \mathcal{A}(s,\psi^m)
    \right] + \left[\mathcal{P}_n- \mathcal{P}_m\right]\mathcal{A}(s,\psi^m)}{\phi^n - \psi^m}_U \\\nonumber  & \qquad \qquad \qquad \qquad \qquad \qquad + \sum_{i=1}^\infty\norm{\mathcal{P}_n\left[\mathcal{G}_i(s,\phi^n) - \mathcal{G}_i(s,\psi^m)\right] + \left[\mathcal{P}_n- \mathcal{P}_m\right]\mathcal{G}_i(s,\psi^m)}_U^2\\\nonumber 
    & \leq 2\inner{\mathcal{P}_n\left[\mathcal{A}(s,\phi^n) - \mathcal{A}(s,\psi^m)
    \right] }{\phi^n - \psi^m}_U + \sum_{i=1}^\infty\norm{\mathcal{P}_n\left[\mathcal{G}_i(s,\phi^n) - \mathcal{G}_i(s,\psi^m)\right]}_U^2\\ \nonumber & \qquad \qquad \qquad + 2\inner{\left[\mathcal{P}_n- \mathcal{P}_m\right]\mathcal{A}(s,\psi^m)}{\phi^n - \psi^m}_U + \sum_{i=1}^\infty\norm{\left[\mathcal{P}_n- \mathcal{P}_m\right]\mathcal{G}_i(s,\psi^m)}_U^2\\\nonumber 
    & \qquad \qquad \qquad \qquad \qquad \qquad + 2\sum_{i=1}^\infty\norm{\mathcal{P}_n\left[\mathcal{G}_i(s,\phi^n) -\mathcal{G}_i(s,\psi^m)\right]}_U\norm{\left[\mathcal{P}_n- \mathcal{P}_m\right]\mathcal{G}_i(s,\psi^m)}_U\\\nonumber
    &\leq 2\inner{\mathcal{A}(s,\phi^n) - \mathcal{A}(s,\psi^m)
     }{\phi^n - \psi^m}_U + \sum_{i=1}^\infty\norm{\mathcal{G}_i(s,\phi^n) - \mathcal{G}_i(s,\psi^m)}_U^2\\ \nonumber& \qquad \qquad \qquad + 2\inner{\mathcal{A}(s,\psi^m)}{[I-\mathcal{P}_m]\phi^n - \psi^m}_U + \sum_{i=1}^\infty\norm{\mathcal{P}_n\left[I - \mathcal{P}_m \right]\mathcal{G}_i(s,\psi^m)}_U^2\\ \nonumber
    & \qquad \qquad \qquad \qquad \qquad \qquad + 2\sum_{i=1}^\infty\norm{\mathcal{G}_i(s,\phi^n) -\mathcal{G}_i(s,\psi^m)}_U\norm{\mathcal{P}_n\left[I - \mathcal{P}_m \right]\mathcal{G}_i(s,\psi^m)}_U
    \\
    &=:\alpha + \beta +\kappa \nonumber
\end{align}
having used that $\mathcal{P}_n$ is an orthogonal projection on $U$ and $\mathcal{P}_n\mathcal{P}_m = \mathcal{P}_m$.
We look to show the appropriate bounds on $\alpha$, $\beta$ and $\kappa$. For $\alpha$ we simply apply Assumption \ref{therealcauchy assumptions}, (\ref{therealcauchy1}). Moving on to $\beta$, we use again that $\mathcal{P}_n$ is an orthogonal projection and the property (\ref{mu2}) to see that $$\beta \leq \frac{2}{\mu_m}\norm{\mathcal{A}(s,\psi^m)}_U\norm{\phi^n - \psi^m}_H + \sum_{i=1}^\infty \frac{1}{\mu_m^2}\norm{\mathcal{G}_i(s,\psi^m)}_{\bar{H}}^2.$$ Through Young's Inequality with a constant $c$ dependent on $\gamma$, we can bound this further by $$\frac{c}{\mu_m^2}\left(\norm{\mathcal{A}(s,\psi^m)}_U^2 + \sum_{i=1}^\infty\norm{\mathcal{G}_i(s,\psi^m)}_{\bar{H}}^2 \right) + \frac{\gamma}{2}\norm{\phi^n - \psi^m}_H^2$$ to which we apply Assumption \ref{newfacilitator}, (\ref{111fac}), to the bracketed term. As for $\kappa$, we have that
\begin{align*}
   \kappa &\leq \sum_{i=1}^\infty\left(\norm{\mathcal{G}_i(s,\phi^n)}_U + \norm{\mathcal{G}_i(s,\psi^m)}_U\right)\norm{\left[I - \mathcal{P}_m \right]\mathcal{G}_i(s,\psi^m)}_U\\
   &\leq c_s\sum_{i=1}^\infty\left(\norm{\mathcal{G}_i(s,\phi^n)}_{\bar{H}} + \norm{\mathcal{G}_i(s,\psi^m)}_{\bar{H}}\right)\norm{\left[I - \mathcal{P}_m \right]\mathcal{G}_i(s,\psi^m)}_U\\
   &\leq \frac{c_s}{\mu_m}\sum_{i=1}^\infty\left(\norm{\mathcal{G}_i(s,\phi^n)}_{\bar{H}} + \norm{\mathcal{G}_i(s,\psi^m)}_{\bar{H}}\right)\norm{\mathcal{G}_i(s,\psi^m)}_{\bar{H}}\\
   &\leq \frac{c_s}{\mu_m}\sum_{i=1}^\infty\left(\norm{\mathcal{G}_i(s,\phi^n)}_{\bar{H}}^2 + \norm{\mathcal{G}_i(s,\psi^m)}_{\bar{H}}^2\right)
\end{align*}
which we handle through (\ref{111fac}) once more. Altogether then, with notation $\lambda_m:= \min\{\mu_m,\mu_m^2\}$, we have that  
\begin{align*}
    A &\leq c_{s}K_U(\phi^n,\psi^m)\left[1 + \norm{\phi^n}_H^2 + \norm{\psi^m}_H^2\right] \norm{\phi^n-\psi^m}_U^2 - \gamma\norm{\phi^n-\psi^m}_H^2\\ & \qquad + \frac{c_s}{\lambda_m} K_U(\psi^m)\left[1 + \norm{\psi^m}_H^4 + \norm{\psi^m}_V^2\right] + \frac{\gamma}{2} \norm{\phi^n - \psi^m}_H^2\\ & \qquad \qquad + \frac{c_s}{\lambda_m}\left(K_U(\phi^n)\left[1 + \norm{\phi^n}_H^4 + \norm{\phi^n}_V^2  \right] + c_s K_U(\psi^m)\left[1 + \norm{\psi^m}_H^4 + \norm{\psi^m}_V^2\right]\right)
\end{align*}
which compresses into (\ref{bound of A}). It remains to show the bound on $B$, which we approach in a similar manner:
\begin{align*}
    B &= \sum_{i=1}^\infty \inner{\mathcal{P}_n\left[\mathcal{G}_i(s,\phi^n) - \mathcal{G}_i(s,\psi^m)\right] + \left[\mathcal{P}_n- \mathcal{P}_m\right]\mathcal{G}_i(s,\psi^m)}{\phi^n-\psi^m}^2_U\\
    &\leq 2\sum_{i=1}^\infty\left(\inner{\mathcal{G}_i(s,\phi^n) - \mathcal{G}_i(s,\psi^m)}{\phi^n-\psi^m}^2_U + \inner{\left[I- \mathcal{P}_m\right]\mathcal{G}_i(s,\psi^m)}{\phi^n-\psi^m}^2_U\right).
\end{align*}
The first term here is precisely what we have in (\ref{therealcauchy2}). For the second we use that $I - \mathcal{P}_m$ is itself an orthogonal projection in $U$, as well as the assumed (\ref{111}) and (\ref{mu2}):
\begin{align*}
    \inner{\left[I- \mathcal{P}_m\right]\mathcal{G}_i(s,\psi^m)}{\phi^n-\psi^m}^2_U &= \inner{\mathcal{G}_i(s,\psi^m)}{\left[I- \mathcal{P}_m\right]\left(\phi^n-\psi^m\right)}^2_U\\
    &\leq \frac{c_s}{\lambda_m}K_U(\psi^m)\left[1 + \norm{\psi^m}_H^2\right]\norm{\phi^n-\psi^m}_H^2\\
    &\leq \frac{c_s}{\lambda_m}K_U(\psi^m)\left[1 + \norm{\phi^n}_H^4 + \norm{\psi^m}_H^4\right]
\end{align*}
which in total provides (\ref{happy achievement}). We note that this is a coarse bound, though sufficient for our purposes. 

\end{proof}

\begin{proposition} \label{theorem:cauchygalerkin}
For any $m,n \in \N$ with $m<n$, define the process $\sy^{m,n}$ by $$\sy^{m,n}_r(\omega) := \sy^n_r(\omega) - \sy^m_r(\omega).$$
Then for the sequence $(\lambda_j)$ proposed in Lemma \ref{lemma to use cauchy} and $m$ sufficiently large such that $\lambda_m > 1$, there exists a constant $C$ dependent on $M$ but independent of $m,n$ such that \begin{equation} \label{cauchy result pt 1}\mathbbm{E}\norm{\sy^{m,n}}_{UH,\tau^{M,T+1}_m \wedge \tau^{M,T+1}_n}^2  \leq C\left[\mathbbm{E}\norm{\sy^{m,n}_0}_U^2 + \frac{1}{\sqrt{\lambda_m}} \right] \end{equation} and in particular,
\begin{equation} \label{cauchy result pt 2}\lim_{m \rightarrow \infty}\sup_{n \geq m}\left[\mathbbm{E}\norm{\sy^{m,n}}_{UH,\tau^{M,T+1}_m \wedge \tau^{M,T+1}_n}^2\right] = 0.\end{equation}
\end{proposition}

\begin{proof}
The proof uses very similar methods to those used in Proposition \ref{theorem:uniformbounds}, this time relying on Lemma \ref{lemma to use cauchy} instead of Assumption \ref{assumptions for uniform bounds2}. For any $0 \leq r \leq t$, $\sy^{m,n}$ satisfies the identity
\begin{align*}
    \sy^{m,n}_{r \wedge \tau^{M,T+1}_m \wedge \tau^{M,T+1}_n} = \sy^{m,n}_0 & + \int_0^{r\wedge \tau^{M,T+1}_m \wedge \tau^{M,T+1}_n}\mathcal{P}_n\mathcal{A}(s,\sy^n_s) - \mathcal{P}_m\mathcal{A}(s,\sy^m_s) ds\\
    & +\sum_{i=1}^\infty \int_0^{r \wedge \tau^{M,T+1}_m \wedge \tau^{M,T+1}_n}\mathcal{P}_n\mathcal{G}_i(s,\sy^n_s) - \mathcal{P}_m\mathcal{G}_i(s,\sy^m_s) dW^i_s
\end{align*}
$\mathbbm{P}-a.s.$ in $V_n$. We thus apply Proposition \ref{rockner prop} to this difference process, reaching the equality
\begin{align*}
    &\norm{\sy^{m,n}_{r \wedge \tau^{M,T+1}_m \wedge \tau^{M,T+1}_n} }_U^2 = \norm{\sy^{m,n}_0}_U^2 + 2\int_0^{r\wedge \tau^{M,T+1}_m \wedge \tau^{M,T+1}_n} \inner{\mathcal{P}_n\mathcal{A}(s,\sy^n_s) - \mathcal{P}_m\mathcal{A}(s,\sy^m_s)}{\sy^{m,n}_s}_Uds\\ & \qquad \qquad +\int_0^{r\wedge \tau^{M,T+1}_m \wedge \tau^{M,T+1}_n}\sum_{i=1}^\infty\norm{\mathcal{P}_n\mathcal{G}_i(s,\sy^n_s) - \mathcal{P}_m\mathcal{G}_i(s,\sy^m_s)}_U^2ds\\  & \qquad \qquad \qquad \qquad + 2\sum_{i=1}^\infty \int_0^{r\wedge \tau^{M,T+1}_m \wedge \tau^{M,T+1}_n}\inner{\mathcal{P}_n\mathcal{G}_i(s,\sy^n_s) - \mathcal{P}_m\mathcal{G}_i(s,\sy^m_s)}{\sy^{m,n}_s}_UdW^i_s.
\end{align*}
We introduce notation similar to (\ref{truncation notation}), that is \begin{align*} &\hat{\sy}^m_\cdot:= \sy^m_{\cdot}\mathbbm{1}_{\cdot \leq \tau^{M,T+1}_m \wedge \tau^{M,T+1}_n} \qquad \hat{\sy}^n_\cdot:= \sy^n_{\cdot}\mathbbm{1}_{\cdot \leq \tau^{M,T+1}_m \wedge \tau^{M,T+1}_n}\\ &\hat{\sy}^{m,n}:= \hat{\sy}^n-\hat{\sy}^m \qquad \qquad \qquad \sy^{m,n,M}_{\cdot} = \sy^{m,n}_{\cdot \wedge \tau^{M,T+1}_m \wedge \tau^{M,T+1}_n}\end{align*}
and thus rewrite our equality as
\begin{align*}
        \norm{\sy^{m,n,M}_{r}}_U^2 &= \norm{\sy^{m,n}_0}_U^2 + 2\int_0^{r} \inner{\mathcal{P}_n\mathcal{A}(s,\sy^n_s) - \mathcal{P}_m\mathcal{A}(s,\sy^m_s)}{\sy^{m,n}_s}_U\mathbbm{1}_{s \leq \tau^{M,T+1}_m \wedge \tau^{M,T+1}_n}ds\\ & \qquad  \qquad \qquad  +\int_0^{r}\sum_{i=1}^\infty\norm{\mathcal{P}_n\mathcal{G}_i(s,\sy^n_s) - \mathcal{P}_m\mathcal{G}_i(s,\sy^m_s)}_U^2\mathbbm{1}_{s \leq \tau^{M,T+1}_m \wedge \tau^{M,T+1}_n}ds\\ & \qquad\qquad \qquad \qquad \qquad \qquad + 2\sum_{i=1}^\infty \int_0^{r}\inner{\mathcal{P}_n\mathcal{G}_i(s,\hat{\sy}^n_s) - \mathcal{P}_m\mathcal{G}_i(s,\hat{\sy}^m_s)}{\hat{\sy}^{m,n}_{s}}_UdW^i_s.
\end{align*}
Identically to Proposition \ref{theorem:uniformbounds}, we fix arbitrary stopping times $0 \leq \theta_j < \theta_k \leq t$ and have that for any $\theta_j \leq r \leq t$ $\mathbbm{P}-a.s.$, 
\begin{align*}
        \norm{\sy^{m,n,M}_{r}}_U^2 &= \norm{\sy^{m,n,M}_{\theta_j}}_U^2 + 2\int_{\theta_j}^{r} \inner{\mathcal{P}_n\mathcal{A}(s,\sy^n_s) - \mathcal{P}_m\mathcal{A}(s,\sy^m_s)}{\sy^{m,n}_s}_U\mathbbm{1}_{s \leq \tau^{M,T+1}_m \wedge \tau^{M,T+1}_n}ds\\ & \qquad  \qquad \qquad  +\int_{\theta_j}^{r}\sum_{i=1}^\infty\norm{\mathcal{P}_n\mathcal{G}_i(s,\sy^n_s) - \mathcal{P}_m\mathcal{G}_i(s,\sy^m_s)}_U^2\mathbbm{1}_{s \leq \tau^{M,T+1}_m \wedge \tau^{M,T+1}_n}ds\\ & \qquad\qquad \qquad \qquad \qquad \qquad + 2\sum_{i=1}^\infty \int_{\theta_j}^{r}\inner{\mathcal{P}_n\mathcal{G}_i(s,\hat{\sy}^n_s) - \mathcal{P}_m\mathcal{G}_i(s,\hat{\sy}^m_s)}{\hat{\sy}^{m,n}_{s}}_UdW^i_s.
\end{align*}
holds $\mathbbm{P}-a.s.$. Combining the time integrals and applying Lemma \ref{lemma to use cauchy}, (\ref{bound of A}), we deduce the inequality
\begin{align*}
        \norm{\sy^{m,n,M}_{r}}_U^2 &\leq \norm{\sy^{m,n,M}_{\theta_j}}_U^2 + c\int_{\theta_j}^{r}\left[
        1 + \norm{\hat{\sy}^m_s}_H^2 +\norm{\hat{\sy}^n_s}_H^2\right]\norm{\hat{\sy}^{m,n}_{s}}_U^2 - \frac{\gamma}{2}\norm{\hat{\sy}^{m,n}_{s}}_H^2ds\\  & \qquad \qquad + \frac{c}{\lambda_m}\int_{\theta_j}^r\left[1 + \norm{\hat{\sy}^m_s}_H^4 + \norm{\hat{\sy}^n_s}_H^4 + \norm{\hat{\sy}^m_s}_V^2 + \norm{\hat{\sy}^n_s}_V^2\right]ds \\& \qquad \qquad \qquad 
        + 2\sum_{i=1}^\infty \int_{\theta_j}^{r}\inner{\mathcal{P}_n\mathcal{G}_i(s,\hat{\sy}^n_s) - \mathcal{P}_m\mathcal{G}_i(s,\hat{\sy}^m_s)}{\hat{\sy}^{m,n}_{s}}_UdW^i_s
\end{align*}
again applying (\ref{galerkinboundsatisfiedbystoppingtime2}). All in one step we take the expectation, supremum over $r \in [\theta_j, \theta_k]$, apply the Burkholder-Davis-Gundy Inequality and employ Lemma \ref{lemma to use cauchy}, (\ref{happy achievement}), to deduce that
\begin{align*}
&\mathbbm{E}\left(\sup_{r\in[\theta_j,\theta_k]}\norm{\sy^{m,n,M}_{r}}_U^2\right) + \frac{\gamma}{2}\mathbbm{E}\int_{\theta_j}^{\theta_k}\norm{\hat{\sy}^{m,n}_{s}}_H^2\\ & \quad \leq 2\mathbbm{E}\left(\norm{\sy^{m,n,M}_{\theta_j}}_U^2\right) + c\mathbbm{E}\int_{\theta_j}^{\theta_k}\left[
        1 + \norm{\hat{\sy}^m_s}_H^2 +\norm{\hat{\sy}^n_s}_H^2\right]\norm{\hat{\sy}^{m,n}_{s}}_U^2 ds\\  & \qquad \qquad + \frac{c}{\lambda_m}\mathbbm{E}\int_{\theta_j}^{\theta_k} 1 + \norm{\hat{\sy}^m_s}_H^4 + \norm{\hat{\sy}^n_s}_H^4 + \norm{\hat{\sy}^m_s}_V^2 + \norm{\hat{\sy}^n_s}_V^2 ds \\& 
        + c\mathbbm{E}\left(\int_{\theta_j}^{\theta_k} \left[1 + \norm{\hat{\sy}^m_s}_H^2 + \norm{\hat{\sy}^n_s}_H^2 \right]\norm{\hat{\sy}^{m,n}_{s}}_U^4
 ds\right)^{\frac{1}{2}} + \frac{1}{\sqrt{\lambda_m}}\mathbbm{E}\left(\int_{\theta_j}^{\theta_k}1 + \norm{\hat{\sy}^m_s}_H^4 + \norm{\hat{\sy}^n_s}_H^4 ds\right)^{\frac{1}{2}}
\end{align*}
where we have used that $(a + b)^\frac{1}{2} \leq a^{\frac{1}{2}} + b^{\frac{1}{2}}$ in the last term. Applying Young's Inequality as we did for (\ref{youngers}), and using that $m$ is large enough so that $\lambda_m \geq 1$ hence $\frac{1}{\lambda_m} \leq \frac{1}{\sqrt{\lambda_m}}$, we have that
\begin{align}
\nonumber \mathbbm{E}\left(\sup_{r\in[\theta_j,\theta_k]}\norm{\sy^{m,n,M}_{r}}_U^2\right) &+ \frac{\gamma}{2}\mathbbm{E}\int_{\theta_j}^{\theta_k}\norm{\hat{\sy}^{m,n}_{s}}_H^2\\ \nonumber & \quad \leq 2\mathbbm{E}\left(\norm{\sy^{m,n,M}_{\theta_j}}_U^2\right) + c\mathbbm{E}\int_{\theta_j}^{\theta_k}\left[
        1 + \norm{\hat{\sy}^m_s}_H^2 +\norm{\hat{\sy}^n_s}_H^2\right]\norm{\hat{\sy}^{m,n}_{s}}_U^2 ds\\ \nonumber  &  \qquad + \frac{c}{\sqrt{\lambda_m}} + \frac{c}{\sqrt{\lambda_m}}\mathbbm{E}\int_{\theta_j}^{\theta_k}  \norm{\hat{\sy}^m_s}_H^4 + \norm{\hat{\sy}^n_s}_H^4 + \norm{\hat{\sy}^m_s}_V^2 + \norm{\hat{\sy}^n_s}_V^2 ds\\& \qquad \qquad  
        + c\mathbbm{E}\left(\int_{\theta_j}^{\theta_k} \left[1 + \norm{\hat{\sy}^m_s}_H^2 + \norm{\hat{\sy}^n_s}_H^2 \right]\norm{\hat{\sy}^{m,n}_{s}}_U^4
 ds\right)^{\frac{1}{2}}. \label{to be returned}
\end{align}
The property that \begin{equation} \label{of vital}
    \int_0^T\norm{\hat{\sy}^m_s}_H^2 + \norm{\hat{\sy}^n_s}_H^2ds \leq c,
\end{equation}
owing to the first hitting times and as used in Proposition \ref{theorem:uniformbounds}, is of vital importance here. We firstly control \begin{equation} \label{first control}
    \frac{c}{\sqrt{\lambda_m}}\mathbbm{E}\int_{\theta_j}^{\theta_k}  \norm{\hat{\sy}^m_s}_H^4 + \norm{\hat{\sy}^n_s}_H^4 + \norm{\hat{\sy}^m_s}_V^2 + \norm{\hat{\sy}^n_s}_V^2 ds
\end{equation}
where we note that $$\frac{c}{\sqrt{\lambda_m}}\mathbbm{E}\int_{\theta_j}^{\theta_k}\norm{\hat{\sy}_s^m}_H^4ds \leq \frac{c}{\sqrt{\lambda_m}}\mathbbm{E}\left(\sup_{r\in[0,T]}\norm{\hat{\sy}^m_r}_H^2\int_{\theta_j}^{\theta_k}\norm{\tilde{\sy}_s^m}_H^2ds\right) \leq \frac{c}{\sqrt{\lambda_m}}\mathbbm{E}\left(\sup_{r\in[0,T]}\norm{\hat{\sy}^m_r}_H^2\right)$$
where this expectation is bounded courtesy of Proposition \ref{theorem:uniformbounds}. Using the same proposition for the $\norm{\cdot}_V^2$ terms, we ultimately bound (\ref{first control}) by $\frac{c}{\sqrt{\lambda_m}}$. We estimate the final term in (\ref{to be returned}) as we did in (\ref{the same process}), which combined with the control on (\ref{first control}), reduces (\ref{to be returned}) to
\begin{align*}
&\mathbbm{E}\left(\sup_{r\in[\theta_j,\theta_k]}\norm{\sy^{m,n,M}_{r}}_U^2\right) + \mathbbm{E}\int_{\theta_j}^{\theta_k}\norm{\hat{\sy}^{m,n}_{s}}_H^2\\ \nonumber & \qquad \qquad \qquad \qquad  \leq c\mathbbm{E}\left(\norm{\sy^{m,n,M}_{\theta_j}}_U^2\right) + c\mathbbm{E}\int_{\theta_j}^{\theta_k}\left[
        1 + \norm{\hat{\sy}^m_s}_H^2 +\norm{\hat{\sy}^n_s}_H^2\right]\norm{\hat{\sy}^{m,n}_{s}}_U^2 ds + \frac{c}{\sqrt{\lambda_m}}
\end{align*}
with similar seen in Proposition \ref{theorem:uniformbounds}. An application of Lemma \ref{gronny} then provides that 
$$ \mathbbm{E}\left(\sup_{r\in[0,T]}\norm{\sy^{m,n,M}_{r}}_U^2\right) + \mathbbm{E}\int_{0}^{T}\norm{\hat{\sy}^{m,n}_{s}}_H^2 \leq C\left[\norm{\sy^{m,n}_0}_U^2 + \frac{1}{\sqrt{\lambda_m}}\right]$$
which gives (\ref{cauchy result pt 1}). The next result, (\ref{cauchy result pt 2}), follows from this as $$ \norm{\sy^{m,n}_0}_U^2 = \norm{\left(\mathcal{P}_n - \mathcal{P}_m\right)\sy_0}_U^2 = \norm{\mathcal{P}_n\left[\left(I - \mathcal{P}_m\right)\sy_0\right]}_U^2 \leq \norm{\left(I - \mathcal{P}_m\right)\sy_0}_U^2$$
and the fact that the system $(a_k)$ forms an orthogonal basis of $U$.

\end{proof}

\subsection{Weak Equicontinuity} \label{subbie equicontinuity}

In this subsection, the remaining condition of Proposition \ref{amazing cauchy lemma}, (\ref{supposition 2}), is verified. 

\begin{lemma} \label{probably unnecessary lemma}
    Let $\theta$ be a stopping time and $(\delta_j)$ a sequence of stopping times which converge to $0$ $\mathbb{P}-a.s.$. Then $$\lim_{j \rightarrow \infty}\sup_{n\in\N}\mathbbm{E}\left(\norm{\sy^n}_{UH,(\theta + \delta_j) \wedge \tau^{M,T}_n}^2 - \norm{\sy^n}_{UH,\theta \wedge \tau^{M,T}_n}^2 \right) =0.$$
\end{lemma}

\begin{proof}
    We look at the energy identity satisfied by $\sy^n$ up until the stopping time $\theta \wedge \tau^{M,T}_n$ and then $(\theta + r)  \wedge \tau^{M,T}_n$ for some $r \geq 0$. We have that
        \begin{align*}
        \norm{\sy^n_{\theta \wedge \tau^{M,T}_n}}_U^2 = \norm{\sy^n_0}_U^2 &+ 2\int_0^{\theta \wedge \tau^{M,T}_n}\inner{\mathcal{P}_n\mathcal{A}\left(s,\sy^n_s\right)}{\sy^n_s}_U ds  +\int_0^{\theta \wedge \tau^{M,T}_n}\sum_{i=1}^\infty \norm{\mathcal{P}_n\mathcal{G}_i\left(s,\sy^n_s\right)}_U^2ds\\ &+ 2\int_0^{\theta \wedge \tau^{M,T}_n}\inner{\mathcal{P}_n\mathcal{G}\left(s,\sy^n_s\right)}{\sy^n_s}_U d\mathcal{W}_s 
    \end{align*}
    and similarly for $(\theta + r)\wedge \tau^{M,T}_n$, from which the difference of the equalities gives
    \begin{align*}
        \norm{\sy^n_{(\theta + r)\wedge \tau^{M,T}_n}}_U^2 &= \norm{\sy^n_{\theta \wedge \tau^{M,T}_n}}_U^2 + 2\int_{\theta \wedge \tau^{M,T}_n}^{(\theta + r)\wedge \tau^{M,T}_n}\inner{\mathcal{P}_n\mathcal{A}\left(s,\sy^n_s\right)}{\sy^n_s}_U ds\\ & +\int_{\theta \wedge \tau^{M,T}_n}^{(\theta + r)\wedge \tau^{M,T}_n}\sum_{i=1}^\infty \norm{\mathcal{P}_n\mathcal{G}_i\left(s,\sy^n_s\right)}_U^2ds + 2\int_{\theta \wedge \tau^{M,T}_n}^{(\theta + r)\wedge \tau^{M,T}_n}\inner{\mathcal{P}_n\mathcal{G}\left(s,\sy^n_s\right)}{\sy^n_s}_U d\mathcal{W}_s.
    \end{align*}
Handling the projections as we did in (\ref{nownonwnumb}), and invoking (\ref{uniformboundsassumpt1actual}), we reduce to
     \begin{align*}
        \norm{\sy^n_{(\theta + r)\wedge \tau^{M,T}_n}}_U^2 &- \norm{\sy^n_{\theta \wedge \tau^{M,T}_n}}_U^2 + \gamma\int_{\theta \wedge \tau^{M,T}_n}^{(\theta + r)\wedge \tau^{M,T}_n} \norm{\sy^n_s}_H^2 ds\\ & \qquad \qquad \qquad \leq c\int_{\theta \wedge \tau^{M,T}_n}^{(\theta + r)\wedge \tau^{M,T}_n}1ds + \int_{\theta \wedge \tau^{M,T}_n}^{(\theta + r)\wedge \tau^{M,T}_n}\inner{\mathcal{G}\left(s,\sy^n_s\right)}{\sy^n_s}_U d\mathcal{W}_s. 
    \end{align*}
where the constant $c$ again depends on $M$, through (\ref{galerkinboundsatisfiedbystoppingtime2}). We must be somewhat careful before taking the supremum over $r \in [0,\delta_j]$ as the term $\norm{\sy^n_{(\theta + r)\wedge \tau^{M,T}_n}}_U^2 - \norm{\sy^n_{\theta \wedge \tau^{M,T}_n}}_U^2$ could be negative, so we cannot necessarily bound $\gamma\int_{\theta \wedge \tau^{M,T}_n}^{(\theta + r)\wedge \tau^{M,T}_n} \norm{\sy^n_s}_H^2 ds$ by the supremum of the right hand side. Therefore, we first treat this integral directly. We rely on both Proposition \ref{theorem:uniformbounds} and the control from $\tau^{M,T}_n$. Indeed,
\begin{align*}
    \mathbbm{E}\int_{\theta \wedge \tau^{M,T}_n}^{(\theta + \delta_j)\wedge \tau^{M,T}_n} \norm{\sy^n_s}_H^2 ds &\leq \mathbbm{E}\left(\sup_{r\in[0, \tau^{M,T}_n]}\norm{\sy^n_r}_H\int_{\theta \wedge \tau^{M,T}_n}^{(\theta + \delta_j)\wedge \tau^{M,T}_n} \norm{\sy^n_s}_H ds\right)\\
    &\leq \left[\mathbbm{E}\left(\sup_{r\in[0, \tau^{M,T}_n]}\norm{\sy^n_r}_H^2\right)\right]^{\frac{1}{2}}\left[\mathbbm{E}\left(\int_{\theta \wedge \tau^{M,T}_n}^{(\theta + \delta_j)\wedge \tau^{M,T}_n} \norm{\sy^n_s}_H ds\right)^2\right]^{\frac{1}{2}}\\
    &\leq c\left[\mathbbm{E}\left(\left(\int_{\theta \wedge \tau^{M,T}_n}^{(\theta + \delta_j)\wedge \tau^{M,T}_n} 1 + \norm{\sy^n_s}_H^2 ds\right)\left(\int_{\theta \wedge \tau^{M,T}_n}^{(\theta + \delta_j)\wedge \tau^{M,T}_n} \norm{\sy^n_s}_H ds\right)\right)\right]^{\frac{1}{2}}\\
    &\leq c\left[\mathbbm{E}\left(\int_{\theta \wedge \tau^{M,T}_n}^{(\theta + \delta_j)\wedge \tau^{M,T}_n} \norm{\sy^n_s}_H ds\right)\right]^{\frac{1}{2}}\\
    &\leq c\left[\left[\mathbbm{E}\left(\sup_{r\in[0, \tau^{M,T}_n]}\norm{\sy^n_r}_H^2\right)\right]^{\frac{1}{2}}\left[\mathbbm{E}\left(\int_{\theta \wedge \tau^{M,T}_n}^{(\theta + \delta_j)\wedge \tau^{M,T}_n}1 ds\right)^2\right]^{\frac{1}{2}}\right]^{\frac{1}{2}}\\
    &\leq c\left[\mathbbm{E}(\delta_j^2)\right]^{\frac{1}{4}}.
\end{align*}
Noting that $\delta_j$ is $\mathbbm{P}-a.s.$ monotone decreasing (as $j \rightarrow \infty$) and  convergent to $0$, the Monotone Convergence Theorem thus justifies that $$c\left[\mathbbm{E}(\delta_j^2)\right]^{\frac{1}{4}} = o_j$$
where $o_j$ represents a constant independent of $n$ which goes to zero as $\delta_j \rightarrow 0$. We can now employ this bound directly as we take expectation, the supremum over $r\in[0,\delta_j]$ and then use the Burkholder-Davis-Gundy Inequality,
\begin{align*}
        &\mathbbm{E}\left[\sup_{r \in [0,\delta_j]}\norm{\sy^n_{(\theta + r)\wedge \tau^{M,T}_n}}_U^2 - \norm{\sy^n_{\theta \wedge \tau^{M,T}_n}}_U^2 + \gamma\int_{\theta \wedge \tau^{M,T}_n}^{(\theta + \delta_j)  \wedge \tau^{M,T}_n} \norm{\sy^n_s}_H^2 ds\right]\\ & \qquad \qquad \qquad \leq o_{j} + c\mathbbm{E}\int_{\theta \wedge \tau^{M,T}_n}^{(\theta + \delta_j)\wedge \tau^{M,T}_n}1 ds + c\mathbbm{E}\left(\int_{\theta \wedge \tau^{M,T}_n}^{(\theta + \delta_j)\wedge \tau^{M,T}_n}\sum_{i=1}^\infty\inner{\mathcal{G}_i\left(s,\sy^n_s\right)}{\sy^n_s}^2_U ds\right)^{\frac{1}{2}}. 
    \end{align*}
We then use (\ref{uniformboundsassumpt2actual}) and again control the $U$ norm by a constant, achieving that
\begin{equation} \label{a la combo} \mathbbm{E}\left[\sup_{r \in [0,\delta_j]}\norm{\sy^n_{(\theta + r)\wedge \tau^{M,T}_n}}_U^2 - \norm{\sy^n_{\theta \wedge \tau^{M,T}_n}}_U^2 + \int_{\theta \wedge \tau^{M,T}_n}^{(\theta + \delta_j)  \wedge \tau^{M,T}_n} \norm{\sy^n_s}_H^2 ds\right] \leq o_j\end{equation}
having also scaled out the $\gamma$. We now need to relate the expression on the left hand side with what we are interested in, which is $\norm{\sy^n}_{UH,(\theta + \delta_j) \wedge \tau^{M,T}_n}^2 - \norm{\sy^n}_{UH,\theta \wedge \tau^{M,T}_n}^2$. We have that \begin{align*}\norm{\sy^n}_{UH,(\theta + \delta_j) \wedge \tau^{M,T}_n}^2 &- \norm{\sy^n}_{UH,\theta \wedge \tau^{M,T}_n}^2\\ &= \sup_{s \in [0,(\theta + \delta_j)\wedge \tau^{M,T}_n]}\norm{\sy^n_s}_U^2 -\sup_{s \in [0,\theta \wedge \tau^{M,T}_n]}\norm{\sy^n_s}_U^2 + \int_{\theta \wedge \tau^{M,T}_n}^{(\theta + \delta_j)  \wedge \tau^{M,T}_n} \norm{\sy^n_s}_H^2 ds\end{align*} and claim \begin{equation}\label{theclaim}\sup_{s \in [0,(\theta + \delta_j)\wedge \tau^{M,T}_n]}\norm{\sy^n_s}_U^2 -\sup_{s \in [0,\theta \wedge \tau^{M,T}_n]}\norm{\sy^n_s}_U^2 \leq \sup_{r \in [0,\delta_j]}\norm{\sy^n_{(\theta + r)\wedge \tau^{M,T}_n}}_U^2 - \norm{\sy^n_{\theta \wedge \tau^{M,T}_n}}_U^2.\end{equation}
Indeed, we have that $$ \sup_{s \in [0,(\theta + \delta_j)\wedge \tau^{M,T}_n]}\norm{\sy^n_s}_U^2 \leq \sup_{s \in [0,\theta \wedge \tau^{M,T}_n]}\norm{\sy^n_s}_U^2 + \sup_{s \in [\theta \wedge \tau^{M,T}_n,(\theta + \delta_j)\wedge \tau^{M,T}_n]}\norm{\sy^n_s}_U^2 - \norm{\sy^n_{\theta \wedge \tau^{M,T}_n}}_U^2$$ as the left hand side must equal either $\sup_{s \in [0,\theta \wedge \tau^{M,T}_n]}\norm{\sy^n_s}_U^2$ or $\sup_{s \in [\theta \wedge \tau^{M,T}_n,(\theta + \delta_j)\wedge \tau^{M,T}_n]}\norm{\sy^n_s}_U^2$, both of which are greater than the subtracted term $\norm{\sy^n_{\theta \wedge \tau^{M,T}_n}}_U^2$. Appreciating that $$\sup_{s \in [\theta \wedge \tau^{M,T}_n,(\theta + \delta_j)\wedge \tau^{M,T}_n]}\norm{\sy^n_s}_U^2 = \sup_{r \in [0,\delta_j]}\norm{\sy^n_{(\theta + r)\wedge \tau^{M,T}_n}}_U^2 $$ then yields the claim (\ref{theclaim}), which in combination with (\ref{a la combo}) grants that
\begin{align} \nonumber
        \mathbbm{E}\left[\norm{\sy^n}_{UH,(\theta + \delta_j) \wedge \tau^{M,T}_n}^2 - \norm{\sy^n}_{UH,\theta \wedge \tau^{M,T}_n}^2\right] \leq o_j.
    \end{align}
This proves the result.
   
\end{proof}

\subsection{Existence and Uniqueness of Strong Solutions} \label{Subbie passage to limit in strong}

We now have all of the ingredients needed to apply Proposition \ref{amazing cauchy lemma}. In this subsection, we fix $\sy$ to be the unique weak solution of the equation (\ref{thespde}) for the initial condition fixed in (\ref{first initial condition bound in H}), known to exist from Theorem \ref{theorem for weak existence}.

\begin{proposition} \label{the one that gives limiting stopping time bound} For any $R> 0$, there exists a subsequence indexed by $(m_j)$ and a constant $M>1$ such that the stopping time $$\tau^{R,T} := T \wedge \inf \left\{s \geq 0: \sup_{r \in [0,s]}\norm{\sy_r}_U^2 + \int_0^s \norm{\sy_r}_H^2dr \geq R \right\}$$ satisfies the property $\tau^{R,T} \leq \tau^{M,T}_{m_j}$ for all $m_j$ $\mathbb{P}-a.s.$.
\end{proposition}

\begin{proof}
    This comes from a direct application of Lemma \ref{amazing cauchy lemma}, for the Banach Space $X_t:=C\left([0,t];U\right) \cap L^2\left([0,t];H \right)$ equipped with norm $\norm{\cdot}_{UH,t}$. The conditions (\ref{supposition 1}) and (\ref{supposition 2}) are given in Proposition \ref{theorem:cauchygalerkin} and Lemma \ref{probably unnecessary lemma}. The only component in need of proof is that the limiting process $\py$ we achieve from Lemma \ref{amazing cauchy lemma} is in fact $\sy$. We are given that $(\sy^{m_j})$ converges to $\sy$ in $L^2\left([0,\tau^{M,T}_{\infty}];H\right)$ which is a stronger convergence (up to $\tau^{M,T}_{\infty}$) than what we had in Theorem \ref{theorem for new prob space}, which was used to show that the limit process was a weak solution of the equation (\ref{thespde}). Hence in the same manner as Proposition \ref{prop for first energy bar}, we see that $\py$ must be a $\textit{local}$ weak solution as well, up until the stopping time $\tau^{M,T}_{\infty}$, so from the uniqueness of weak solutions we must have that $\py$ is in fact equal to $\sy$ on $[0,\tau^{M,T}_{\infty}]$, which is sufficient for the result.
\end{proof}

This allows us to deduce the additional regularity required for $\sy$ to be a strong solution.

\begin{proposition} \label{final regularity of solutions}
    The process $\sy$ is progressively measurable in $V$ and such that for $\mathbb{P}-a.e.$ $\omega$, $\sy_{\cdot}(\omega) \in L^{\infty}\left([0,T];H\right) \cap L^2\left([0,T];V\right)$.
\end{proposition}
\begin{proof}
Propositions \ref{theorem:uniformbounds} and \ref{the one that gives limiting stopping time bound} together imply that, for any $R>0$, $$ \mathbbm{E}\left[\sup_{r \in [0,\tau^{R,T}]}\norm{\sy^{m_j}_r}_1^2 + \int_0^{ \tau^{R,T}}\norm{\sy^{m_j}_r}_2^2dr  \right] \leq C_{R}$$
where $C_{R}$ now incorporates $\norm{\sy_0}^2_{L^\infty(\Omega;H)}$ and shows the dependency on $M$ in terms of $R$. We also note use of (\ref{uniform bounds of projection}) in this deduction. In consequence the sequence of processes $(\sy^{m_j}_{\cdot}\mathbbm{1}_{\cdot \leq \tau^{R,T}})$ is uniformly bounded in both $L^2\left(\Omega; L^\infty\left([0,T];H\right) \right)$ and $L^2\left(\Omega; L^2\left([0,T];V\right) \right)$. We can deduce the existence of a subsequence which is weakly convergent in the Hilbert Space $L^2\left(\Omega;L^2\left([0,T];V\right)\right)$ to some $\check{\py}$, but we may also identify $L^2\left(\Omega;L^\infty\left([0,T];H\right)\right)$ with the dual space of $L^2\left(\Omega;L^1\left([0,T];H\right)\right)$ and as such from the Banach-Alaoglu Theorem we can extract a further subsequence which is convergent to some $\hat{\py}$ in the weak* topology. These limits imply convergence to both $\check{\py}$ and $\hat{\py}$ in the weak topology of $L^2\left(\Omega;L^2\left([0,T];H\right)\right)$. However, we know that $(\sy^{m_j}_{\cdot}\mathbbm{1}_{\cdot \leq \tau^{R,T}})$ converges to $\sy_{\cdot}\mathbbm{1}_{\cdot \leq \tau^{R,T}}$ $\mathbb{P}-a.s.$ in $L^2\left([0,T];H\right)$ from Lemma \ref{amazing cauchy lemma}, and from the Dominated Convergence Theorem (domination comes easily from   $\mathbbm{1}_{\cdot \leq \tau^{R,T}}$, $\tau^{R,T} \leq \tau^{M,T}_{m_j}$) then in fact the convergence holds in $L^2\left(\tilde{\Omega};L^2\left([0,T];H\right)\right)$. Of course the same convergence then holds in the weak topology, and by uniqueness of limits in the weak topology then $\sy_{\cdot}\mathbbm{1}_{\cdot \leq \tau^{R,T}} = \check{\py} = \hat{\py}$ as elements of $L^2\left(\Omega;L^2\left([0,T];H\right)\right)$, so they agree $\mathbbm{P} \times \lambda-a.s.$. Thus for $\mathbbm{P}-a.e.$ $\omega$, $\sy_{\cdot}(\omega)\mathbbm{1}_{\cdot \leq \tau^{R,T}(\omega)} \in L^{\infty}\left([0,T];H\right)\cap L^2\left([0,T];V\right)$. At each such $\omega$ the regularity of $\sy$ as a weak solution ensures that for sufficiently large $R$ (dependent on $\omega$), $\tau^{R,T}(\omega) = T$. Therefore for $\mathbb{P}-a.e.$ $\omega$
, $\sy_{\cdot}(\omega) \in L^{\infty}\left([0,T];H\right)\cap L^2\left([0,T];V\right)$.\\

The progressive measurability is justified similarly; for any $t\in [0,T]$, we can use the progressive measurability of $(\sy^{m_j})$ and hence $(\sy^{m_j}_{\cdot}\mathbbm{1}_{\cdot \leq \tau^{R,T}})$ to instead deduce $\sy_{\cdot}\mathbbm{1}_{\cdot \leq \tau^{R,T}}$ as the weak limit in $L^2\left(\Omega \times [0,t];V\right)$ where $\Omega \times [0,t]$ is equipped with the $\mathcal{F}_t \times \mathcal{B}\left([0,t]\right)$ sigma-algebra. Therefore $\sy_{\cdot}\mathbbm{1}_{\cdot \leq \tau^{R,T}}: \Omega \times [0,t] \rightarrow V$ is measurable with respect to this product sigma-algebra which justifies the progressive measurability of the truncated process. To carry this property over to $\sy$, we recognise $\sy$ as the $\mathbb{P} \times \lambda$ almost everywhere limit of the sequence $(\sy_{\cdot}\mathbbm{1}_{\cdot \leq \tau^{R,T}})$ as $R \rightarrow \infty$ over the product space $\Omega \times [0,t]$ equipped with product sigma-algebra $\mathcal{F}_t \times \mathcal{B}\left([0,t]\right)$. Such a limit preserves the measurability in this product sigma-algebra, justifying the progressive measurability of $\sy$.

\end{proof}


Applying Lemma \ref{lemma for regular weak is strong} in conjunction with Proposition \ref{final regularity of solutions} now proves Theorem \ref{theorem for strong existence} in the case of a bounded initial condition (\ref{first initial condition bound in H}). Extending this result to a general $\mathcal{F}_0-$measurable $\sy_0:\Omega \rightarrow H$ is done exactly as we did for the weak solution in Subsection \ref{exis uniqu}, so we do not repeat these steps and conclude the proof of Theorem \ref{theorem for strong existence} here.

\subsection{Continuity?} \label{subbie continuity?}

We close this section with a somewhat more informal discussion around the pathwise continuity of the strong solution $\sy$ in $H$. One may well have expected this to appear in the definition of the strong solution, though it seems challenging to verify and additional assumptions that could facilitate it are perhaps too restrictive. A first result is given now.

\begin{lemma} \label{continuity lemma 1}
    Let Assumption Sets 1, 2 and 3 hold. Suppose that there exists a continuous bilinear form $\inner{\cdot}{\cdot}_{U \times V}: U \times V \rightarrow \R$  such that for every $f \in H$, $\phi \in V$, $$ \inner{f}{\phi}_{U \times V} = \inner{f}{\phi}_H.$$
    Then for $\mathbbm{P}-a.e.$ $\omega$, $\sy_{\cdot}(\omega) \in C_w\left([0,T];H\right)$. 
\end{lemma}

\begin{proof}
    Fixing arbitrary $\phi \in V$, $\sy$ satisfies the identity
    $$ \inner{\sy_t}{\phi}_{U \times V} = \inner{\sy_0}{\phi}_{U \times V} + \inner{\int_0^t\mathcal{A}(s,\sy_s)ds}{\phi}_{U \times V} + \inner{\int_0^t\mathcal{G}(s,\sy_s)d\mathcal{W}_s}{\phi}_{U \times V}$$
    $\mathbbm{P}-a.s.$ for all $t \in [0,T]$, and by continuity of the bilinear form,
    $$ \inner{\sy_t}{\phi}_{U \times V} = \inner{\sy_0}{\phi}_{U \times V} + \int_0^t\inner{\mathcal{A}(s,\sy_s)}{\phi}_{U \times V}ds + \int_0^t\inner{\mathcal{G}(s,\sy_s)}{\phi}_{U \times V}d\mathcal{W}_s.$$
    We want to use that $\inner{\sy_t}{\phi}_{U \times V} = \inner{\sy_t}{\phi}_{H}$, but it is not yet clear that $\sy_t \in H$ $\mathbbm{P}-a.s.$ for all $t \in [0,T]$. It can, in fact, be verified by a very similar process to the proof of Proposition \ref{final regularity of solutions}. To show this let us fix any $t \in [0,T]$. For $\mathbbm{P}-a.e.$ $\omega$, $\sy_{\cdot}(\omega) \in L^\infty\left([0,T];H\right)$ thus for a countable dense subset $(t_n)$ of $[0,T]$, $\sy_{t_n}(\omega) \in H$ and such that $\norm{\sy_{t_n}(\omega)}_H \leq c$ independent of $n$. We choose now any sequence $(t_k)$ which converges to $t$. The sequence $(\sy_{t_k}(\omega))$ is uniformly bounded in $H$, so admits a weakly convergent subsequence to a limit $\psi$. However, as $\sy_{\cdot}(\omega) \in C\left([0,T];U\right)$ then $(\sy_{t_k}(\omega))$ is strongly hence weakly convergent to $\sy_t(\omega)$ in $U$, so by uniqueness of limits in the weak topology then $\sy_t(\omega) = \psi$ and in particular $\sy_t(\omega) \in H$. Thus, $\mathbbm{P}-a.s.$ for all $t \in [0,T]$,
    $$ \inner{\sy_t}{\phi}_{H} = \inner{\sy_0}{\phi}_{U \times V} + \int_0^t\inner{\mathcal{A}(s,\sy_s)}{\phi}_{U \times V}ds + \int_0^t\inner{\mathcal{G}(s,\sy_s)}{\phi}_{U \times V}d\mathcal{W}_s$$
    so that $\inner{\sy_t}{\phi}_{H}$ is $\mathbbm{P}-a.s.$ continuous. This is not quite enough for weak continuity, as we instead require that $\inner{\sy_t}{f}_{H}$ is $\mathbbm{P}-a.s.$ continuous for any $f \in H$, not just any $\phi \in V$. From this point, though, we can extend the continuity to $f$ by the density of $V$ in $H$, identically to the end of the proof of Proposition \ref{prop for first energy bar}.
\end{proof}

Pathwise continuity of $\sy$ in $H$ would now follow from only pathwise continuity of the norm $\norm{\sy_{\cdot}}_H$, see e.g. [\cite{goodair2022stochastic}] Lemma 1.5.3. A popular way to show such a property is through Kolmogorov's Continuity Theorem, by verifying that there exists a constant $C$ such that for all $s < t \in [0,T]$, $$\mathbbm{E}\left(\left\vert \norm{\sy_t}_H^2 - \norm{\sy_s}_H^2\right\vert^4 \right) \leq C(t-s)^2$$
which is the typical choice of parameters to apply Kolmogorov's Continuity Theorem for a Brownian Motion and related SDEs. Our assumptions don't quite allow us to show this, even at the level of the Galerkin System: the issue arises in the $V$-norm dependence in (\ref{uniformboundsassumpt2}), which is necessary to have for applications in the case of Navier Boundary Conditions as seen in Theorem \ref{2d strong  navier}. There is, however, a simpler case where continuity can be deduced.




\begin{lemma} \label{continuity lemma 2}
    Let Assumption Sets 1, 2 and 3 hold. Suppose that there exists a continuous bilinear form $\inner{\cdot}{\cdot}_{U \times V}: U \times V \rightarrow \R$  such that for every $f \in H$, $\phi \in V$, $$ \inner{f}{\phi}_{U \times V} = \inner{f}{\phi}_H.$$ In addition, suppose that $\bar{H}$ can be taken as $H$. Then for $\mathbbm{P}-a.e.$ $\omega$, $\sy_{\cdot}(\omega) \in C\left([0,T];H\right)$. 
\end{lemma}

\begin{proof}
    This is now a direct application of Proposition \ref{rockner prop}.
\end{proof}

\section{Applications: Stochastic Navier-Stokes Equations} \label{Section: applications}

We present applications of our framework for the Navier-Stokes Equation under Stochastic Advection by Lie Transport (SALT) proposed in [\cite{holm2015variational}], given by \begin{equation} \label{number2equationSALT}
    u_t = u_0 - \int_0^t\mathcal{L}_{u_s}u_s\,ds + \nu\int_0^t \Delta u_s\, ds + \int_0^t B(u_s) \circ d\mathcal{W}_s - \nabla \rho_t
\end{equation}
where $u$ represents the fluid velocity, $\rho$ the pressure\footnote{The pressure term is a semimartingale, and an explicit form for the SALT Euler Equation is given in [\cite{street2021semi}] Subsection 3.3}, $\mathcal{L}$ represents the nonlinear term and $B$ is a first order differential operator (the SALT Operator). A complete introduction to this equation is given in [\cite{goodair20233d}], with full technical details that are glossed over here, although we note that our results can be applied for a variety of additive, multiplicative and transport noise structures. The choice of SALT noise is particularly challenging and demonstrates the efficacy of our framework. Intrinsic to this stochastic methodology is that $B$ is defined relative to a collection of functions $(\xi_i)$ which physically represent spatial correlations. These $(\xi_i)$ can be determined at coarse-grain resolutions from finely resolved numerical simulations, and mathematically are derived as eigenvectors of a velocity-velocity correlation matrix (see [\cite{cotter2020data}, \cite{cotter2018modelling}, \cite{cotter2019numerically}]). Indeed, this operator $B$ is given by the actions of its components $B_i$ on a vector field $\phi$ by $$B_i: \phi \mapsto \sum_{j=1}^N\left(\xi_i^j\partial_j\phi + \phi^j\nabla \xi_i^j\right)$$
in $N$-dimensions, where the superscript denotes the $j^{\textnormal{th}}$ component mapping. We note this is the sum of a classical transport term and a zeroth-order term. The equation (\ref{number2equationSALT}) is posed on a smooth bounded domain $\mathscr{O} \subset \R^N$, and we introduce the notation $C^{\infty}_{0,\sigma}$ to be the subspace of $C^{\infty}_0\left(\mathscr{O};\R^N\right)$ consisting of divergence-free functions, which is to say those $\phi \in C^{\infty}_0\left(\mathscr{O};\R^N\right)$ such that $\sum_{j=1}^N\partial_j\phi^j = 0$. Following this, we define $L^2_{\sigma}$ as the closure of $C^{\infty}_{0,\sigma}$ in $L^2\left(\mathscr{O};\R^N\right)$, and the Leray Projector $\mathcal{P}$ as the orthogonal projection in $L^2\left(\mathscr{O};\R^N\right)$ onto $L^2_\sigma$.\\

To study this equation more freely, we commit two manipulations of (\ref{number2equationSALT}); the first is to project via $\mathcal{P}$, and the second is to convert to It\^{o} Form. With this, we arrive at
\begin{equation} \label{projected Ito}
    u_t = u_0 - \int_0^t\mathcal{P}\mathcal{L}_{u_s}u_s\ ds + \nu\int_0^t \mathcal{P} \Delta u_s\, ds + \frac{1}{2}\int_0^t\sum_{i=1}^\infty \mathcal{P}B_i^2u_s ds + \int_0^t \mathcal{P}B(u_s) d\mathcal{W}_s 
\end{equation}
which is now in the form of (\ref{thespde}). We emphasise again that a thorough overview of this process is given in [\cite{goodair20233d}]. Applications of our framework for different dimensions and boundary conditions are given below.

\subsection{No-Slip Boundary Condition} \label{applied no slip}

Let us impose the boundary condition $u=0$ on $\partial \mathscr{O}$; this is the so-called no-slip boundary condition. We define the space $W^{1,2}_{\sigma}$ as the closure of $C^{\infty}_{0,\sigma}$ in $W^{1,2}\left(\mathscr{O};\R^N\right)$, and $W^{2,2}_{\sigma}$ as the intersection of $W^{2,2}\left(\mathscr{O};\R^N\right)$ with $W^{1,2}_{\sigma}$. Indeed functions in $W^{1,2}_{\sigma}$ satisfy both the divergence-free and zero-trace conditions, so the incompressibility and boundary constraints are embedded into this function space. The functional framework of Subsection \ref{subs functional framework} is satisfied for the spaces
$$V:= W^{2,2}_{\sigma}, \qquad H:= W^{1,2}_{\sigma}, \qquad U:= L^2_{\sigma} $$ and the system $(a_k)$ of eigenfunctions of the Stokes Operator $-\mathcal{P}\Delta$. The mappings $\mathcal{P}\mathcal{L}$, $\mathcal{P}\Delta$ are understood from $W^{1,2}_{\sigma}$ into $\left(W^{1,2}_{\sigma}\right)^*$ by the duality pairings for $f,g \in W^{1,2}_{\sigma}$ of
\begin{align*}
    \inner{\mathcal{P}\mathcal{L}_ff}{g}_{\left(W^{1,2}_{\sigma}\right)^* \times W^{1,2}_{\sigma}} &= \inner{\mathcal{L}_ff}{g}_{L^{6/5}\times L^6}\\
    \inner{\mathcal{P}\Delta f}{g}_{\left(W^{1,2}_{\sigma}\right)^* \times W^{1,2}_{\sigma}} &= -\inner{f}{g}_1
\end{align*}
where $\inner{\cdot}{\cdot}_1$ represents the homogeneous $W^{1,2}$ inner product, which is to say without the $L^2$ component. As a direct application of Theorem \ref{theorem for martingale weak existence}, we can obtain the following.

\begin{theorem} \label{3D weak no slip}
    Let $N=2$ or $3$, $u_0 \in L^\infty\left(\Omega;L^2_{\sigma}\right)$ be $\mathcal{F}_0-$measurable, $(\xi_i) \in W^{1,2}_{\sigma} \cap W^{2,\infty}$ with $\sum_{i=1}^\infty \norm{\xi_i}_{W^{2,\infty}}^2 < \infty$. Then there exists a filtered probability space $\left(\tilde{\Omega},\tilde{\mathcal{F}},(\tilde{\mathcal{F}}_t), \tilde{\mathbbm{P}}\right)$, a Cylindrical Brownian Motion $\tilde{\mathcal{W}}$ over $\mathfrak{U}$ with respect to $\left(\tilde{\Omega},\tilde{\mathcal{F}},(\tilde{\mathcal{F}}_t), \tilde{\mathbbm{P}}\right)$, an $\mathcal{F}_0-$measurable $\tilde{u}_0: \tilde{\Omega} \rightarrow U$ with the same law as $u_0$, and a progressively measurable process $\tilde{u}$ in $W^{1,2}_{\sigma}$ such that for $\tilde{\mathbb{P}}-a.e.$ $\tilde{\omega}$, $\tilde{u}_{\cdot}(\omega) \in L^{\infty}\left([0,T];L^2_{\sigma}\right)\cap C_w\left([0,T];L^2_{\sigma}\right) \cap L^2\left([0,T];W^{1,2}_{\sigma}\right)$ and
\begin{align*} 
  \tilde{u}_t = \tilde{u}_0 - \int_0^t\mathcal{P}\mathcal{L}_{\tilde{u}_s}\tilde{u}_s\ ds + \nu\int_0^t \mathcal{P} \Delta \tilde{u}_s\, ds + \frac{1}{2}\int_0^t\sum_{i=1}^\infty \mathcal{P}B_i^2\tilde{u}_s ds + \int_0^t \mathcal{P}B(\tilde{u}_s) d\tilde{\mathcal{W}}_s
\end{align*}
holds $\tilde{\mathbb{P}}-a.s.$ in $\left(W^{1,2}_{\sigma}\right)^*$ for all $t \in [0,T]$.
\end{theorem}

In fact we can do better in 2D, as an application of Theorem \ref{theorem for weak existence}.

\begin{theorem} \label{2D weak no slip}
    Let $N=2$, $u_0: \Omega \rightarrow L^2_{\sigma}$ be $\mathcal{F}_0-$measurable, $(\xi_i) \in W^{1,2}_{\sigma} \cap W^{2,\infty}$ with $\sum_{i=1}^\infty \norm{\xi_i}_{W^{2,\infty}}^2 < \infty$. Then there exists a progressively measurable process $u$ in $W^{1,2}_{\sigma}$ such that for $\mathbb{P}-a.e.$ $\tilde{\omega}$, $u_{\cdot}(\omega) \in C\left([0,T];L^2_{\sigma}\right) \cap L^2\left([0,T];W^{1,2}_{\sigma}\right)$ and (\ref{projected Ito}) holds $\mathbb{P}-a.s.$ in $\left(W^{1,2}_{\sigma}\right)^*$ for all $t \in [0,T]$.
\end{theorem}

These results are proven in [\cite{goodair2023zero}] Theorems 1.9 and 1.10, stated slightly differently although Lemma 3.9 connects the definitions. Whilst these results are not new, they are very recent and our framework provides an efficient way to obtain them. The existence of strong solutions in 2D is far more challenging, however, and remains unsolved. Towards a success, we note the necessity of considering an extension $\bar{H}$ of $H$ in Assumption Set 3, Subsection \ref{assumption set 3}. The Leray Projector $\mathcal{P}$ does not preserve the zero-trace property, so $\mathcal{P}B_i$ does not map from $W^{2,2}_{\sigma}$ into $W^{1,2}_{\sigma}$, but instead an extended space $\bar{W}^{1,2}_{\sigma}$ defined as the intersection of $W^{1,2}$ with $L^2_{\sigma}$. This is again a Hilbert Space with $\inner{\cdot}{\cdot}_1$ inner product; see [\cite{goodair2023navier}] Subsection 1.2.\\

Nevertheless, we are still unable to verify (\ref{mu2}) and Assumption \ref{assumptions for uniform bounds2}. This owes to the fact that $\mathcal{P}_n$ is self-adjoint only on $W^{1,2}_{\sigma}$ and not $\bar{W}^{1,2}_{\sigma}$, leaving us stuck with the finite dimensional projection in a way which offers no clear solution. The situation is different for the Navier boundary conditions. 

\subsection{Navier Boundary Conditions} \label{applied navier boundary}

In two spatial dimensions we can impose different boundary conditions for (\ref{projected Ito}), namely the Navier boundary conditions. These are defined on $\partial\mathscr{O}$ by \begin{equation} \label{navier boundary conditions}
    u \cdot \mathbf{n} = 0, \qquad 2(Du)\mathbf{n} \cdot \mathbf{\iota} + \alpha u\cdot \mathbf{\iota} = 0
\end{equation} 
where $\mathbf{n}$ is the unit outwards normal vector, $\mathbf{\iota}$ the unit tangent vector, $Du$ is the rate of strain tensor $(Du)^{k,l}:=\frac{1}{2}\left(\partial_ku^l + \partial_lu^k\right)$ and $\alpha \in C^2(\partial \mathscr{O};\R)$ represents a friction coefficient which determines the extent to which the fluid slips on the boundary relative to the tangential stress. These conditions were first proposed by Navier in [\cite{navier1822memoire}, \cite{navier1827lois}], and have been derived in [\cite{maxwell1879vii}] from the kinetic theory of gases and in [\cite{masmoudi2003boltzmann}] as a hydrodynamic limit. Furthermore these conditions have proven viable for modelling rough boundaries as seen in [\cite{basson2008wall}, \cite{gerard2010relevance}, \cite{pare1992existence}]. To fit the framework of this paper we again have to embed the boundary conditions into useful function spaces. We shall use the space $\bar{W}^{1,2}_{\sigma}$, which was the intersection of $W^{1,2}$ with $L^2_{\sigma}$, and contains the divergence-free and impermeable boundary condition (that $u \cdot \mathbf{n} = 0$). The remaining component of (\ref{navier boundary conditions}) has to be included at the $W^{2,2}$ level, as we are concerned with the trace of a derivative which needs more than $W^{1,2}$ regularity to be understood in the usual sense. Thus, we define
$$\bar{W}^{2,2}_{\alpha}:= \left\{f \in W^{2,2}(\mathscr{O};\R^2) \cap \bar{W}^{1,2}_{\sigma}: 2(Df)\mathbf{n} \cdot \iota + \alpha f \cdot \iota = 0 \textnormal{ on } \partial \mathscr{O}\right\}.$$ Of course this space does not appear in the definition of a weak solution, so it is perhaps unclear how the boundary conditions (\ref{navier boundary conditions}) inform the weak solution. The answer comes from how to extend the Stokes Operator $-\mathcal{P}\Delta$ to $\bar{W}^{1,2}_{\sigma}$. In [\cite{kelliher2006navier}] equation (5.1), c.f. [\cite{goodair2023navier}] Lemma 1.4, it is verified that for $\phi \in \bar{W}^{2,2}_{\alpha}$, $f \in \bar{W}^{1,2}_{\sigma}$,
$$ \inner{\mathcal{P}\Delta \phi}{f}_{L^2} = -\inner{\phi}{f}_1 + \inner{(\kappa - \alpha)\phi}{f}_{L^2(\partial \mathscr{O};\R^2)}$$
where $\kappa:\partial\mathscr{O} \rightarrow \R$ represents the curvature of the boundary. Therefore, for each $\alpha$ as in (\ref{navier boundary conditions}), we extend the Stokes Operator from $\bar{W}^{2,2}_{\alpha}$ to $\bar{W}^{1,2}_{\sigma}$ as a mapping into $\left(\bar{W}^{1,2}_{\sigma}\right)^*$ by the duality pairing for $g,f \in \bar{W}^{1,2}_{\sigma}$ of $$\inner{\mathcal{P}\Delta g}{f}_{\left(\bar{W}^{1,2}_{\sigma}\right)^* \times \bar{W}^{1,2}_{\sigma}} = -\inner{g}{f}_1 + \inner{(\kappa - \alpha)g}{f}_{L^2(\partial \mathscr{O};\R^2)}. $$
The nonlinear term requires no special attention to be understood in the weak sense, similarly to the no-slip boundary condition. We equip $\bar{W}^{2,2}_{\alpha}$ with the inner product $\inner{f}{g}_2:= \inner{\mathcal{P}\Delta f}{\mathcal{P}\Delta g}_{L^2}$ which is equivalent to the standard $W^{2,2}$ inner product (see [\cite{goodair2023navier}] Lemma 1.2), and $\bar{W}^{1,2}_{\sigma}$ with the $\inner{\cdot}{\cdot}_1$ inner product. Then as a direct application of Theorem \ref{theorem for weak existence} we obtain the following.

\begin{theorem} \label{2D weak navier}
    Let $\alpha \in C^2(\partial \mathscr{O};\R)$, $u_0: \Omega \rightarrow L^2_{\sigma}$ be $\mathcal{F}_0-$measurable, $(\xi_i) \in W^{1,2}_{\sigma} \cap W^{2,\infty}$ with $\sum_{i=1}^\infty \norm{\xi_i}_{W^{2,\infty}}^2 < \infty$. Then there exists a progressively measurable process $u$ in $\bar{W}^{1,2}_{\sigma}$ such that for $\mathbb{P}-a.e.$ $\tilde{\omega}$, $u_{\cdot}(\omega) \in C\left([0,T];L^2_{\sigma}\right) \cap L^2\left([0,T];\bar{W}^{1,2}_{\sigma}\right)$ and (\ref{projected Ito}) holds $\mathbb{P}-a.s.$ in $\left(\bar{W}^{1,2}_{\sigma}\right)^*$ for all $t \in [0,T]$.
\end{theorem}

 This result is given in [\cite{goodair2023navier}] Theorem 1.14, and a strong existence result is proven as Theorem 1.15 in the same paper. For this, we need to make some adjustments; to verify Assumption \ref{assumptions for uniform bounds2} we cannot use the $\inner{\cdot}{\cdot}_1$ inner product for $H$ as this leads to an uncontrollable boundary integral. Instead, we have to manufacture a more reasonable boundary integral into our inner product. Thus, we instead equip $\bar{W}^{1,2}_{\sigma}$ with $$\inner{f}{g}_H := \inner{f}{g}_1 + \inner{(\kappa - \alpha)f}{g}_{L^2(\partial \mathscr{O};\R^2)}$$
which is an inner product equivalent to the usual $W^{1,2}$ form when $\alpha \geq \kappa$ everywhere on $\partial \mathscr{O}$. This requirement appears in the result, obtainable through Theorem \ref{theorem for strong existence}.

\begin{theorem} \label{2d strong  navier}
    Let $\alpha \in C^2(\partial \mathscr{O};\R)$ be such that $\alpha \geq \kappa$, $u_0: \Omega \rightarrow \bar{W}^{1,2}_{\sigma}$ be $\mathcal{F}_0-$measurable, $(\xi_i) \in L^2_{\sigma} \cap W^{3,2}_0 \cap W^{3,\infty}$ with $\sum_{i=1}^\infty \norm{\xi_i}_{W^{3,\infty}}^2 < \infty$. Then there exists a progressively measurable process $u$ in $\bar{W}^{2,2}_{\alpha}$ such that for $\mathbb{P}-a.e.$ $\tilde{\omega}$, $u_{\cdot}(\omega) \in C\left([0,T];\bar{W}^{1,2}_{\sigma}\right) \cap L^2\left([0,T];\bar{W}^{2,2}_{\alpha}\right)$ and (\ref{projected Ito}) holds $\mathbb{P}-a.s.$ in $L^2_{\sigma}$ for all $t \in [0,T]$.
\end{theorem}

We note use of Lemma \ref{continuity lemma 2} here, as in this case, the extension $\bar{H}$ of $H$ in Assumption Set 3, Subsection \ref{assumption set 3}, can simply be $\bar{W}^{1,2}_{\sigma}$ itself. In addition, it should be noted that the requirement $\alpha \geq \kappa$ is rather reasonable; at least heuristically, as $\alpha$ grows large then $u \cdot \iota = 0$ dominates the second identity of (\ref{navier boundary conditions}), which would then result in the traditional no-slip condition. Given the wide acceptance of the no-slip condition, deviation from it with Navier boundary conditions is only expected for large $\alpha$. A rigorous result regarding the convergence of solutions to the deterministic Navier-Stokes equation with Navier boundary conditions to the no-slip solution for $\alpha \rightarrow \infty$ is available in [\cite{kelliher2006navier}] Section 9.\\ 

Of course this invites the question as to how the Navier boundary conditions solve the issue of $\mathcal{P}_n$ present for the no-slip case. More detail is given in the conclusion of [\cite{goodair2023navier}], but in essence, this owes to the fact that the basis of eigenfunctions of the Stokes Operator satisfying the Navier boundary conditions are dense in the range of the Leray Projector in $W^{1,2}$. That is, these eigenfunctions form a basis of $\bar{W}^{1,2}_{\sigma}$ instead of $W^{1,2}_{\sigma}$, so the Leray Projector mapping only into $\bar{W}^{1,2}_{\sigma}$ is now non-problematic.

\section{Appendix} \label{section appendix}

\subsection{The Cauchy Result} \label{subby to prove cauchy}

\begin{proposition} \label{amazing cauchy lemma}
    Fix $T>0$. For $t\in[0,T]$ let $X_t$ denote a Banach Space with norm $\norm{\cdot}_{X,t}$ such that for all $s > t$, $X_s \xhookrightarrow{}X_t$ and $\norm{\cdot}_{X,t} \leq \norm{\cdot}_{X,s}$. Suppose that $(\sy^n)$ is a sequence of processes $\sy^n:\Omega \mapsto X_T$, $\norm{\sy^n}_{X,\cdot}$ is adapted and $\mathbb{P}-a.s.$ continuous, $\sy^n \in L^2\left(\Omega;X_T\right)$, and such that $\sup_{n}\norm{\sy^n}_{X,0} \in L^\infty\left(\Omega;\R\right)$. For any given $M >1$ define the stopping times
    \begin{equation} \label{another taumt}
        \tau^{M,T}_n := T \wedge \inf\left\{s \geq 0: \norm{\sy^n}_{X,s}^2 \geq M + \norm{\sy^n}_{X,0}^2 \right\}.
    \end{equation}
Furthermore suppose \begin{equation} \label{supposition 1}
    \lim_{m \rightarrow \infty}\sup_{n \geq m}\mathbbm{E}\left[\norm{\sy^n-\sy^m}^2_{X,\tau
_{m}^{M,t}\wedge \tau_{n}^{M,t}} \right] =0
\end{equation}
and that for any stopping time $\gamma$ and sequence of stopping times $(\delta_j)$ which converge to $0$ $\mathbb{P}-a.s.$, \begin{equation} \label{supposition 2} \lim_{j \rightarrow \infty}\sup_{n\in\N}\mathbbm{E}\left(\norm{\sy^n}_{X,(\gamma + \delta_j) \wedge \tau^{M,T}_n}^2 - \norm{\sy^n}_{X,\gamma \wedge \tau^{M,T}_n}^2 \right) =0.
\end{equation}
Then there exists a stopping time $\tau^{M,T}_{\infty}$, a process $\sy:\Omega \mapsto X_{\tau^{M,T}_{\infty}}$ whereby $\norm{\sy}_{X,\cdot \wedge \tau^{M,T}_{\infty}}$ is adapted and $\mathbb{P}-a.s.$ continuous, and a subsequence indexed by $(m_j)$ such that 
\begin{itemize}
    \item $\tau^{M,T}_{\infty} \leq \tau^{M,T}_{m_j}$ $\mathbb{P}-a.s.$,
    \item $\lim_{j \rightarrow \infty}\norm{\sy - \sy^{m_j}}_{X,\tau^{M,T}_{\infty}} = 0$ $\mathbb{P}-a.s.$.
\end{itemize}
Moreover for any $R>0$ we can choose $M$ to be such that the stopping time \begin{equation} \label{another tauR}
        \tau^{R,T} := T \wedge \inf\left\{s \geq 0: \norm{\sy}_{X,s\wedge\tau^{M,T}_{\infty}}^2 \geq R \right\}
    \end{equation}
satisfies $\tau^{R,T} \leq \tau^{M,T}_{\infty}$ $\mathbb{P}-a.s.$. Thus $\tau^{R,T}$ is simply $T \wedge \inf\left\{s \geq 0: \norm{\sy}_{X,s}^2 \geq R \right\}$.

\end{proposition}

\begin{remark}
    A consequence of the properties that $\sup_{n}\norm{\sy^n}_{X,0} \in L^\infty\left(\Omega;\R\right)$ and\\ $\lim_{j \rightarrow \infty}\norm{\sy - \sy^{m_j}}_{X,\tau^{M,T}_{\infty}} = 0$ $\mathbb{P}-a.s.$ is that $\norm{\sy}_{X,0} \in  L^\infty\left(\Omega;\R\right)$. Therefore for\\ $R > \norm{\norm{\sy}_{X,0}}_{L^\infty\left(\Omega;\R\right)}$ we have that $\tau^{R,T}$ is $\mathbbm{P}-a.s.$ positive, hence so too is $\tau^{M,T}_{\infty}$ for appropriately chosen $M$.
\end{remark}

    \begin{proof}
   Property (\ref{supposition 1}) implies that for any given $j \in \N$ we can choose an $n_j\in \N$ such that for all $k \geq n_j$, \begin{equation}\label{a good property}\mathbbm{E}\left(\norm{\sy^k - \sy^{n_j}}^2_{X,\tau
_{n_{j}}^{M,t}\wedge \tau_{k}^{M,t}} \right) \leq 2^{-4j}.\end{equation}
We shall make use of delicate manipulations of the subsequence indexed by $(n_j)$, and for this we introduce a new sequence of stopping times. We now impose that $$M > 2 + \left\Vert\sup_{n \in \N}\norm{\sy^n}_{X,0}^2\right\Vert_{L^\infty(\Omega;\R)}$$ and define $$\tilde{M}^2:= \frac{M -\left\Vert\sup_{n \in \N }\norm{\sy^n}_{X,0}^2\right\Vert_{L^\infty(\Omega;\R)}}{2} > 1.$$
The purpose of this is to define $$\sigma^M_j:= T \wedge \inf\left\{s > 0: \norm{\sy^{n_{j}}}_{X,s}\geq (\tilde{M} - 1 +2^{-j}) + \norm{\sy^{n_{j}}}_{X,0} \right\}$$
and ensure that $\sigma^M_j \leq \tau^{M,T}_{n_j}$ at every $\omega$. Note the key difference in not squaring the norm, and also that $\tilde{M}>1$ so each $\sigma^M_j$ is necessarily positive. To demonstrate the inequality, it is sufficient to show that, $\mathbb{P}-a.s.$, \begin{equation} \label{or more easily}
    \left((\tilde{M} - 1 +2^{-j}) + \norm{\sy^{n_{j}}}_{X,0}\right)^2 \leq M +  \norm{\sy^{n_{j}}}_{X,0}^2
\end{equation}
or more easily 
$$  \left(\tilde{M} + \norm{\sy^{n_{j}}}_{X,0}\right)^2 \leq M +  \norm{\sy^{n_{j}}}_{X,0}^2.$$
This is possible as 
\begin{align*}
    \left(\tilde{M} + \norm{\sy^{n_{j}}}_{X,0}\right)^2 &\leq 2\tilde{M}^2 + 2\norm{\sy^{n_{j}}}_{X,0}^2 = M -\left\Vert\sup_{n }\norm{\sy^n}_{X,0}^2\right\Vert_{L^\infty(\Omega;\R)} + 2\norm{\sy^{n_{j}}}_{X,0}^2\\ &\leq M + \norm{\sy^{n_{j}}}_{X,0}^2.
\end{align*}
The property (\ref{or more easily}) is thus verified, so $\sigma^M_j \leq \tau^{M,T}_{n_j}$ and hence the subsequence $(\sy^{n_j})$ enjoys the same properties up until the corresponding $\sigma^M_j$. In particular from (\ref{a good property}), \begin{equation}\label{a good property 2}
    \mathbbm{E}\left(\norm{\sy^{n_{j+1}} - \sy^{n_j}}_{X,\sigma^M
_j\wedge \sigma_{j+1}^{M}} \right) \leq \left[\mathbbm{E}\left(\norm{\sy^{n_{j+1}} - \sy^{n_j}}_{X,\sigma^M
_j\wedge \sigma_{j+1}^{M}}^2 \right)\right]^{\frac{1}{2}}\leq 2^{-2j},
\end{equation}
hence in defining the sets
\begin{equation}\label{defined omega j}\Omega_j:=\left\{\omega \in \Omega: \norm{\sy^{n_{j+1}}(\omega)-\sy^{n_j}(\omega)}_{X,\sigma^{M}_{j}(\omega)\wedge \sigma^{M}_{j+1}(\omega)} < 2^{-(j+2)} \right\}\end{equation}
we have, by Chebyshev's Inequality and (\ref{a good property 2}),
\begin{align*}
    \mathbb{P}\left(\Omega_j^C \right) \leq 2^{j+2}\mathbbm{E}\left(\norm{\sy^{n_{j+1}}-\sy^{n_j}}_{X,\sigma^{M}_{j}\wedge \sigma^{M}_{j+1}} \right) \leq 2^{-j + 2}.
\end{align*}
We have, therefore, that $$\sum_{j=1}^\infty \mathbb{P}\left( \Omega_j^C   \right) < \infty$$ from which we see $$\mathbb{P} \left( \bigcap_{K=1}^\infty \bigcup_{j=K}^\infty \Omega_j^C \right) = 0$$ courtesy of Borel-Cantelli. It then follows that the set $$\hat{\Omega} := \bigcup_{K=1}^\infty \bigcap_{j=K}^\infty \Omega_j$$ is such that $\mathbb{P}( \hat{\Omega}) = 1$ so that in verifying $\mathbb{P}-a.s.$ properties, we can in fact simply show that they hold \textit{everywhere} on $\hat{\Omega}$. More precisely, we also take $\hat{\Omega}$ to be such that every $\norm{\sy^{n_j}}_{X,\cdot}$ is continuous on $\hat{\Omega}$, which is only a further countable intersection of full measure sets. We proceed by considering the sets $$\hat{\Omega}_K:= \bigcap_{j=K}^\infty \Omega_j$$ with the idea to just show such properties on $\hat{\Omega}_K$ for all $K$ (as their union makes up $\hat{\Omega}$). We look to construct a new stopping time $\sigma^M_{\infty}$ (which will prove to be the desired $\tau^{M,T}_{\infty}$) given as the $\mathbb{P}-a.e.$ limit of $(\sigma^M_j)$, built from demonstrating that $(\sigma^M_j)$ is monotone decreasing everywhere on $\hat{\Omega}_K$ for all $K$. In other words we show that for sufficiently large $j$ (in fact, just $j\geq K$) that the set \begin{equation} \label{def of set} \Gamma:= \{ \sigma^M_j < \sigma^M_{j+1} \} \cap \hat{\Omega}_K\end{equation} is empty. Firstly we observe from the strict inequality $\sigma^M_j < \sigma^M_{j+1}$ on this set that $\sigma^M_j < T$, implying that 
$$\sigma^M_j= \inf\left\{s > 0: \norm{\sy^{n_{j}}}_{X,s}\geq (\tilde{M} - 1 +2^{-j}) + \norm{\sy^{n_{j}}}_{X,0} \right\}$$
so by the continuity of $\norm{\sy^{n_j}}_{X,\cdot}$, \begin{equation}\label{by the continuity}\norm{\sy^{n_{j}}}_{X,\sigma^M_j} =(\tilde{M} - 1 +2^{-j}) + \norm{\sy^{n_{j}}}_{X,0}. \end{equation} Using the definition of $\Omega_j$, (\ref{defined omega j}), for $j \geq K$, we have that \begin{equation}\label{it is useful}\norm{\sy^{n_j}}_{X,\sigma^M_j \wedge \sigma^M_{j+1}} - \norm{\sy^{n_{j+1}}}_{X,\sigma^M_j \wedge \sigma^M_{j+1}} \leq \norm{\sy^{n_{j+1}} - \sy^{n_j} }_{X,\sigma^M_j \wedge \sigma^M_{j+1}} < 2^{-\left(j+2\right)}\end{equation}
and also
\begin{equation} \label{useful ic}
    \norm{\sy^{n_{j+1}}}_{X,0} - \norm{\sy^{n_j}}_{X,0}  \leq \norm{\sy^{n_{j+1}} - \sy^{n_j} }_{X,0} < 2^{-\left(j+2\right)}.
\end{equation}
Combining (\ref{by the continuity}), (\ref{it is useful}) and (\ref{useful ic}), whilst using that $\sigma^M_j < \sigma^M_{j+1}$, we see that \begin{align}
    \nonumber \norm{\sy^{n_{j+1}}}_{X,\sigma^M_j \wedge \sigma^M_{j+1}} &> \norm{\sy^{n_j}}_{X,\sigma^M_j \wedge \sigma^M_{j+1}} - 2^{-\left(j+2\right)}\\ \nonumber
    &= \norm{\sy^{n_j}}_{X,\sigma^M_j} - 2^{-\left(j+2\right)}\\ \nonumber
    &=  \norm{\sy^{n_{j}}}_{X,0} + (\tilde{M} - 1 +2^{-j}) - 2^{-\left(j+2\right)}\\  \nonumber
    &> \norm{\sy^{n_{j+1}}}_{X,0} - 2^{-(j+2)} + (\tilde{M} - 1 +2^{-j}) - 2^{-\left(j+2\right)}\\
    &= \norm{\sy^{n_{j+1}}}_{X,0} + (\tilde{M} - 1 +2^{-(j+1)}), \label{end of the align}
\end{align}
where in the last line we have used the manipulation
$$ - 2^{-(j+2)} - 2^{-(j+2)} + 2^{-j} = -2^{-j-1} + 2^{-j} = 2^{-j}\left( -2^{-1} + 1\right) = 2^{-j}\left( 2^{-1} \right) = 2^{-(j+1)}.$$ The hard work is done in showing that the set $\Gamma$ defined in (\ref{def of set}) is empty, as on this set note that $$\norm{\sy^{n_{j+1}}}_{X,\sigma^M_j \wedge \sigma^M_{j+1}} \leq \norm{\sy^{n_{j+1}}}_{X, \sigma^M_{j+1}} \leq \norm{\sy^{n_{j+1}}}_{X,0} + (\tilde{M} - 1 +2^{-(j+1)}),$$ which contradicts (\ref{end of the align}), hence $\Gamma$ must be empty. Thus on every $\hat{\Omega}_K$, and furthermore the whole of $\hat{\Omega}$, the sequence $(\sigma^M_j)$ is eventually monotone decreasing (and bounded below by $0$). Furthermore we define $\sigma^M_{\infty}$ as the pointwise limit $\lim_{j \rightarrow \infty}\sigma^M_j$ on $\hat{\Omega}$, which must itself be a stopping time as the $\mathbb{P}-a.s.$ limit of stopping times. As mentioned this shall prove to be our $\tau^{M,T}_\infty$, and for the existence of $\sy$ we show that on $\hat{\Omega}$ the subsequence $(\sy^{n_j})$ is Cauchy in $X_{\sigma^M_{\infty}}$. Every $\omega \in \hat{\Omega}$ belongs to $\hat{\Omega}_K$ for some $K$, and furthermore to $\hat{\Omega}_L$ for all $L > K$. We fix arbitrary $\omega \in \hat{\Omega}$ and select an associated $K$. At this $\omega$, for any $j > k \geq K$, observe that
\begin{align*}
    \norm{\sy^{n_j} - \sy^{n_k}}_{X, \sigma^M_{\infty}} &= \norm{\sy^{n_{j}} -\sy^{n_{k+1}} + \sy^{n_{k+1}} - \sy^{n_k}}_{X, \sigma^M_{\infty}}\\ &\leq \norm{\sy^{n_{j}} -\sy^{n_{k+1}}}_{X, \sigma^M_{\infty}} + \norm{\sy^{n_{k+1}} - \sy^{n_k}}_{X, \sigma^M_{\infty}}\\ &\leq \norm{\sy^{n_{j}} -\sy^{n_{k+1}}}_{X, \sigma^M_{\infty}} + 2^{-(k+2)}\\& \leq \sum_{l=k}^{j}2^{-(l+2)}\\
    &\leq 2^{-(k+1)}
\end{align*}
having carried out an inductive argument in the penultimate step. We are thus free to take $K$ large enough so that this difference is arbitrarily small; therefore there exists a limit in the Banach Space $X_{\sigma^M_{\infty}}$, which we call $\sy$. The process $\norm{\sy}_{X,\cdot \wedge \sigma^M_{\infty}}$ is adapted and $\mathbb{P}-a.s.$ continuous, as $$\sup_{r \in [0,T]}\left\vert \norm{\sy}_{X,r \wedge \sigma^M_{\infty}} - \norm{\sy^{n_j}}_{X,r \wedge \sigma^M_{\infty}}\right\vert \leq \sup_{r \in [0,T]}\left\vert \norm{\sy - \sy^{n_j}}_{X,r \wedge \sigma^M_{\infty}}\right\vert = \norm{\sy - \sy^{n_j}}_{X,\sigma^M_{\infty}},$$ 
which has $\mathbb{P}-a.s.$ limit as $j \rightarrow \infty$ equal to zero. Thus $\norm{\sy}_{X,\cdot \wedge \sigma^M_{\infty}}$ is given, $\mathbb{P}-a.s.$, as the uniform in time limit of adapted and continuous processes, verifying the result. Moving on, it is now that we make use of (\ref{supposition 2}) much in the same way as we did for (\ref{supposition 1}). This will be done in the context of $\gamma:= \sigma^M_{\infty}$ and $\delta_j:= \sigma^M_j - \sigma^M_{\infty}$. Indeed for any $j \in \N$ we can choose an $m_j \in \N$ (where $m_j=n_l$ some $l$) such that for all $k \geq m_j$, $$\sup_{n\in\N}\mathbbm{E}\left(\norm{\sy^n}_{X,\sigma^M_k \wedge \tau^{M,T}_{n}}^2 - \norm{\sy^n}_{X,\sigma^M_{\infty} \wedge \tau^{M,T}_n}^2\right) \leq 2^{-2j}.$$ In particular, through a relabelling of $\sigma^M_{m_j}=\sigma^M_l$, $$\mathbbm{E}\left(\norm{\sy^{m_j}}_{X,\sigma^M_{m_j}}^2 - \norm{\sy^{m_j}}_{X,\sigma^M_{\infty}}^2\right) \leq 2^{-2j}$$
by choosing $n$ as $m_j$ and using that $\sigma^M_{\infty} \leq \sigma^M_k \leq \sigma^M_{m_j} \leq \tau^{M,T}_{m_j}$. In a familiar way we define $$\Omega'_j:= \left\{\norm{\sy^{m_j}}_{X,\sigma^M_{m_j}}^2 - \norm{\sy^{m_j}}_{X,\sigma^M_{\infty}}^2 < 2^{-(j+2)} \right\} $$
so that, just as we showed for (\ref{defined omega j}), $$\check{\Omega}_K:=\bigcap_{j=K}^\infty \Omega'_j, \qquad \check{\Omega}:=\bigcup_{K=1}^\infty \check{\Omega}_K, \qquad  \mathbb{P}\left(\check{\Omega}\right)=1.$$ For arbitrary given $R>0$, the plan now is to find a constant $M$ such that at every $\omega \in \hat{\Omega}\cap\check{\Omega}$, either $\sigma^M_{\infty} = T$ or $\norm{\sy}^2_{X,\sigma^M_{\infty}} \geq R$. In both instances it is clear that $\tau^{R,T} \leq \sigma^M_{\infty}$, thus proving the proposition. To this end we fix an $\omega \in \hat{\Omega}\cap\check{\Omega}$ such that $\sigma^M_{\infty} < T$. As $\sigma^M_{\infty}$ is the decreasing limit of $(\sigma^M_{m_j})$ then for sufficiently large $m_j$ we must also have that $\sigma^M_{m_j} < T$. Exactly as in (\ref{by the continuity}), \begin{equation}\label{newcontinuitything} \norm{\sy^{m_{j}}}_{X,\sigma^M_{m_j}} =(\tilde{M} - 1 +2^{-j}) + \norm{\sy^{m_{j}}}_{X,0}. \end{equation}
From the proven convergence we also have that for sufficiently large $m_j$, \begin{equation}\label{applynewestcauchy}
    \norm{\sy - \sy^{m_j}}_{X,\sigma^M_{\infty}} < 1
\end{equation}
which implies that $\norm{\sy}_{X,\sigma^M_{\infty}} > \norm{\sy^{m_j}}_{X,\sigma^M_{\infty}} -1$, and likewise as $\omega \in \check{\Omega}_K$ for some $K$, \begin{equation} \label{lkjh}
    \norm{\sy^{m_j}}_{X,\sigma^M_{m_j}}^2 - \norm{\sy^{m_j}}_{X,\sigma^M_{\infty}}^2 < 1.
\end{equation}
We fix an $m_j$ large enough so that (\ref{newcontinuitything}), (\ref{applynewestcauchy}) and (\ref{lkjh}) all hold. Substituting (\ref{newcontinuitything}) into (\ref{lkjh}) gives that $$\norm{\sy^{m_j}}_{X,\sigma^M_{\infty}}^2 > \left((\tilde{M} - 1 +2^{-j}) + \norm{\sy^{m_{j}}}_{X,0}\right)^2 -1 > (\tilde{M}-1)^2 -1.$$ If $\tilde{M} > 2$ then the expression on the right is positive and $$ \norm{\sy^{m_j}}_{X,\sigma^M_{\infty}} > \left((\tilde{M}-1)^2 -1 \right)^{\frac{1}{2}}.$$
Furthermore $$ \norm{\sy}_{X,\sigma^M_{\infty}} > \left((\tilde{M}-1)^2 -1 \right)^{\frac{1}{2}} -1$$ where the right hand side is of course monotone increasing and unbounded in $\tilde{M}$ and hence $M$. By choosing $M$ large enough such that $$ \left[\left((\tilde{M}-1)^2 -1 \right)^{\frac{1}{2}} -1\right]^2 > R$$ we complete the proof.
\end{proof}

\subsection{Useful Results from the Literature} \label{subby statements from lit}

We state some key theorems used throughout the paper. Firstly is a well-posedness result for an SPDE in a finite dimensional Hilbert Space (though driven by a Cylindrical Brownian Motion) in the standard case of Lipschitz and linear growth constraints.
\begin{proposition} \label{Skorotheorem}
Fix a finite-dimensional Hilbert Space $\mathcal{H}$. Suppose the following:
\begin{itemize}
    \item[1:] For any $T>0$, the operators $\mathscr{A}:[0,T] \times \mathcal{H} \rightarrow \mathcal{H}$ and $\mathscr{G}:[0,T] \times \mathcal{H} \rightarrow \mathscr{L}^2(\mathfrak{U};\mathcal{H})$ are measurable;\\
    \item[2:] There exists a $C_{\cdot}:[0,\infty) \rightarrow \R$ bounded on $[0,T]$ for every $T$, and constants $c_i$ such that for every $\boldsymbol{\phi}, \psi \in \mathcal{H}$ and $t \in [0,\infty)$, \begin{align*}\norm{\mathscr{A}(t,\boldsymbol{\phi})}^2_{\mathcal{H}}  &\leq C_t\left[1 + \norm{\boldsymbol{\phi}}_{\mathcal{H}}^2\right]\\
     \norm{\mathscr{G}_i(t,\boldsymbol{\phi})}^2_{\mathcal{H}} &\leq C_tc_i\left[1 + \norm{\boldsymbol{\phi}}_{\mathcal{H}}^2\right]\\
     \sum_{i=1}^\infty c_i &< \infty\\
    \norm{\mathscr{A}(t,\boldsymbol{\phi}) - \mathscr{A}(t,\psi)}^2_{\mathcal{H}} &+\sum_{i=1}^\infty \norm{\mathscr{G}_i(t,\boldsymbol{\phi}) - \mathscr{G}_i(t,\psi)}^2_{\mathcal{H}} \leq C_t \norm{\boldsymbol{\phi}-\psi}_{\mathcal{H}}^2
    \end{align*}
    \item[3:] $\py_0 \in L^2(\Omega;\mathcal{H})$.\\
\end{itemize}
Then there exists a process $\py:[0,\infty) \times \Omega \rightarrow \mathcal{H}$ such that for $\mathbbm{P}-a.e.$ $\omega$, $\py_\cdot(\omega) \in C\left([0,T];\mathcal{H}\right)$ for every $T>0$, $\py$ is progressively measurable in $\mathcal{H}$ and the identity \begin{equation}\label{identityingalerkinsolution}\py_t = \py_0 + \int_0^t \mathscr{A}(s,\py_s)ds + \int_0^t \mathscr{G}(s,\py_s)d\mathcal{W}_s\end{equation} holds $\mathbbm{P}-a.s.$ in $\mathcal{H}$ for every $t \geq 0$.\\ 

Moreover, suppose that $\sy:[0,\infty) \times \Omega \rightarrow \mathcal{H}$ is another process such that for $\mathbbm{P}-a.e.$ $\omega$, $\sy_\cdot(\omega) \in C\left([0,T];\mathcal{H}\right)$ for every $T>0$, $\sy$ is progressively measurable in $\mathcal{H}$ and the identity (\ref{identityingalerkinsolution}) holds $\mathbbm{P}-a.s.$ in $\mathcal{H}$ for every $t \geq 0$. Then for every $T \geq 0$, $$\mathbbm{P}\left(\left\{\omega \in \Omega: \py_t(\omega) = \sy_t(\omega) \quad \forall t \in [0,T] \right\}\right)=1.$$
\end{proposition}

\begin{proof}
See [\cite{goodair2022stochastic}] Theorems 3.1.1 and 3.1.2.

\end{proof}

\begin{proposition} \label{rockner prop}
Let $\mathcal{H}_1 \subset \mathcal{H}_2 \subset \mathcal{H}_3$ be a triplet of embedded Hilbert Spaces where $\mathcal{H}_1$ is dense in $\mathcal{H}_2$, with the property that there exists a continuous nondegenerate bilinear form $\inner{\cdot}{\cdot}_{\mathcal{H}_3 \times \mathcal{H}_1}: \mathcal{H}_3 \times \mathcal{H}_1 \rightarrow \R$ such that for $\phi \in \mathcal{H}_2$ and $\psi \in \mathcal{H}_1$, $$\inner{\phi}{\psi}_{\mathcal{H}_3 \times \mathcal{H}_1} = \inner{\phi}{\psi}_{\mathcal{H}_2}.$$ Suppose that for some $T > 0$ and stopping time $\tau$,
\begin{enumerate}
        \item $\sy_0:\Omega \rightarrow \mathcal{H}_2$ is $\mathcal{F}_0-$measurable;
        \item $f:\Omega \times [0,T] \rightarrow \mathcal{H}_3$ is such that for $\mathbbm{P}-a.e.$ $\omega$, $f(\omega) \in L^2([0,T];\mathcal{H}_3)$;
        \item $B:\Omega \times [0,T] \rightarrow \mathscr{L}^2(\mathfrak{U};\mathcal{H}_2)$ is progressively measurable and such that for $\mathbbm{P}-a.e.$ $\omega$, $B(\omega) \in L^2\left([0,T];\mathscr{L}^2(\mathfrak{U};\mathcal{H}_2)\right)$;
        \item  \label{4*} $\sy:\Omega \times [0,T] \rightarrow \mathcal{H}_1$ is such that for $\mathbbm{P}-a.e.$ $\omega$, $\sy_{\cdot}(\omega)\mathbbm{1}_{\cdot \leq \tau(\omega)} \in L^2([0,T];\mathcal{H}_1)$ and $\sy_{\cdot}\mathbbm{1}_{\cdot \leq \tau}$ is progressively measurable in $\mathcal{H}_1$;
        \item \label{item 5 again*} The identity
        \begin{equation} \label{newest identity*}
            \sy_t = \sy_0 + \int_0^{t \wedge \tau}f_sds + \int_0^{t \wedge \tau}B_s d\mathcal{W}_s
        \end{equation}
        holds $\mathbbm{P}-a.s.$ in $\mathcal{H}_3$ for all $t \in [0,T]$.
    \end{enumerate}
The the equality 
  \begin{align} \label{ito big dog*}\norm{\sy_t}^2_{\mathcal{H}_2} = \norm{\sy_0}^2_{\mathcal{H}_2} + \int_0^{t\wedge \tau} \bigg( 2\inner{f_s}{\sy_s}_{\mathcal{H}_3 \times \mathcal{H}_1} + \norm{B_s}^2_{\mathscr{L}^2(\mathfrak{U};\mathcal{H}_2)}\bigg)ds + 2\int_0^{t \wedge \tau}\inner{B_s}{\sy_s}_{\mathcal{H}_2}d\mathcal{W}_s\end{align}
  holds for any $t \in [0,T]$, $\mathbbm{P}-a.s.$ in $\R$. Moreover for $\mathbbm{P}-a.e.$ $\omega$, $\sy_{\cdot}(\omega) \in C([0,T];\mathcal{H}_2)$. 
\end{proposition}

\begin{proof}
This is a minor extension of [\cite{prevot2007concise}] Theorem 4.2.5, which is stated and justified as Proposition 2.5.5. in [\cite{goodair2022stochastic}]. This is necessary for us as it extends the Gelfand Triple setting, removes the need for moment estimates and allows for local solutions. 
\end{proof}

\begin{lemma}[Stochastic Gr\"{o}nwall] \label{gronny}
Fix $t>0$ and suppose that $\boldsymbol{\phi},\boldsymbol{\psi}, \boldsymbol{\eta}$ are real-valued, non-negative stochastic processes. Assume, moreover, that there exists constants $c',\hat{c}, \tilde{c}$ (allowed to depend on $t$) such that for $\mathbbm{P}-a.e.$ $\omega$, \begin{equation} \label{boundingronny} \int_0^t\boldsymbol{\eta}_s(\omega) ds \leq c'\end{equation} and for all stopping times $0 \leq \theta_j < \theta_k \leq t$,
$$\mathbbm{E}\left(\sup_{r \in [\theta_j,\theta_k]}\boldsymbol{\phi}_r\right) + \mathbbm{E}\int_{\theta_j}^{\theta_k}\boldsymbol{\psi}_sds \leq \hat{c}\mathbbm{E}\left(\left[\boldsymbol{\phi}_{\theta_j} + \tilde{c} \right] + \int_{\theta_j}^{\theta_k} \boldsymbol{\eta}_s\boldsymbol{\phi}_sds\right) < \infty. $$Then there exists a constant $C$ dependent only on $c',\hat{c},\tilde{c},t$ such that $$\mathbbm{E}\sup_{r \in [0,t]}\boldsymbol{\phi}_r + \mathbbm{E}\int_{0}^{t}\boldsymbol{\psi}_sds \leq C\left[\mathbbm{E}(\boldsymbol{\phi}_{0}) + \tilde{c}\right].$$
\end{lemma}

\begin{proof}
See [\cite{glatt2009strong}] Lemma 5.3.
\end{proof}

\begin{lemma} \label{Lemma 5.2}
    Let $\mathcal{H}_1, \mathcal{H}_2$ be Hilbert Spaces such that $\mathcal{H}_1$ is compactly embedded into $\mathcal{H}_2$, and for some fixed $T>0$ let $(\sy^n): \Omega \times [0,T] \rightarrow \mathcal{H}_1$ be a sequence of measurable processes such that \begin{equation} \label{first condition} \sup_{n\in \N}\mathbbm{E}\int_0^T\norm{\sy^n_s}^2_{\mathcal{H}_1}ds < \infty\end{equation} and for any $\varepsilon > 0$, 
    \begin{equation}\label{second condition} \lim_{\delta \rightarrow 0^+}\sup_{n \in \N}\mathbbm{P}\left(\left\{\omega \in \Omega:\int_0^{T-\delta}\norm{\sy^n_{s + \delta}(\omega) - \sy^n_s(\omega)}^2_{\mathcal{H}_2}ds > \varepsilon\right\} \right) =0.\end{equation}
    Then the sequence of the laws of $(\sy^n)$ is tight in the space of probability measures over $L^2\left([0,T];\mathcal{H}_2\right)$.
\end{lemma}

\begin{proof}
    See [\cite{rockner2022well}] Lemma 5.2.
\end{proof}

\begin{lemma} \label{lemma for D tight}
    Let $\mathcal{Y}$ be a reflexive separable Banach Space and $\mathcal{H}$ a separable Hilbert Space such that $\mathcal{Y}$ is compactly embedded into $\mathcal{\mathcal{H}}$, and consider the induced Gelfand Triple
    $$\mathcal{Y} \xhookrightarrow{} \mathcal{H} \xhookrightarrow{} \mathcal{Y}^*. $$ For some fixed $T>0$ let $\sy^n: \Omega \rightarrow C\left([0,T];\mathcal{H}\right)$ be a sequence of measurable processes such that for every $t\in[0,T]$, \begin{equation} \label{first condition primed}
        \sup_{n \in \N}\mathbbm{E}\left(\sup_{t\in[0,T]}\norm{\sy^n_t}_{\mathcal{H}}\right) < \infty
    \end{equation}
    and for any sequence of stopping times $(\gamma_n)$ with $\gamma_n: \Omega \rightarrow [0,T]$, and any $\varepsilon > 0$, $y \in \mathcal{Y}$,
    \begin{equation} \label{second condition primed}
        \lim_{\delta \rightarrow 0^+}\sup_{n \in \N}\mathbbm{P}\left(\left\{
    \omega \in \Omega: \left\vert \left\langle \sy^n_{(\gamma_n + \delta) \wedge T} -\sy^n_{\gamma_n }   , y     \right\rangle_{\mathcal{H}} \right\vert > \varepsilon \right\}\right)  = 0.
    \end{equation}
    Then the sequence of the laws of $(\sy^n)$ is tight in the space of probability measures over $\mathcal{D}\left([0,T];\mathcal{Y}^*\right)$.
\end{lemma}

\begin{proof}
    See [\cite{goodair2022stochastic}] Lemma 3.3.2.
\end{proof}

\textbf{Thanks:} I would like to give my sincerest thanks to Dan Crisan for the regular and extended discussions around the paper, his feedback on it, and overall guidance during this process.\\

\textbf{Acknowledgements:} The author was supported by the Engineering and Physical Sciences Research Council (EPSCR) Project 2478902.

\addcontentsline{toc}{section}{References}

\bibliographystyle{spmpsci}
\bibliography{myBibliography2}

\end{document}